\documentclass[12pt,a4paper,oneside,final]{amsart}

\usepackage[T1]{fontenc}
\usepackage[utf8]{inputenc}
\usepackage[english]{babel}
\usepackage{mathtools, amsthm, amssymb, amscd} 
\usepackage{newtxtext, newtxmath} 
\usepackage[DIV=11]{typearea} 
\usepackage[hidelinks]{hyperref} 
\usepackage[notref,notcite]{showkeys} 
\usepackage{enumitem} 
\usepackage{xypic} 
\usepackage{tikz}
\usetikzlibrary{decorations.markings, arrows} 

\DeclareMathAlphabet{\mathcal}{OMS}{cmsy}{m}{n} 

\newcommand{\DD}{\mathsf{D}}
\newcommand{\GG}{\mathcal{G}}
\newcommand{\homeoc}{\operatorname{Homeo}_c}
\newcommand{\homeo}{\operatorname{Homeo}}
\newcommand{\orb}{\operatorname{Orb}}
\newcommand{\G}{\operatorname{Germ}}
\newcommand{\osh}[1][E]{\sigma_{#1}}
\newcommand{\OSS}[1][E]{\partial{#1}}
\newcommand{\catname}[1]{\mathbf{#1}}
\newcommand{\TopG}{\catname{SpatG}}
\newcommand{\Groupoid}{\catname{Gpoid}}
\newcommand{\Cstar}{C^*}

\DeclareMathOperator{\supp}{supp}
\DeclareMathOperator{\aut}{Aut}
\DeclareMathOperator{\ck}{CK}
\DeclareMathOperator{\co}{CO}
\DeclareMathOperator{\id}{id}
\DeclareMathOperator{\spn}{span}

\DeclareRobustCommand{\SkipTocEntry}[5]{} 

\newtheorem{lemma}{Lemma}[section]
\newtheorem{corollary}[lemma]{Corollary}
\newtheorem{theorem}[lemma]{Theorem}
\newtheorem{proposition}[lemma]{Proposition}
\newtheorem{introtheorem}{Theorem}

\theoremstyle{definition}
\newtheorem{definition}[lemma]{Definition}
\newtheorem{example}[lemma]{Example}
\newtheorem{remark}[lemma]{Remark}

\makeatletter
\def\l@section{\@tocline{1}{0pt}{1pc}{}{}}
\def\l@subsection{\@tocline{2}{0pt}{1pc}{4.6em}{}}
\def\l@subsubsection{\@tocline{3}{0pt}{1pc}{7.6em}{}}
\renewcommand{\tocsection}[3]{%
  \indentlabel{\@ifnotempty{#2}{\makebox[2.3em][l]{%
    \ignorespaces#1 #2.\hfill}}}#3}
\renewcommand{\tocsubsection}[3]{%
  \indentlabel{\@ifnotempty{#2}{\hspace*{2.3em}\makebox[2.3em][l]{%
    \ignorespaces#1 #2.\hfill}}}#3}
\renewcommand{\tocsubsubsection}[3]{%
  \indentlabel{\@ifnotempty{#2}{\hspace*{4.6em}\makebox[3em][l]{%
    \ignorespaces#1 #2.\hfill}}}#3}
\makeatother

\makeatletter
\newcommand\@dotsep{4.5}
\def\@tocline#1#2#3#4#5#6#7{\relax
  \ifnum #1>\c@tocdepth 
  \else
    \par \addpenalty\@secpenalty\addvspace{#2}%
    \begingroup \hyphenpenalty\@M
    \@ifempty{#4}{%
      \@tempdima\csname r@tocindent\number#1\endcsname\relax
    }{%
      \@tempdima#4\relax
    }%
    \parindent\z@ \leftskip#3\relax
    \advance\leftskip\@tempdima\relax
    \rightskip\@pnumwidth plus1em \parfillskip-\@pnumwidth
    #5\leavevmode\hskip-\@tempdima #6\relax
    \leaders\hbox{$\m@th
      \mkern \@dotsep mu\hbox{.}\mkern \@dotsep mu$}\hfill
    \hbox to\@pnumwidth{\@tocpagenum{#7}}\par
    \nobreak
    \endgroup
  \fi}
 \def\l@subsection{\@tocline{2}{0pt}{30pt}{5pc}{}}
\makeatother

\title[Topological Full Groups of Ample Groupoids]{Topological Full Groups of Ample Groupoids with Applications to Graph Algebras}
 
\date{\today}

\author[1]{Petter Nyland}
\author[2]{Eduard Ortega}

\address{Department of Mathematical Sciences, Faculty of Information Technology and Electrical Engineering, NTNU -- Norwegian University of Science and Technology, Trondheim, Norway}
\email{petter.nyland@ntnu.no}
\address{Department of Mathematical Sciences, Faculty of Information Technology and Electrical Engineering, NTNU -- Norwegian University of Science and Technology, Trondheim, Norway}
\email{eduard.ortega@ntnu.no}

\hypersetup{final} 

\begin{document}

\begin{abstract}
We study the topological full group of ample groupoids over locally compact spaces. We extend Matui's definition of the topological full group from the compact, to the locally compact case. We provide two general classes of étale groupoids for which the topological full group, as an abstract group, is a complete isomorphism invariant. Hereby extending Matui's Isomorphism Theorem. As an application, we study graph groupoids and their topological full groups, and obtain sharper results for this class. The machinery developed in this process is used to prove an embedding theorem for ample groupoids, akin to Kirchberg's Embedding Theorem for \mbox{$C^*$-algebras}. Consequences for graph \mbox{$C^*$-algebras} and Leavitt path algebras are also spelled out. In particular, we improve on a recent embedding theorem of Brownlowe and S\o rensen for Leavitt path algebras.
\end{abstract}

\maketitle

\tableofcontents

\section{Introduction}

\addtocontents{toc}{\SkipTocEntry}
\subsection*{Background}
The study of (topological) full groups in the setting of topological dynamics was initiated by Giordano, Putnam and Skau~\cite{GPS}. This was inspired by the work of Dye~\cite{Dye} in the measurable setting, and by Krieger's study of so-called ample groups on the Cantor space~\cite{Kri}. For Cantor minimal systems, Giordano, Putnam and Skau showed that certain distinguished subgroups of the full group determine completely the orbit equivalence class, the strong orbit equivalence class, and the flip conjugacy class, respectively, of the system. The \emph{full group} of a Cantor system (i.e.\ a $\mathbb{Z}$-action on a Cantor space) consists of all homeomorphisms of the Cantor space which leave the orbits invariant. Roughly speaking, the \emph{topological full group} is the subgroup of the full group consisting of those homeomorphisms which additionally preserve the orbits in a continuous manner. Giordano, Putnam and Skau also connected the dynamics with the theory of \mbox{$C^*$-algebras}, via the crossed product construction and its $K$-theory~\cite{GPS3}. Thus, they exhibited a strong relationship between these, a priori, quite different mathematical structures.

This is but one example of the rich interplay between dynamical systems and \mbox{$C^*$-algebras}. (This interplay essentially goes all the way back to the inception of the field by Murray and von~Neumann~\cite{MvN}.) Another prominent example of this interplay is the connection between shifts of finite type and Cuntz-Krieger algebras; discovered by Cuntz and Krieger in the early eighties~\cite{CK}. In the setting of irreducible one-sided shifts of finite type, Matsumoto defined the topological full group of such a dynamical system and proved that this group determines the shift up to continous orbit equivalence, and also the associated Cuntz-Krieger algebra up to diagonal preserving isomorphism~\cite{Mats},~\cite{Mats2}. This parallelled Giordano, Putnam and Skau's results, although the dynamical systems were quite different. For instance, the former has no periodic points whereas the latter has a dense set of periodic points.

Using topological groupoids to model dynamical systems has unified many of these seemingly different connections between dynamics and \mbox{$C^*$-algebras}. Whenever one has a dynamical system of some sort, one may typically associate to it a topological groupoid, and from the groupoid one can construct its groupoid \mbox{$C^*$-algebra}. In many cases, isomorphism of such groupoids correspond to some suitable notion of continuous orbit equivalence of the dynamical systems, and also to diagonal preserving isomorphism of the groupoid \mbox{$C^*$-algebras}~\cite{MM}, \cite{BCW}, \cite{Li}, \cite{Li2}. That groupoid isomorphism corresponds to diagonal preserving isomorphism of the \mbox{$C^*$-algebras} (in the topologically principal case) is due to the pioneering work of Renault~\cite{Ren2}. This reconstruction result has recently been generalized in e.g.~\cite{CRST}; wherein it is also shown that by adding more structure on the groupoids, such as gradings, one can recover stronger types of equivalence of the dynamical systems.  

In~\cite{Mat1}, Matui defined the topological full group of an étale groupoid with compact unit space. His definition generalized virtually all the previously given definitions for different kinds of dynamical systems at one fell swoop. Matui realized that homeomorphisms which preserve orbits in a continuous manner are always given by \emph{full bisections} from the associated groupoid. In the subsequent paper~\cite{Mat} Matui proved (among other things) a remarkable isomorphism theorem. Supressing some assumptions, this theorem says that any two minimal étale groupoids over a Cantor space are isomorphic, as topological groupoids, if and only if their topological full groups\footnote{Actually, the same is true for several distinguished subgroups of the topological full group as well, such as its commutator subgroup. See~\cite{Mat} and \cite{Nek} for details.} are isomorphic, as abstract groups. Matui's Isomorphism Theorem generalized the results of Giordano, Putnam and Skau, and Matsumoto, and others. 

The study of topological full groups has also found interesting applications to group theory. Matui's isomorphism theorem means that one can classify the groupoids (and therefore any underlying dynamics, and the \mbox{$C^*$-algebras}) in terms of the topological full group. However, by going the other direction, one can use étale groupoids to distinguish certain discrete groups. Given two discrete groups, say in terms of their generators and relations, it can be hard to tell whether they are isomorphic or not. But if one can realize these groups as topological full groups (or distinguished subgroups) of some groupoids, then one can use the groupoids (i.e.\ the dynamics) to tell the groups apart---as one often has much dynamical information about the groupoids. For instance, this was the strategy used by Brin to show that Thompson's group $V$ is not isomorphic to its two-dimensional analog $2V$~\cite{Brin} (although he did not consider the groupoid explicitly). A more recent application of this form is by Matte Bon~\cite{MB} who showed that the higher dimensional Thompson group\footnote{It is known that the groups $nV$ are all non-isomorphic~\cite{BL}.} $nV$ embeds into $mV$ if and only if $n \leq m$. Matte Bon's paper also includes a novel approach to Matui's Isomorphism Theorem in terms of a certain dichotomy for such groupoids. Another application is that topological full groups have provided new examples of groups with exotic properties. Most notably, topological full groups (or more precisely, their commutator subgroups) of Cantor minimal systems provided the first examples of finitely generated simple groups that are amenable (and infinite)~\cite{JM}. On another note, topological full groups arising from non-amenable groups acting minimally and topologically free on the Cantor space were recently shown to be $C^*$-simple~\cite{BrixS}.

Topological full groups have also found their way into Lawson's program of non-commutative Stone duality~\cite{Law}. In~\cite{Law2}, the topological full group of an étale groupoid is shown to coincide with the group of units of the so-called Tarski monoid to which the groupoid corresponds under non-commutative Stone duality.

\addtocontents{toc}{\SkipTocEntry}
\subsection*{Our results}
The main motivation for the present paper was Matsumoto and Matui's work on irreducible one-sided shifts of finite type mentioned above. If we rephrase their work in terms of (directed) graphs, then they showed that for two strongly connected finite graphs~$E$ and~$F$ the following are equivalent:
\begin{enumerate}
\item The shifts $(E^\infty, \sigma_E)$ and $(F^\infty, \sigma_F)$ are continuously orbit equivalent.
\item The graph groupoids $\mathcal{G}_E$ and $\mathcal{G}_F$ are isomorphic as topological groupoids.
\item There is an isomorphism of the graph $C^*$-algebras $C^*(E)$ and $C^*(F)$ which maps the diagonal $\mathcal{D}(E)$  onto $\mathcal{D}(F)$.   
\item The topological full groups $\llbracket \mathcal{G}_{E} \rrbracket$ and $\llbracket \mathcal{G}_{F} \rrbracket$ are isomorphic as abstract groups.
\end{enumerate}
The equivalence of (1), (2) and (3) above have since been generalized to more general graphs which need neither be finite nor strongly connected~\cite{CEOR}, \cite{BCW}. Our initial goal was to study the topological full group $\llbracket \mathcal{G}_{E} \rrbracket$ of general graph groupoids $\mathcal{G}_E$ and see if we could also add statement (4) to said equivalence.

Matui's Isomorphism Theorem~\cite[Theorem~3.10]{Mat} gives the equivalence of (2) and~(4) above for the general class of ample effective Hausdorff minimal second countable groupoids over (compact) Cantor spaces (see Subsection~\ref{subs:groupoids} for definitions). This covers in particular graph groupoids of strongly connected finite graphs. In light of this we attempted to extend Matui's Isomorphism Theorem a little further in order to cover graph groupoids of more general graphs. To do this it is necessary to relax both the compactness assumption of the unit space (which corresponds to the graph having finitely many vertices) and the minimality assumption (which corresponds to strong connectedness of the graph).

As our main findings we first describe two modest extensions of Matui's Isomorphism Theorem that apply to general ample groupoids. Then we describe two (sharper) isomorphism theorems for the class of graph groupoids. Finally, we present a novel embedding theorem for ample groupoids. First of all we have to extend the definition of the topological full group to the locally compact setting. This is done in Definition~\ref{def:tfg}, where we stipulate that the homeomorphisms in the topological full group should be compactly supported (in addition to being induced by bisections). This seems a natural choice, as we then retain the ``finitary'' nature of the elements in the topological full group, as well as the countability of the topological full group (for second countable groupoids). Additionally, most of the arguments from~\cite{Mat} still work with suitable modifications. For an ample groupoid $\mathcal{G}$ we denote its unit space by $\mathcal{G}^{(0)}$. The topological full group of $\mathcal{G}$ is denoted by $\llbracket \mathcal{G} \rrbracket$. And the commutator subgroup of $\llbracket \mathcal{G} \rrbracket$ is denoted by $\DD(\llbracket \mathcal{G} \rrbracket)$. The first of these isomorphism theorems is a straightforward extension of Matui's Isomorphism Theorem which relaxes the compactness assumption on~$\mathcal{G}^{(0)}$ and the second countability assumption on $\mathcal{G}$.

\begin{introtheorem}[c.f.\ Theorem~\ref{thm:KFgroupoid}, {\cite[Theorem~3.10]{Mat}}]\label{intro:KFgroupoid}
Suppose $\mathcal{G}_1$ and $\mathcal{G}_2$ are effective ample minimal Hausdorff groupoids whose unit spaces have no isolated points. Then following are equivalent:
\begin{enumerate}
\item $\mathcal{G}_1 \cong \mathcal{G}_2$ as topological groupoids.
\item $\llbracket \mathcal{G}_1 \rrbracket \cong \llbracket \mathcal{G}_2 \rrbracket$ as abstract groups.
\item $\DD(\llbracket \mathcal{G}_1 \rrbracket) \cong \DD(\llbracket \mathcal{G}_2 \rrbracket)$ as abstract groups.
\end{enumerate}
\end{introtheorem}

We mention that when restricting to the class of graph groupoids we are also able to relax the minimality assumption in Theorem~\ref{intro:KFgroupoid} substantially (see Theorem~\ref{intro:KFrigid} below). The second isomorphism theorem replaces the minimality assumption with a significantly weaker ``mixing property'' that we call \emph{non-wandering} (see Definition~\ref{def:wandering}). However, the result does not apply to the commutator subgroups. And we also require the unit spaces to be second countable. (By a \emph{locally compact Cantor space} we mean either the \emph{compact} Cantor space or the locally compact \emph{non-compact} Cantor space (up to homeomorphism) c.f.~Subsection~\ref{subs:top}.)

\begin{introtheorem}[c.f.\ Theorem~\ref{thm:KLCCgroupoid}]\label{intro:KLCCgroupoid}
Let $\mathcal{G}_1$ and $\mathcal{G}_2$ be effective ample Hausdorff groupoids over locally compact Cantor spaces. If, for $i = 1,2$, $\mathcal{G}_i$ is non-wandering and each $\mathcal{G}_i$-orbit has length at least $3$, then the following are equivalent:
\begin{enumerate}
\item $\mathcal{G}_1 \cong \mathcal{G}_2$ as topological groupoids.
\item $\llbracket \mathcal{G}_1 \rrbracket \cong \llbracket \mathcal{G}_2 \rrbracket$ as abstract groups.
\end{enumerate}
\end{introtheorem}

Let us say a few words about the proofs.  As the implications $(1) \Rightarrow (2) \Rightarrow (3)$ in Theorem~\ref{intro:KFgroupoid} and $(1) \Rightarrow (2)$ in Theorem~\ref{intro:KLCCgroupoid} are trivial, there is only one direction to prove. The proof strategy is similar in both cases and is summarized in the following diagram\footnote{If $\Gamma \leq \homeo(X)$ and $\Lambda \leq \homeo(Y)$ are groups of homeomorphisms, then a \emph{spatial isomorphism} between them is a homeomorphism $\phi \colon X \to Y$ such that $\gamma \mapsto \phi \circ \gamma \circ \phi^{-1}$ for $\gamma \in \Gamma$ is a group isomorphism.}, where $\Gamma_i$ is a subgroup of $\homeo\left(\mathcal{G}_i^{(0)}\right)$:
\[\xymatrix@C=-3em{
{\Gamma_1 \cong \Gamma_2} \ar@{=>}[d]  & {\text{(abstract isomorphism)}}  \\
{\left(\Gamma_1, \mathcal{G}_1^{(0)}\right) \cong \left(\Gamma_2, \mathcal{G}_2^{(0)}\right)} \ar@{=>}[d]^{\text{ Functoriality}}  & {\text{(spatial isomorphism)}} \\
{\G\left(\Gamma_1, \mathcal{G}_1^{(0)}\right) \cong \G\left(\Gamma_2, \mathcal{G}_2^{(0)}\right)} \ar@{=>}[d]^{\ \Gamma_i \text{ covers } \mathcal{G}_i}  &     \\
{ \G\left(\Gamma_1, \mathcal{G}_1^{(0)}\right) \cong \ \mathcal{G}_1 \cong \mathcal{G}_2 \ \cong\G\left(\Gamma_2, \mathcal{G}_2^{(0)}\right)} &       } \]
The first step is showing that for certain classes of homeomorphism groups, any (abstract) group isomorphism is induced by a homeomorphism of the underlying spaces. We call this a \emph{spatial realization result}. In~\cite{Mat}, Matui proves a spatial realization result that applies to any $\Gamma$ with $\DD(\llbracket \mathcal{G} \rrbracket) \leq \Gamma \leq \llbracket \mathcal{G}\rrbracket$ (for minimal $\mathcal{G}$). And from a spatial isomorphism he directly constructs an isomorphism of the groupoids and obtains his Isomorphism Theorem. In this paper we have chosen to break this direct step into two more parts in order to also study when the groupoid can be recovered from the action of (subgroups of) the topological full group on the unit space, as the groupoid of germs of this action. We find that such a groupoid of germs always embed into the groupoid we started with, and that they are isomorphic if and only if the subgroup in question is generated by enough bisections to cover the groupoid (Proposition~\ref{prop_germs}, Corollary~\ref{cor:necessaryCover}). We also show that for a natural choice of maps, the assignment of the groupoid of germs is functorial (Proposition~\ref{prop:functorial}).

Having this machinery in place, proving Theorem~\ref{intro:KFgroupoid} is then just a matter of checking that Matui's spatial realization result also holds in the locally compact setting (Theorem~\ref{classF}). Although this is but a small extension of Matui's result we have chosen to include it as a theorem since it is applicable to a larger class of groupoids. Regarding our initial motivation, namely the graph groupoids, we are able to characterize exactly when the aforementioned spatial realization result applies, and it turns out that we can get away with much weaker mixing properties than minimality when we restrict to graph groupoids---see Theorem~\ref{intro:KFrigid} below. 

For the proof of Theorem~\ref{intro:KLCCgroupoid} we employ a spatial realization result (Theorem~\ref{thm:KBfaithful}) based on Rubin's work in~\cite{Rub} in the first step. We mention that Medynets has previously obtained a similar spatial realization result~\cite[Remark~3]{Med} for (topological) full groups arising from group actions on the Cantor space, building on Fremlins work in~\cite[Section~384]{Frem}. After some modifications, Theorem~\ref{intro:KLCCgroupoid} could also be deduced from this result. However, Theorem~\ref{thm:KBfaithful} is more general as it can potentially be applied to other groups than topological full groups, e.g.\ homeomorphism groups of 0-dimensional linearly ordered spaces. See Remark~\ref{rem:Medynets} for a more detailed discussion on the differences and similarities of these approaches. Although Theorems~\ref{intro:KFgroupoid} and~\ref{intro:KLCCgroupoid} can be deduced by employing arguments along the lines of~\cite{Mat} and~\cite{Med}, we believe that the way we trisect the proofs does add some new insight. In particular, this was how we discovered the embedding result given below in Theorem~\ref{intro:eqEmbed}. 

Let us now describe the isomorphism theorem we obtain for graph groupoids, when starting with the spatial reconstruction result à la Matui. As mentioned above, it turns out that we can replace minimality (strong connectedness of the graphs) with some weaker ``exit and return''-conditions. Each of these three conditions (see Definition~\ref{def:adhoc}) can be considered strengthenings of the three conditions that characterize when the boundary path space $\partial E$ has no isolated points (Proposition~\ref{prop:DEperfect}). Condition~(K) means that every cycle can be exited, and then returned to. Condition~(W) means that every wandering path can be exited, and then returned to. And Condition~($\infty$) means that every singular vertex can be exited (i.e.\ is an infinite emitter), and then returned to (along infinitely many of the emitted edges).

\begin{introtheorem}[c.f.\ Theorem~\ref{thm:KFrigid}]\label{intro:KFrigid}
Let $E$ and $F$ be graphs with no sinks, and suppose they both satisfy Condition~(K), (W) and ($\infty$). Then the following are equivalent:
\begin{enumerate}
\item $\mathcal{G}_{E} \cong \mathcal{G}_{F}$ as topological groupoids.
\item $\llbracket \mathcal{G}_{E} \rrbracket \cong \llbracket \mathcal{G}_{F} \rrbracket$ as abstract groups.
\item $\DD(\llbracket \mathcal{G}_{E} \rrbracket) \cong \DD(\llbracket \mathcal{G}_{F} \rrbracket)$ as abstract groups.
\end{enumerate}
\end{introtheorem}

By interpreting the assumptions in Theorem~\ref{intro:KLCCgroupoid} for graph groupoids we obtain Theorem~\ref{intro:KLCCrigid} below. Therein, Condition~(L) is the well-known exit condition of Kumjian, Pask and Raeburn~\cite{KPR}, namely, that every cycle should have an exit. Condition~(T) (see Definition~\ref{condT}) essentially means that the graph does not have a component which is a tree. Finally, what we call \emph{degenerate vertices} (see Definition~\ref{def:degenerate}) are the ones giving $\mathcal{G}_E$-orbits of length $1$ or $2$. This theorem may be considered a  generalization of Matsumoto's result in the case of irreducible one-sided shifts of finite type~\cite{Mats2} (which correspond to finite strongly connected graphs).

\begin{introtheorem}[c.f.\ Theorem~\ref{thm:KLCCrigid}]\label{intro:KLCCrigid}
Let $E$ and $F$ be countable graphs satisfying Condition~(L) and (T), and having no degenerate vertices. Then the following are equivalent:
\begin{enumerate}
\item $\mathcal{G}_{E} \cong \mathcal{G}_{F}$ as topological groupoids.
\item $\llbracket \mathcal{G}_{E} \rrbracket \cong \llbracket \mathcal{G}_{F} \rrbracket$ as abstract groups.
\end{enumerate}
\end{introtheorem}

Hence we establish the equivalence of (1)--(4) mentioned in the beginning of this subsection for graphs satisfying the assumptions of Theorem~\ref{intro:KLCCrigid}. In Corollary~\ref{cor:BCW}, we spell out this rigidity result for the associated graph algebras.

Our final main result is an embedding theorem for ample groupoids---inspired by embedding theorems for \mbox{$C^*$-algebras} and Leavitt path algebras. The seminal embedding theorem of Kirchberg~\cite{KP} states that any separable exact (unital) \mbox{$C^*$-algebra} embeds (unitally) into the Cuntz algebra $\mathcal{O}_2$. In particular, this means that any graph \mbox{$C^*$-algebra} $C^*(E)$, where~$E$ is a countable graph, embeds into $\mathcal{O}_2$. The latter, being the universal \mbox{$C^*$-algebra} generated by two orthogonal isometries, can be canonically identified with a graph \mbox{$C^*$-algebra}. Namely, the graph \mbox{$C^*$-algebra} of the graph $E_2$ which consists of a single vertex with two loops. In~\cite{BS}, Brownlowe and Sørensen show that the Leavitt path algebra $L_R(E)$, where $E$ is any countable graph and $R$ any commutative unital ring, embeds into $L_R(E_2)$---the algebraic analog of $\mathcal{O}_2$. An inspection of their proof reveals that this embedding also maps the canonical diagonal subalgebra $D_R(E)$ into $D_R(E_2)$. As a consequence, Kirchberg's embedding for the graph \mbox{$C^*$-algebras} may then also be taken to be diagonal preserving---with respect to the diagonal\footnote{Technically, this is a Cartan subalgebra in the sense of Renault, not a \mbox{$C^*$-diagonal} in the sense of Kumjian. But it's common to refer to it as ``the diagonal'' in a graph \mbox{$C^*$-algebra}.} in $\mathcal{O}_2$ coming from its identification with $C^*(E_2)$. At this point, it starts smelling a bit like groupoids might be lurking about. Indeed, using the properties of the \emph{Germ-functor} (see Section~\ref{sec:SpatG}), we are able to prove that the underlying graph groupoid~$\mathcal{G}_E$ embeds into the Cuntz groupoid $\mathcal{G}_{E_2}$ (modulo topological obstructions in the sense of isolated points). Thus, the known embeddings of the graph algebras actually occur at the level of the underlying groupoid models. We were also able to extend this embedding result to all groupoids which are groupoid equivalent (or stably isomorphic) to a graph groupoid. To the best of the authors' knowledge, this is the first embedding result of its kind for ample groupoids.

\begin{introtheorem}[c.f.\ Theorem~\ref{eqEmbed}]\label{intro:eqEmbed}
Let $\mathcal{H}$ be an effective ample second countable Hausdorff groupoid with $\mathcal{H}^{(0)}$ a locally compact Cantor space. If $\mathcal{H}$ is groupoid equivalent to $\mathcal{G}_E$, for some countable graph $E$ satisfying Condition~(L) and having no sinks nor semi-tails, then~$\mathcal{H}$ embeds into $\mathcal{G}_{E_2}$. Moreover, if $\mathcal{H}^{(0)}$ is compact, then the embedding maps $\mathcal{H}^{(0)}$ onto $E_2^\infty$.

In particular, any graph groupoid $\mathcal{G}_E$, with $E$ as above, embeds into $\mathcal{G}_{E_2}$, and any \mbox{AF-groupoid} (with perfect unit space) embeds into $\mathcal{G}_{E_2}$.
\end{introtheorem}

The main ingredient in the proof is constructing an injective local homeomorphism $\phi \colon \partial E \to E_2^\infty$ which induces a spatial embedding of the associated topological full groups. This construction is entirely explicit. As a consequence we also obtain explicit embeddings of any graph \mbox{$C^*$-algebra} $C^*(E)$ (or Leavitt path algebra $L_R(E)$), in terms of their canonical generators, into $\mathcal{O}_2$ (or $L_R(E_2)$. This embedding is \emph{diagonal preserving}, and when $C^*(E)$ is unital (i.e.\ $E^0$ is finite) this embedding is unital and maps the diagonal \emph{onto} the diagonal. These embeddings are described in Corollary~\ref{cor:graphAlgEmb} and Remark~\ref{rem:emb}. We also record a result on diagonal embeddings of \mbox{AF-algebras} in Corollary~\ref{cor:AFemb}.

Another consequence of Theorem~\ref{intro:eqEmbed} is that each topological full group $\llbracket \mathcal{G}_E \rrbracket$, for $E$ as above, embeds into Thompson's group $V$---since $V$ is isomorphic to $\llbracket \mathcal{G}_{E_2} \rrbracket$. The Higman-Thompson groups $V_{n,r}$ (where $nV = V_{n,1}$) can be realized as topological full groups of graph groupoids of certain strongly connected finite graphs (see Subsection~11.3). Hence, our embedding theorem may be considered a generalization of the well-known embedding of~$V_{n,r}$ into $V$. The embedding entails that the topological full groups $\llbracket \mathcal{H} \rrbracket$, of groupoids $\mathcal{H}$ as in Theorem~\ref{intro:eqEmbed}, has the Haagerup property (but they are generally not amenable). In terms of groups, our embedding also includes all the so-called \emph{LDA-groups} (see Remark~\ref{rem:LDA}).

In~\cite{Mat3}, Matui introduced two conjectures for minimal ample groupoids over the Cantor space. The \emph{HK-conjecture} relates the groupoid homology to the $K$-theory of the groupoid \mbox{$C^*$-algebra}. And the \emph{AH-conjecture} relates the topological full group to the groupoid homology. These conjectures have been verified in several cases~\cite{Mat2}, in particular for (products of) graph groupoids arising from strongly connected finite graphs. For the more general graph groupoids studied in the present paper, the second named author will, together with Toke Meier Carlsen, attack these conjectures in a forthcoming paper. (In the recent preprint~\cite{Ort}, the second named author verifies the HK-conjecture for a class of groupoids which includes the graph groupoids of row-finite graphs.)

\addtocontents{toc}{\SkipTocEntry}
\subsection*{Précis}
The structure of the paper is as follows. We recall some basic notions regarding étale groupoids and (classical) Stone duality in Section~\ref{sec:prelim}. This section also serves the purpose of establishing notation and conventions. The rest of the paper is divided into two parts. The first, sections~\ref{sec:tfg}--\ref{sec:reggrpd}, deals with ample groupoids in general, while the second, sections~\ref{sec:gg}--\ref{sec:embed}, deals with graph groupoids.

In Section~\ref{sec:tfg} we give the definition of the topological full group $\llbracket \mathcal{G} \rrbracket$ of an ample groupoid~$\mathcal{G}$ with locally compact unit space $\mathcal{G}^{(0)}$. We also prove some elementary results on the existence of elements in the topological full group with certain properties. Then we move on to study the groupoid of germs $\smash{\G\left(\Gamma, \mathcal{G}^{(0)} \right)}$ associated to a subgroup~$\Gamma \leq \llbracket \mathcal{G} \rrbracket$ of the topological full group, in Section~\ref{sec:germ}. We establish that $\G\left(\Gamma, \mathcal{G}^{(0)} \right)$ always embeds into~$\mathcal{G}$, and that this embedding is an isomorphism as long as $\Gamma$ contains ``enough elements''. In Section~\ref{sec:SpatG} we introduce the two categories; $\TopG$ and $\Groupoid$. The former consists of pairs~$(\Gamma, X)$ where $X$ is a space and $\Gamma$ is a subgroup of $\homeo(X)$. The latter consists of certain ample groupoids. By defining suitable morphisms in these categories and what the germ of a morphism in $\TopG$ should be, we establish that the assigment~$(\Gamma, X) \mapsto \G\left(\Gamma, X \right)$ is functorial.  We also show that monomorphisms in $\TopG$ induce étale embeddings of the associated groupoids of germs.

The spatial realization results needed to deduce that an abstract isomorphism of two topological full groups always is spatially implement are provided in Section~\ref{sec:spatrel}. In Section~\ref{sec:reggrpd} we prove the two general isomorphism theorems, Theorem~\ref{intro:KFgroupoid} and Theorem~\ref{intro:KLCCgroupoid}. This is now mostly a matter of interpreting the spatial realization results from Section~\ref{sec:spatrel}  in terms of the groupoid and its topological full group, and then combine this with the results of Section~\ref{sec:germ} and Section~\ref{sec:SpatG}.

In Section~\ref{sec:gg} we begin our in-depth study of graph groupoids $\mathcal{G}_E$ of general graphs $E$. This section is devoted to a thorough introduction of graph terminology and the dynamics that give rise to the graph groupoids. For several of the generic properties a topological groupoid can have, we list their characterizations for graph groupoids in terms of the graphs. We continue in Section~\ref{sec:tfgg} with describing explicitly all elements in the topological full group~$\llbracket \mathcal{G}_E \rrbracket$ of any graph groupoid. To do this we need to specify a new (yet equivalent) basis for the topology on $\mathcal{G}_E$. We then pursue specialized isomorphism theorems for the class of graph groupoids in Section~\ref{sec:isogg}. This yields Theorem~\ref{intro:KFrigid} and Theorem~\ref{intro:KLCCrigid}. At the end of this section we spell out the induced rigidity result for the associated graph algebras.

In the final section of the paper we employ the machinery from Sections~\ref{sec:germ}, \ref{sec:SpatG} and \ref{sec:tfgg} to obtain our groupoid embedding result; Theorem~\ref{intro:eqEmbed}. We also describe the explicit diagonal embeddings of the graph algebras that follow from the embedding of the groupoids. Examples of these embeddings for graph algebras are provided for several infinite graphs. At the end of Section~\ref{sec:embed} we show that any \mbox{AF-groupoid} is groupoid equivalent to a graph groupoid, going via Bratteli diagrams, hence $\mathcal{G}_{E_2}$-embeddable. We then spell out consequences for diagonal embeddings of \mbox{AF-algebras}. Additionally, we remark that transformation groupoids arising from locally compact (non-compact) Cantor minimal systems are \mbox{AF-groupoids}, and hence $\mathcal{G}_{E_2}$-embeddable as well.

\addtocontents{toc}{\SkipTocEntry}
\subsection*{Acknowledgments}
We would like to express our gratitude to Volodymyr Nekrashevych for sharing his private notes on Rubin's theorems. We would also like to thank Hiroki Matui for pointing out Lemma~\ref{lem_haus} to us, as well as for other valuable comments while he visited NTNU in the spring of 2018. The first named author thanks Eric Wofsey for helpful remarks on Stone duality. We also wish to thank Ulrik Enstad and Christian Skau for comments on the first draft of this paper. Further, we are grateful to Fredrik Hildrum for assistance with typesetting. Finally, we want to thank the anonymous referee for his or her thorough feedback.

\section{Preliminaries}\label{sec:prelim}

We will now recall the basic notions needed throughout the paper, as well as establish notation and conventions. We denote the positive integers by $\mathbb{N}$ and the non-negative integers by $\mathbb{N}_0$. If two sets $A$ and $B$ are disjoint we will denote their union by $A \sqcup B$ if we wish to emphasize that they are disjoint. When we write $C = A \sqcup B$ we mean that $C = A \cup B$ \emph{and} that $A$ and $B$ are disjoint sets.

\subsection{Topological notions}\label{subs:top}
Following~\cite{KL}, \cite{Stein} we say that a topological space is \emph{Boolean} if it is Hausdorff and has a basis of compact open sets. (This is also the terminology orginally used by Stone~\cite{Stone}.) A \emph{Stone space} is then a compact Boolean space. We say that a topological space is \emph{perfect} if it has no isolated points. By a \emph{locally compact Cantor space} we mean a (non-empty) second countable perfect Boolean space. Up to homeomorphism there are two such spaces; one compact (the Cantor set) and one non-compact (the Cantor set with a point removed). The latter may also be realized as any non-closed open subset of the Cantor set, or as the product of the Cantor set and a countably infinite discrete space.

For a topological space $X$ we denote the group of self-homeomorphisms of~$X$ by $\homeo(X)$. We will occasionally denote $\id_X$ simply by $1$ for brevity. By an \emph{involution} we mean a homeomorphism (or more generally, a group element) $\phi$ with $\phi^2 = 1$.  For a homeomorphism~$\phi \in \homeo(X)$, we define the \emph{support of $\phi$} to be the (regular) closed set~$\overline{\{x\in X \ \vert \ \phi(x)\neq x\}}$, and denote it by $\supp(\phi)$. We also define \[\homeoc(X) \coloneqq \{ \phi \in \homeo(X) \ \vert \ \supp(\phi) \text{ compact open} \}.\] 
When $\Gamma$ is a subgroup of a group $\Gamma'$ we write $\Gamma \leq \Gamma'$. Beware that we will abuse this notation when we write $\Gamma \leq \homeoc(X)$ to mean that $\Gamma$ is a subgroup of $\homeo(X)$ \emph{and} that $\Gamma \subseteq \homeoc(X)$. (It is not clear whether $\homeoc(X)$ itself is a group.)

\subsection{Stone duality}
We will now briefly recall the basics of (classical) Stone duality needed for Section~\ref{sec:spatrel}. For more details the reader may consult~\cite{Kop}, \cite[Chapter~31]{Frem} (or even the fountainhead~\cite{Stone}, \cite{Doct}). By a \emph{Boolean algebra} we mean a complemented distributive lattice with a top and bottom element. And by a \emph{generalized Boolean algebra} we mean a relatively complemented distributive lattice with a bottom element. For a topological space $X$, we denote the set of clopen subsets of $X$ by $\co(X)$. The set of compact open subsets of $X$ are denoted by $\ck(X)$. Finally, the set of regular open subsets of $X$ are denoted by $\mathcal{R}(X)$.

\begin{example}
Let $X$ be a topological space.
\begin{enumerate}
\item $\co(X)$ is a Boolean algebra under the operations of set-theorietic union, intersection and complement by $X$.
\item $\ck(X)$ is a generalized Boolean algebra in the same way as $\co(X)$, except for admitting only relative (set-theoretic) complements.
\item $\mathcal{R}(X)$ is a Boolean algebra with the following operations. Let $A,B \in \mathcal{R}(X)$. The join of $A$ and $B$ is $\left(\overline{A \cup B}\right)^\circ$, where $\circ$ denotes the interior. The meet of $A$ and $B$ is $A \cap B$. And the complement of $A$ is $\sim A \coloneqq (X \setminus A)^\circ$.
\end{enumerate}
\end{example}

A crude way of stating Stone duality is to say that every Boolean algebra arises as $\co(X)$ for some Stone space $X$, and that every generalized Boolean algebra arises as $\ck(Y)$ for some Boolean space $Y$. Hence, Stone spaces correspond to Boolean algebas and Boolean spaces correspond to generalized Boolean algebras.

More precisely, it is a duality in the following sense. A continuous map $f \colon X \to Y$ between topological spaces $X$ and $Y$ is \emph{proper} if $f^{-1}(K)$ is compact in $X$ whenever $K$ is a compact subset of $Y$. A map $\psi \colon \mathcal{A} \to \mathcal{B}$ between generalized Boolean algebras $\mathcal{A}$ and $\mathcal{B}$ is a \emph{Boolean homomorphism} if it preserves joins, meets and relative complements. We say that $\psi$ is \emph{proper} if for each $b \in \mathcal{B}$, there exists $a \in \mathcal{A}$ such that $\psi(a) \geq b$. Boolean spaces with proper continuous maps form a category. So does generalized Boolean algebras with proper Boolean homomorphisms. For a proper continuous map $f \colon X \to Y$, let $\ck(f)(A) \coloneqq f^{-1}(A)$ for $A \in \ck(Y)$. This makes $\ck(-)$ a contravariant functor from the category of Boolean spaces to the category of generalized Boolean algebras (with maps as above).

For a generalized Boolean algebra $\mathcal{A}$, let $\mathbb{S}(\mathcal{A})$ denote the set of ultrafilters in $\mathcal{A}$. For each~$a \in \mathcal{A}$, let $\mathbb{S}(a) \coloneqq \{ \alpha \in \mathbb{S}(\mathcal{A}) \mid a \in \alpha \}$. Equipping $\mathbb{S}(\mathcal{A})$ with the topology generated by the (compact open) cylinder sets $\mathbb{S}(a)$ turns it into a Boolean space. For a proper Boolean homomorphism $\psi \colon \mathcal{A} \to \mathcal{B}$ and an ultrafilter $\beta \in \mathbb{S}(\mathcal{B})$, let $\mathbb{S}(\psi)(\beta) \coloneqq \{ \psi^{-1}(b) \mid b \in \beta \}$. This makes $\mathbb{S}(-)$ a contravariant functor in the other direction, and we refer to it as the \emph{Stone functor}. \emph{Stone duality} asserts that the contravariant functors $\ck(-)$ and $\mathbb{S}(-)$ implement a dual equivalence. In other words, the category of Boolean spaces is dually equivalent to the category of generalized Boolean algebras. It is more common to state Stone duality in terms of Stone spaces and Boolean algebras. This is just the restriction of the duality above to the aforementioned sub-categories.

For a generalized Boolean algebra $\mathcal{A}$, we let $\aut(\mathcal{A})$ denote the group of Boolean isomorphisms from $\mathcal{A}$ to $\mathcal{A}$.

\subsection{Étale groupoids}\label{subs:groupoids}
The standard references for étale groupoids (and their \mbox{$C^*$-algebras}) are Renault's thesis~\cite{Ren} and Paterson's book~\cite{Pat}. See also the excellent lecture notes by Sims~\cite{Sims}. A \emph{groupoid} is a small category of isomorphisms, that is, a set $\GG$ (the morphisms, or arrows in the category) equipped with a partially defined multiplication~$(g_1, g_2) \mapsto g_1 \cdot g_2$ for a distinguished subset $\mathcal{G}^{(2)} \subseteq \mathcal{G} \times \mathcal{G}$, and everywhere defined involution $g \mapsto g^{-1}$ satisfying the following axioms:
\begin{enumerate}
\item If $g_1g_2$ and $(g_1g_2)g_3$ are defined, then $g_2g_3$ and $g_1(g_2g_3)$ are defined and $(g_1g_2)g_3=g_1(g_2g_3)$,
\item The products $gg^{-1}$ and $g^{-1}g$ are always defined. If $g_1g_2$ is defined, then $g_1=g_1g_2g_2^{-1}$ and $g_2=g_1^{-1}g_1g_2$.
\end{enumerate}

A \emph{topological groupoid} is a groupoid equipped with a topology making the operations of multiplication and taking inverse continuous. The elements of the form $gg^{-1}$ are called \emph{units}. We denote the set of units of a groupoid $\GG$ by $\GG^{(0)}$, and refer to this as the \emph{unit space}. We think of the unit space as a topological space equipped with the relative topology from~$\mathcal{G}$. The \emph{source} and \emph{range} maps are
 \[s(g) \coloneqq g^{-1}g\qquad\text{and}\qquad r(g) \coloneqq gg^{-1}\]
for $g\in \GG$. These maps are necessarily continuous when $\mathcal{G}$ is a topological groupoid. We implicitly assume that all unit spaces appearing are of infinite cardinality (in order to avoid some degenerate cases). An \emph{étale} groupoid is a topological groupoid where the range map (and necessarily also the source map) is a local homeomorphism (as a map from $\mathcal{G}$ to $\mathcal{G}$). The unit space $\mathcal{G}^{(0)}$ of an étale groupoid is always an open subset of $\mathcal{G}$.  An \emph{ample} groupoid is an étale groupoid whose unit space is a Boolean space. 

It is quite common for operator algebraists to restrict to Hausdorff groupoids. One reason for this is that a topological groupoid is Hausdorff if and only if the unit space is a closed subset of the groupoid. In the end our main results will only apply to groupoids that are Hausdorff, but some of the theory applies when $\mathcal{G}$ is merely ample (and effective). For as long as the unit space $\mathcal{G}^{(0)}$ is Hausdorff the groupoid will be locally Hausdorff. We shall therefore clearly indicate whenever we actually need the groupoid to be Hausdorff for some result to hold.
 
Two units $x,y\in \GG^{(0)}$ belong to the same \emph{$\GG$-orbit} if there exists $g\in \GG$ such that $s(g)=x$ and $r(g)=y$.  We denote by $\orb_{\GG}(x)$ the $\GG$-orbit of $x$. When every $\GG$-orbit is dense in $\GG^{(0)}$, $\GG$ is called \emph{minimal}. In the special case that there is just one orbit, we call $\GG$ \emph{transitive}. A subset~$A \subseteq \GG^{(0)}$ is called \emph{$\mathcal{G}$-full} if $r(s^{-1}(A)) = \mathcal{G}^{(0)}$, in other words if $A$ meets every $\mathcal{G}$-orbit. For an open subset $A \subseteq \GG^{(0)}$ the subgroupoid~$\GG_{|A} \coloneqq\{g\in \GG \mid s(g), r(g)\in A \}$ is called the \emph{restriction} of $\GG$ to $A$. When $\mathcal{G}$ is étale, the restriction $\GG_{|A}$ is an open étale subgroupoid. The \emph{isotropy group} of a unit $x\in \GG^{(0)}$ is the group $\GG_x^x \coloneqq \{g\in \GG \mid s(g)=r(g)=x\}$, and the \emph{isotropy bundle} is
\[\GG' \coloneqq \{g\in \GG \mid s(g)=r(g)\} = \bigsqcup_{x \in \mathcal{G}^{(0)}} \GG_x^x.\]
A groupoid $\GG$ is said to be \emph{principal} if $\GG' = \mathcal{G}^{(0)}$, i.e.\ if all isotropy groups are trivial. Any principal groupoid can be identified with an equivalence relation on its unit space $\mathcal{G}^{(0)}$, but the topology need not be the relative topology from $\mathcal{G}^{(0)} \times \mathcal{G}^{(0)}$. We say that $\GG$ is \emph{effective} if the interior of $\GG'$ equals $\mathcal{G}^{(0)}$. We call $\mathcal{G}$ \emph{topologically principal} if the set of points in $\mathcal{G}^{(0)}$ with trivial isotropy group are dense in $\mathcal{G}^{(0)}$.
 
\begin{remark}\label{rem:effective}
We should point out that the condition we are calling \emph{effective} often goes under the name \emph{essentially principal} (or even topologically principal) elsewhere in the literature. In general, topologically principal implies effective. However, for most groupoids considered by operator algebraists the two notions are in fact equivalent (see~\cite[Proposition~3.1]{Ren2}), so often these names all mean the same thing. In particular, this is the case for second countable locally compact Hausdorff étale groupoids.
 \end{remark}
 
\begin{definition}
Let $\mathcal{G}$ be an étale groupoid. A \emph{bisection} is an open subset $U\subseteq \GG$ such that $s$ and $r$ are both injective when restricted to $U$. A bisection $U$ is called \emph{full} if we have~$s(U)=r(U)=\GG^{(0)}$.  
\end{definition}
 
When $U$ is a bisection in $\mathcal{G}$, then $s_{\vert U} \colon U \to s(U)$ is a homeomorphism, and similarly for the range map. An étale groupoid can thus be characterized by admitting a topological basis consisting of bisections, and an ample groupoid as one with a basis of compact bisections. In particular, ample groupoids are locally compact, and if $\mathcal{G}$ is Hausdorff and ample, then $\mathcal{G}$ is also a Boolean space. One of the most basic class of examples of étale groupoids are the following, which arise from group actions.

\begin{example}
Let $\Gamma$ be a discrete group acting by homeomorphisms on a topological space~$X$. The associated \emph{transformation groupoid} is  \[\Gamma \ltimes X \coloneqq  \Gamma \times X\] with product according to~${(\tau, \gamma(x)) \cdot (\gamma,x) = (\tau \gamma, x)}$ (and undefined otherwise), and inverse~$(\gamma,x)^{-1} = (\gamma^{-1}, \gamma(x))$. Identifying the unit space $\left(\Gamma \ltimes X \right)^{(0)} = \{1\} \times X$ with $X$ in the obvious way we have $s((\gamma, x)) = x$ and $r((\gamma, x)) = \gamma(x)$. Equipping $\Gamma \ltimes X$ with the product topology makes it an étale groupoid (essentially because $\Gamma$ is discrete), and a basis of bisections is given by the cylinder sets \[Z(\gamma, U) \coloneqq \{(\gamma,x) \ \vert \ x\in U\}\]
indexed over $\gamma \in \Gamma$ and open subsets $U \subseteq X$. The identification of $X$ with the unit space as above is compatible with this topology. In particular $\Gamma \ltimes X$ is Hausdorff and ample exactly when $X$ is Boolean, and second countable when $\Gamma$ is countable and $X$ is second countable. The transformation groupoid is effective if and only if every non-trivial group element has support equal to $X$. In the second countable setting, this coincides with the action being topologically principal (meaning that the set of points that are fixed only by the identity element of the group form a dense subset of $X$). The groupoid orbit $\orb_{\Gamma \ltimes X}(x)$ of a point $x \in X$ coincide with the orbit under the action, i.e.\ $\orb_{\Gamma \ltimes X}(x) = \{\gamma(x) \ \vert \gamma \in \Gamma \} = \orb_{\Gamma \curvearrowright X}(x)$. 
\end{example}

A \emph{groupoid homomorphism} between two groupoids $\mathcal{G}$ and $\mathcal{H}$ is a map $\Phi \colon \mathcal{G} \to \mathcal{H}$ such that $(\Phi(g), \Phi(g')) \in \mathcal{H}^{(2)}$ whenever $(g, g') \in \mathcal{G}^{(2)}$, and moreover $\Phi(g) \cdot \Phi(g') = \Phi(g\cdot g')$. It follows that $\Phi(g^{-1}) = \Phi(g)^{-1}$ for all $g \in \mathcal{G}$, $\Phi$ commutes with the source and range maps and $\Phi\left(\mathcal{G}^{(0)}\right) \subseteq \mathcal{H}^{(0)}$. If $\Phi$ is a bijection, then $\Phi^{-1}$ is a groupoid homomorphism and we call $\Phi$ an \emph{algebraic isomorphism}. For étale groupoids $\mathcal{G}$ and $\mathcal{H}$ an \emph{étale homomorphism} is a groupoid homomorphism $\Phi \colon \mathcal{G} \to \mathcal{H}$ which is also a local homeomorphism. It is a fact that a groupoid homomorphism $\Phi \colon \mathcal{G} \to \mathcal{H}$ between étale groupoids is a local homeomorphism if and only if the restriction $\Phi^{(0)} \colon \mathcal{G}^{(0)} \to \mathcal{H}^{(0)}$ to the unit spaces is a local homemorphism. By an \emph{isomorphism} of topogical (or étale) groupoids we mean an algebraic isomorphism which is also a homeomorphism. So a bijective étale homomorphism is an isomorphism of étale groupoids. Note that if $\Phi \colon \mathcal{G} \to \mathcal{H}$ is an étale homomorphism, then the image $\Phi(\mathcal{G})$ is an open étale subgroupoid of $\mathcal{H}$.

\section{The topological full group}\label{sec:tfg}

In this section we will expand Matui's definition of the topological full group of an ample groupoid from the compact to the locally compact case, and establish some elementary properties. To each bisection $U\subseteq \GG$ in an étale groupoid we associate a homeomorphism \[\pi_U \colon s(U)\to r(U)\] given by $r_{\vert U} \circ (s_{\vert U})^{-1}$. This means that for each $g \in U,$ $\pi_U$ maps $s(g)$ to $r(g)$. Whenever $U$ is a full bisection, $\pi_U$ is a homeomorphism of $\mathcal{G}^{(0)}$. We now show that the (partial) homeomorphism $\pi_U$ determines the bisection $U$, when the groupoid is effective and Hausdorff.

\begin{lemma}\label{lemma:homeoBis}
Let $\GG$ be an effective ample Hausdorff groupoid and let $U,V \subseteq \mathcal{G}$ be bisections with $s(U) = s(V)$ and $r(U) = r(V)$. If $\pi_U = \pi_V$, then $U = V$.
\end{lemma}
\begin{proof}
That $\pi_U = \pi_V$ means that for each $x \in s(U)$, the unique elements $g \in U, h \in V$ with~$s(g) = x = s(h)$ also satisfies $r(g) = r(h)$. This implies that $V^{-1}U \subseteq \mathcal{G}'$. As $\mathcal{G}$ is Hausdorff, $\mathcal{G}^{(0)}$ is closed, and therefore $V^{-1}U \cap \left( \mathcal{G} \setminus \mathcal{G}^{(0)}\right)$ is an open subset of $\mathcal{G}' \setminus \mathcal{G}^{(0)}$. But since $\mathcal{G}$ is effective this set must be empty. This entails that $V^{-1}U \subseteq \mathcal{G}^{(0)}$, and hence~$U = V$.
\end{proof}

\begin{definition}\label{def:tfg}
Let $\GG$ be an effective ample groupoid. The \emph{topological full group} of $\mathcal{G}$, denoted $\llbracket \mathcal{G} \rrbracket$, is the subgroup of $\homeo\left(\mathcal{G}^{(0)}\right)$ consisting of all homeomorphisms of the form $\pi_U$, where $U$ is a full bisection in $\GG$ such that $\supp(\pi_U)$ is compact. We will denote by~$\DD(\llbracket \mathcal{G} \rrbracket)$ its commutator subgroup.
\end{definition}

In the topological full group, composition and inversion of the homeomorphisms correspond to multiplication and inversion of the bisections, viz.: 
\begin{itemize}
\item $\pi_{\mathcal{G}^{(0)}} = \id_{\mathcal{G}^{(0)}} = 1$
\item $\pi_U \circ \pi_V = \pi_{UV}$
\item $\left(\pi_U\right)^{-1} = \pi_{U^{-1}}$
\end{itemize}

\begin{remark}\label{rem:matui}
It is clear that when the unit space is compact, this definition coincides with Matui's~\cite[Definition~2.3]{Mat1}---which again generalizes the definitions given in~\cite{GPS} and~\cite{Mats}, for Cantor dynamical systems and one-sided shifts of finite type, respectively, to étale groupoids. Moreover, in~\cite{Mat4} Matui defined six different full groups associated with a minimal homeomorphism $\phi$ of a locally compact Cantor space. The smallest one of these, denoted $\tau[\phi]_c$ in~\cite{Mat4}, equals the topological full group (as in Definition~\ref{def:tfg}) of the associated transformation groupoid.
\end{remark}

\begin{remark}
After the completion of this work, we were made aware of Matte Bon's preprint~\cite{MB} where he defines the topological full group of an arbitrary étale groupoid~$\mathcal{G}$ as the group of all full bisections $U \subseteq \mathcal{G}$ such that $U \setminus \mathcal{G}^{(0)}$ is compact. For effective groupoids, this agrees with Definition~\ref{def:tfg}, modulo identifying a full bisection with its associated homeomorphism. For not necessarily effective groupoids it is arguably better to define the topological full group in terms of the bisections themselves, for then one does not ``lose'' the information contained in the (non-trivial) isotropy (but also to separate the group from its canonical---no longer faithful---action on the unit space). This is done in e.g.~\cite{Nek} and~\cite{BrixS} as well. However, the approach taken in this paper---in particular in Section~\ref{sec:spatrel}---is based on working with subgroups of the homeomorphism group of a space (i.e.\ faithful group actions), which is why we have defined $\llbracket \mathcal{G} \rrbracket$ as we have.
\end{remark}

\begin{remark}
We emphasize that the topological full group $\llbracket \mathcal{G} \rrbracket$ is viewed as a \emph{discrete} group. The term \emph{topological} is historical, and refers to the fact that the homeomorphisms in the topological full group preserves orbits in a ``continuous way'', as opposed to the full groups, which appeared first---in the measurable setting---c.f.~\cite[page~2]{GPS}. 
\end{remark}

For descriptions of the topological full group in certain classes of examples, see Proposition~\ref{prop:bis}, Remark~\ref{AFfullgroups} and Remark~\ref{LCCMfullgroups}. See also~\cite{Mat2} for a survey of about topological full groups of étale groupoids with compact unit space. 

By virtue of the groupoid being effective, the support of a homeomorphism in the topological full group is in fact open as well. Matui's proof of this fact for compact unit spaces carries over verbatim to our setting.

\begin{lemma}[c.f.\ {\cite[Lemma~2.2]{Mat}}]\label{lem_clopen}
Let $\GG$ be an effective ample Hausdorff groupoid. Then $\supp(\pi_U) = s(U \setminus \mathcal{G}^{(0)})$ for each $\pi_U \in \llbracket \mathcal{G} \rrbracket$. In particular, $\supp(\pi_U)$ is a compact open subset of $\mathcal{G}^{(0)}$.
\end{lemma}

We now present a few basic results on the existence of elements in the topological full group. They will be used in later sections to construct elements in the topological full group with localized support.

\begin{lemma}\label{lem:extendBisection}
Let $\GG$ be an effective ample groupoid, and let $\pi_U \in \llbracket \mathcal{G} \rrbracket$. Then we have a decomposition 
\[U = U^\perp \bigsqcup \left( \mathcal{G}^{(0)} \setminus \supp(\pi_U) \right),\]
where $U^\perp$ is a compact bisection with $s(U^\perp) = r(U^\perp) = \supp(\pi_U)$.

Conversely, any compact bisection $V \subseteq \mathcal{G}$ with $s(V) = r(V)$ defines an element $\pi_{\tilde{V}} \in \llbracket \mathcal{G} \rrbracket$ with $\supp(\pi_{\tilde{V}}) \subseteq s(V)$ by setting $\tilde{V} = V \sqcup \left( \mathcal{G}^{(0)} \setminus s(V) \right)$.
\end{lemma}
\begin{proof}
It is clear that $\supp(\pi_U)$ is invariant under $\pi_U$. Therefore we may simply put $U^\perp = \smash{s_{\vert U}^{-1}(\supp(\pi_U))}$. The second statement is obvious.
\end{proof}

\begin{lemma}\label{lemma:bisectionInvolution}
Let $\GG$ be an effective ample groupoid. Any compact bisection $V \subseteq \mathcal{G}$ satisfying~$s(V) \cap r(V) = \emptyset$ defines an involutive element $\pi_{\hat{V}} \in \llbracket \mathcal{G} \rrbracket$ with $\supp(\pi_{\hat{V}}) \subseteq s(V) \cup r(V)$ by setting $\hat{V} = V \sqcup V^{-1} \sqcup \left( \mathcal{G}^{(0)} \setminus \left( s(V) \cup r(V) \right) \right)$.
\end{lemma}
\begin{proof}
Immediate.
\end{proof}

\begin{lemma}\label{lemma:bisectionExistence}
Let $\mathcal{G}$ be an effective ample groupoid. If $g \in \mathcal{G} \setminus \mathcal{G}'$, that is $s(g) \neq r(g)$, then there is a (nontrivial) bisection $U \subseteq \mathcal{G}$ containing $g$ with $\pi_U \in \llbracket \mathcal{G} \rrbracket$. Furthermore, for any open set $A \subseteq \mathcal{G}^{(0)}$ containing both $s(g)$ and $r(g)$, $U$ can be chosen so that $\supp(\pi_U) \subseteq A$. We may also choose $\pi_U$ to be an involution.
\end{lemma}
\begin{proof}
As $\mathcal{G}$ is ample there is a compact bisection $W$ containing $g$. Let $B_1$, $B_2$ be disjoint open neighbourhoods of $s(g)$, $r(g)$ respectively in $\mathcal{G}^{(0)}$. By intersecting we may take~${B_1 \subseteq s(W) \cap A}$ and ${B_2 \subseteq r(W) \cap A}$. By continuity of $s$ and $r$ there are compact open sets $W_1, W_2 \subseteq W$, both containing $g$, such that $s(W_1) \subseteq B_1$ and $r(W_2) \subseteq B_2$. And then~$V = W_1 \cap W_2$ is a compact bisection containing $g$ with $s(V) \cap r(V) = \emptyset$ and~$s(V) \cup r(V) \subseteq A$. Hence $U = \hat{V}$ (as in Lemma~\ref{lemma:bisectionInvolution}) is the desired full bisection.
\end{proof}

\begin{remark}
In the non-compact case we may view the topological full group as a direct limit of topological full groups of groupoids over \emph{compact spaces} as follows. Consider~$\smash{\ck\left(\mathcal{G}^{(0)}\right)}$ as a directed set (ordered by inclusion). Given $A,B \in \smash{\ck\left(\mathcal{G}^{(0)}\right)}$ with~$A\subseteq B$ we define the group homomorphism $\iota_{A,B} \colon \llbracket \mathcal{G}_A \rrbracket \to \llbracket \mathcal{G}_B \rrbracket$ by $\pi_U\mapsto \pi_{\tilde{U}}$ where we set~$\tilde{U} = U \sqcup (B\setminus A)$. Then we have that \[\llbracket \mathcal{G} \rrbracket  \cong \lim\limits_{\rightarrow}(\llbracket \mathcal{G}_A \rrbracket, \iota).\]
\end{remark}

\section{The groupoid of germs}\label{sec:germ}

We are now going to adapt the notions of~\cite[Section~3]{Ren2} to the (special) case of groups, rather than inverse semigroups, to fit the framework of the topological full group and its subgroups, rather than the \emph{pseudogroup} studied in~\cite{Ren2}. Our goal is to reconstruct an ample groupoid $\mathcal{G}$ from subgroups of the topological full group $\llbracket \mathcal{G} \rrbracket$ as a so-called \emph{groupoid of germs}---which is a quotient of a transformation groupoid.

\begin{remark}
In the following three sections we will be working with subgroups of $\homeo(X)$, where~$X$ is a topological space. Thus we are essentially studying faithful actions by discrete groups on $X$. In the end we will have $X = \mathcal{G}^{(0)}$ for some ample groupoid~$\mathcal{G}$, and we will be looking at subgroups of $\llbracket \mathcal{G} \rrbracket$. Yet it will be convenient to state most results for general subgroups $\Gamma \leq \homeo(X)$ without reference to groupoids. Also, beware that the term \emph{faithful} will be used differently in Section~\ref{sec:spatrel} (c.f.\ Definition~\ref{def:faithful}).
\end{remark}

Recall that two homeomorphisms $\gamma, \tau \colon X \to X$ have the same \emph{germ} at a point $x \in X$ if there is a neighbourhood $U$ of $x$ such that $\smash{{\gamma}_{|U} = {\tau}_{|U}}$.

\begin{definition}
Let $X$ be a locally compact Hausdorff space and let $\Gamma \leq \homeo(X)$. The \emph{groupoid of germs} of $(\Gamma,X)$ is 
\[\G(\Gamma,X) \coloneqq \left(\Gamma \ltimes X\right) / \sim \]
where $(\gamma, x) \sim (\tau, y)$ iff $x=y$ and $\gamma, \tau$ have the same germ at $x$. 
\end{definition}

Denote the equivalence class of $(\gamma, x) \in \Gamma \ltimes X$ under $\sim$ by $[\gamma, x]$. It is straightforward to check that the groupoid operations of the transformation groupoid are well-defined on representatives of the equivalence classes in the groupoid of germs (and that they are continuous). The bisections 
\[Z[\gamma, A] \coloneqq \{[\gamma,x] \ \vert \ x\in A\},\] for $\gamma \in \Gamma$ and $A \subseteq X$ open, form a basis for the quotient topology. The unit space of~$\G(\Gamma,X)$ is also identified with $X$ in the obvious way. Hence the groupoid $\G(\Gamma,X)$ is étale (and ample when $X$ is Boolean), and it is furthermore always effective (as any group element acting identically on an open set is identified with the identity at each point of this open set). Hausdorffness of the groupoid however, is no longer guaranteed, but it can be characterized as follows.

\begin{lemma}\label{lem_haus}
Let $X$ be a locally compact Hausdorff space and let $\Gamma \leq \homeo(X)$. Then the groupoid of germs $\G(\Gamma,X)$ is Hausdorff if and only if $\supp(\gamma)$ is clopen in $X$ for every $\gamma \in \Gamma$.
\end{lemma}
\begin{proof}
Since $X$ is Hausdorff, any two groupoid elements $[\gamma, x], [\tau, y] \in \G(\Gamma,X)$ with distinct sources (i.e.\ $x \neq y$) can always be separated by open sets. We only have to worry about separating elements in the same isotropy group, and it suffices to be able to separate the unit from any other element. Also note that $[\gamma, x] \neq [1, x]$ if and only if $x \in \supp(\gamma)$.

First, assume that all the supports are clopen. If $[\gamma, x] \neq [1, x]$, then by the observation above, $Z[\gamma, \supp(\gamma)]$ and $Z[1, \supp(\gamma)]$ are disjoint open neighbourhoods of these elements. To separate $[\gamma, x]$ from $[\tau, x]$ (when these are distinct), we first note that~${[\gamma, x][\tau, x]^{-1} = [\gamma \tau^{-1}, \tau(x)] \neq [1, \tau(x)]}$. Hence $\tau(x) \in \supp(\gamma \tau^{-1})$, so by the argument above $Z[\gamma \tau^{-1}, A]$ and $Z[1, A]$, with $A = \supp(\gamma \tau^{-1})$, separates $[\gamma \tau^{-1}, \tau(x)]$ from~$[1, \tau(x)]$. It follows that $Z[\gamma, \tau^{-1}(A)]$ and $Z[\tau, \tau^{-1}(A)]$ separates $[\gamma, x]$ and $[\tau, x]$.

Conversely, suppose there is a $\gamma\in\Gamma$ such that $\supp(\gamma)$ is not open. Let $x$ be any point on the boundary of $\supp(\gamma)$. Then $\gamma(x) = x$, but $[\gamma, x] \neq [1, x]$, and these two groupoid elements cannot be separated by open sets. To see this take any two basic neighbourhoods~$Z[\gamma, A], Z[1, B]$ where $A,B$ are open neighbourhoods of $x$ in $X$. They both contain the basic set $Z[1, C]$ where $C = (A \cap B) \setminus \supp(\gamma)$, since $\gamma$ acts identically on $C$.
\end{proof}

In the sequel we shall restrict our attention to groups of homeomorphisms which have open, as well as compact, support. Topological full groups are determined by the ``local behaviour'' of its elements. This is made precise in the following definition.

\begin{definition}
Let $X$ be a locally compact Hausdorff space and let $\Gamma\leq \homeoc(X)$. We say that a homeomorphism $\varphi\in \homeoc(X)$ \emph{locally belongs to $\Gamma$} if for every $x\in X$, there exists an open neighborhood $U$ of $x$ and $\gamma\in \Gamma$ such that $\varphi_{|U}=\gamma_{|U}$. The group  $\Gamma$ is called \emph{locally closed} if whenever $\varphi\in \homeoc(X)$ locally belongs to $\Gamma$, then $\varphi\in \Gamma$. 
\end{definition}

\begin{proposition}\label{prop_TFGample}
Let $\GG$ be an effective ample Hausdorff groupoid. Then the topological full group $\llbracket \mathcal{G} \rrbracket \leq \smash{\homeoc\left(\mathcal{G}^{(0)}\right)}$ is locally closed.
\end{proposition}
\begin{proof}
Let $\varphi \in \homeoc(\mathcal{G}^{(0)})$ locally belong to $\llbracket \mathcal{G} \rrbracket$. Then, since $\supp(\varphi)$ is compact open, we can find finitely many open sets $A_i \subseteq \supp(\varphi)$, covering $\supp(\varphi)$, such that~$\smash{\varphi_{\vert A_i} = (\pi_{U_i})_{\vert A_i}}$, where $\pi_{U_i} \in \llbracket \mathcal{G} \rrbracket$. Since $\mathcal{G}^{(0)}$ is Boolean we may assume that the $A_i$'s are clopen and disjoint. We then have a clopen partition $\supp(\varphi) = A_1 \sqcup A_2 \sqcup \cdots \sqcup A_n$, and~$\varphi$ restricts to a self-homeomorphism of $\supp(\varphi)$ which on each region $A_i$ equals $\pi_{U_i}$. It follows that the set $\smash{V = \cup_{i=1}^n V_i}$, where $\smash{V_i = \left(s_{\vert U_i} \right)^{-1}(A_i)}$, is a compact bisection in $\mathcal{G}$ with~$s(V) = \supp(\varphi) = r(V)$. And then $\varphi = \pi_{\tilde{V}} \in \llbracket \mathcal{G} \rrbracket$, where $\tilde{V}$ is as in Lemma~\ref{lem:extendBisection}.
\end{proof}

Given a group $\Gamma\leq \homeoc(X)$ we denote by $\langle \Gamma \rangle$ the set of $\varphi \in \homeoc(X)$ which locally belong to $\Gamma$. Clearly $\langle \Gamma \rangle$ is a locally closed group in $\homeoc(X)$ and $\Gamma \leq \langle \Gamma \rangle$. As the groupoid of germs is defined in the same local terms as the local closure we have a canonical isomorphism $\G(\langle \Gamma \rangle, X) \cong \G(\Gamma,X)$. From this we obtain the analog of~\cite[Proposition~3.2]{Ren2}, namely that the topological full group of a groupoid of germs equals the local closure of the group we started with.

\begin{proposition}\label{prop_ample}
Let $X$ be a Boolean space and let $\Gamma \leq \homeoc(X)$. Then we have that $\llbracket \G(\Gamma,X) \rrbracket \cong \langle \Gamma \rangle$. 
\end{proposition} 
\begin{proof}
Since $\G(\Gamma,X) \cong \G(\langle \Gamma \rangle, X)$, it suffices to show that $\llbracket \G(\langle \Gamma \rangle, X) \rrbracket = \langle \Gamma \rangle$. For each $\varphi \in \langle \Gamma \rangle$ the full bisection $Z[\varphi, X] = U_\varphi$ in $\G(\langle \Gamma \rangle, X)$ satisfies $\pi_{U_\varphi} = \varphi$. And since $\varphi$ has compact support it belongs to $\llbracket \G(\langle \Gamma \rangle, X) \rrbracket$.

For the reverse inclusion, take any $\pi_U \in \llbracket \G(\langle \Gamma \rangle, X) \rrbracket$. Recall that the support of~$\pi_U$ is open, as well as compact, since any groupoid of germs is effective (c.f.\ Lemma~\ref{lem_clopen}). To see that $\pi_U$ locally belongs to $\Gamma$ take any $x \in X$, and let $[\varphi, x]$ be the unique element in~$U$ whose source is $x$. Since $U$ is open there is a basic set $Z[\varphi, A] \subseteq U$, where $A$ is an open neighbourhood of $x$ in $X$. As $\varphi \in \langle \Gamma \rangle$ there is an open neighbourhood $B$ of $x$ and an element~$\gamma \in \Gamma$ with $\smash{\varphi_{\vert B} = \gamma_{\vert B}}$. By intersecting with $A$ we may assume that $B \subseteq A$. Now observe that~$\smash{(\pi_U)_{\vert B} = \varphi_{\vert B} = \gamma_{\vert B}}$, and we are done.
\end{proof}

As topological full groups are locally closed (Proposition~\ref{prop_TFGample}) we obtain the following immediate corollary.

\begin{corollary}\label{cor:tfgGerm}
Let $\mathcal{G}$ be an effective ample Hausdorff groupoid. Then \[\llbracket \G(\llbracket \mathcal{G} \rrbracket, \mathcal{G}^{(0)}) \rrbracket \cong \llbracket \mathcal{G} \rrbracket. \]
\end{corollary}

The preceding results show that a locally closed group $\Gamma \leq \homeoc(X)$ can be reconstructed from its associated groupoid of germs $\smash{\G\left(\Gamma,\mathcal{G}^{(0)}\right)}$, namely as the topological full group of this groupoid. We now turn to the question of how an ample groupoid $\mathcal{G}$ relates to the groupoid of germs, $\smash{\G\left(\llbracket \mathcal{G} \rrbracket, \mathcal{G}^{(0)}\right)}$, determined by its topological full group. We will see that these will also be isomorphic under some mild condition on the groupoid---namely that the groupoid can be covered by bisections as in the following definition.

\begin{definition}
Let $\mathcal{G}$ be an effective ample groupoid. We say that a subgroup $\Gamma \leq \llbracket \mathcal{G} \rrbracket$ \emph{covers $\mathcal{G}$} if there for each $g \in \mathcal{G}$ exists a $\pi_U \in \Gamma$ such that $g \in U$.
\end{definition}

Note that if $\Gamma \leq \llbracket \mathcal{G} \rrbracket$ covers $\mathcal{G}$, then so does any group $\Gamma'$ in between, i.e.\ $\Gamma \leq \Gamma' \leq \llbracket \mathcal{G} \rrbracket$, and in particular $\llbracket \mathcal{G} \rrbracket$ itself covers $\mathcal{G}$.  Sufficient conditions on the orbits of $\mathcal{G}$ for $\llbracket \mathcal{G} \rrbracket$, or the commutator $\DD(\llbracket \mathcal{G} \rrbracket)$, to cover $\mathcal{G}$ is given by the following result (which is the analog of~\cite[Lemma~3.7]{Mat}).

\begin{lemma}\label{ex_cover}
Let $\mathcal{G}$ be an effective ample groupoid.
\begin{enumerate}
\item If $\vert \orb_\GG(x) \vert \geq 2$ for every $x\in \GG^{(0)}$, then $\llbracket \mathcal{G} \rrbracket$ covers $\mathcal{G}$.
\item If $\vert \orb_\GG(x) \vert \geq 3$ for every $x\in \GG^{(0)}$, then $\DD(\llbracket \mathcal{G} \rrbracket)$ covers $\mathcal{G}$.
\end{enumerate}
\end{lemma}
\begin{proof}
(1) First consider $g \in \mathcal{G} \setminus \mathcal{G}'$. Then Lemma~\ref{lemma:bisectionExistence} immediately gives a~$\pi_U \in \llbracket \mathcal{G} \rrbracket$ with~$g \in U$. Next, suppose $s(g) = r(g) = x$. By assumption there is a point $y$ different from $x$ in $\orb_\GG(x)$. This means that there is some $h \in \mathcal{G}$ with $s(h) = x \neq y = r(h)$. And then $h^{-1}$ is composable with $g$ and $gh^{-1} \in \mathcal{G} \setminus \mathcal{G}'$.  Applying Lemma~\ref{lemma:bisectionExistence} to both $gh^{-1}$ and~$h$ we get $\pi_{U_1}, \pi_{U_2}  \in \llbracket \mathcal{G} \rrbracket$ with $gh^{-1} \in U_1$ and $h \in U_2$. Since $\pi_{U_1U_2} \in \llbracket \mathcal{G} \rrbracket$ and $g \in U_1U_2$ we see that $\llbracket \mathcal{G} \rrbracket$ covers $\mathcal{G}$. 

(2) As in the previous part we first consider $g \in \mathcal{G} \setminus \mathcal{G}'$. By assumption there is a third (distinct) point $y$ in the same orbit as $s(g)$ and $r(g)$. Therefore there is an element $h \in \mathcal{G}$ with $s(h) = y$ and $r(h) = s(g)$.  Lemma~\ref{lemma:bisectionExistence} gives involutions $\pi_U, \pi_V \in \llbracket \mathcal{G} \rrbracket$ such that $g \in U$ and $h \in V$. We may also arrange so that $y \notin \supp(\pi_U)$ by the second part of Lemma~\ref{lemma:bisectionExistence}. Then
\[ [\pi_U, \pi_V] = \pi_U \pi_V (\pi_U)^{-1} (\pi_V)^{-1} = \pi_{(UV)^2}  \in \DD(\llbracket \mathcal{G} \rrbracket), \]
and we claim that $g$ belongs to the associated full bisection $(UV)^2$. To see that this is the case, note that $y \in U$ since $y \notin \supp(\pi_U)$. Thus we have $g = g\cdot h \cdot y \cdot h^{-1} \in UVUV$ as~$s(h) = y$.

Finally, for the case $s(g) = r(g)$ we proceed similar as in part (1). We take $h \in \mathcal{G}$ with~$s(h) = s(g)$ and $r(h) \neq s(g)$ and apply the above part to $gh^{-1}$ and $h$, which both belong to $\mathcal{G} \setminus \mathcal{G}'$. Multiplying the bisections we get gives the desired bisection containing~$g$. 
\end{proof}

The conditions in Lemma~\ref{ex_cover} are not necessary (c.f.\ Example~\ref{ex:1orbitCover}), but they are typically easy to check in specific examples. Note that for minimal groupoids all orbits are in particular infinite, so the covering as above is automatic. We are now ready to give the main result on how a groupoid $\mathcal{G}$ can be reconstructed from the germs of $\llbracket \mathcal{G} \rrbracket$. It is the analog of~\cite[Proposition~3.2]{Ren2}.

\begin{proposition}\label{prop_germs}
Let $\mathcal{G}$ be an effective ample Hausdorff groupoid and let $\Gamma \leq \llbracket \mathcal{G} \rrbracket$. Then there is an injective étale homomorphism 
\[\smash{\iota \colon \G\left(\Gamma, \mathcal{G}^{(0)} \right) \hookrightarrow \mathcal{G}}\]
given by $\iota([\pi_U,x]) = (s_{\vert U})^{-1}(x)$ for ${[\pi_U,x] \in \G\left(\Gamma, \mathcal{G}^{(0)} \right)}$. Furthermore, $\iota$ is surjective, and hence an isomorphism, if and only if $\Gamma$ covers $\mathcal{G}$.
\end{proposition}
\begin{proof}
We first have to verify that $\iota$ is well-defined. Let $x \in \mathcal{G}^{(0)}$ and suppose that $\pi_U, \pi_V \in \Gamma$ have the same germ over $x$. Let $A$ be an open neighbourhood of $x$ on which $\pi_U$ and $\pi_V$ agree. Then 
\[\pi_{UA} = \left(\pi_U\right)_{|A} = \left(\pi_V\right)_{|A} = \pi_{VA},\] so by Lemma~\ref{lemma:homeoBis} we have $UA = VA$. This means that the unique groupoid elements in $U$ and $V$ that have source equal to $x$ coincide, so $\iota$ is well-defined.

To see that $\iota$ is a groupoid homomorphism recall that $([\pi_V,y],[\pi_U,x])$ is a composable pair iff $\pi_U(x) = y$. Suppose this is the case and let $g \in U$ be the element with $s(g) = x$, and let $h \in V$ be the element with $s(h) = y$. As $r(g) = \pi_U(x) = y = s(h)$ we have $(h,g) \in \mathcal{G}^{(2)}$ and 
\[\iota([\pi_V,y] \cdot [\pi_U,x]) = \iota([\pi_{VU},x]) = hg,\] since $hg \in VU$ and $s(hg) = x$.

Now note that $\iota(x) = x$ for $x \in \mathcal{G}^{(0)}$ (under the identification of the unit space of the groupoid of germs). So $\iota^{(0)} = \id_{\mathcal{G}^{(0)}}$ is a (local) homeomorphism, hence $\iota$ is an étale homomorphism.

To see that $\iota$ is injective note first that $\iota([\pi_U,x]) \neq \iota([\pi_V,y])$ if $x \neq y$ since $\iota^{(0)}$ is the identity. Suppose now that $\iota([\pi_U,x]) = \iota([\pi_V,x])$ for some $\pi_U, \pi_V \in \Gamma$. This means that there is a groupoid element $g \in U \cap V$ with $s(g) = x$. Thus $B = s(U \cap V)$ is an open neighbourhood of $x$ in $\mathcal{G}^{(0)}$ and clearly $\left(\pi_U\right)_{|B} = \left(\pi_V\right)_{|B}$, which means that $[\pi_U,x] = [\pi_V,x]$.

Finally, that $\iota$ is surjective is easily seen to be the same as $\Gamma$ covering $\mathcal{G}$.
\end{proof}

\begin{remark}
When the map $\iota$ in the previous proposition is an isomorphism the inverse is given by $\iota^{-1}(g) = [\pi_U,s(g)]$, where $U$ is any full bisection such that $\pi_U \in \Gamma$ and $g \in U$.
\end{remark}

\begin{remark}
Let $\mathcal{G}$ be an effective ample Hausdorff groupoid. Combining Propositions~\ref{prop_germs} and~\ref{prop_ample} we see that for each locally closed subgroup $\Gamma \leq \llbracket \mathcal{G} \rrbracket$, there is an open  étale subgroupoid $\mathcal{H}_\Gamma \subseteq \mathcal{G}$ such that $\llbracket \mathcal{H}_\Gamma \rrbracket \cong \Gamma$, namely $\mathcal{H}_\Gamma = \G (\Gamma, \mathcal{G}^{(0)})$.
\end{remark}

Since we are really interested in knowing when $\mathcal{G}$ is isomorphic to $\smash{\G\left(\Gamma, \mathcal{G}^{(0)} \right)}$ (particularly for the case $\Gamma = \llbracket \mathcal{G} \rrbracket$) it is natural to ask whether they could be isomorphic even if the canonical map $\iota$ fails to be an isomorphism. We will see shortly that this is not possible. For~$\Gamma \leq \homeoc(X)$ with $X$ Boolean we have seen that $\Gamma \leq \langle \Gamma \rangle \cong \llbracket \G(\Gamma,X) \rrbracket$. Identifying the latter two we see that $\Gamma$ covers $\G(\Gamma,X)$ since $[\gamma, x] \in Z[\gamma, X]$ and~$\pi_{Z[\gamma, X]} = \gamma \in \Gamma$ for each $[\gamma, x] \in \G(\Gamma,X)$.

\begin{corollary}\label{cor:necessaryCover}
Let $\mathcal{G}$ be an effective ample Hausdorff groupoid. Then~$\smash{\G\left(\llbracket \mathcal{G} \rrbracket, \mathcal{G}^{(0)} \right)}$ and~$\mathcal{G}$ are isomorphic as étale groupoids if and only if $\llbracket \mathcal{G} \rrbracket$ covers $\mathcal{G}$.
\end{corollary} 
\begin{proof}
Suppose $\smash{\Phi \colon \mathcal{G} \to \G\left(\llbracket \mathcal{G} \rrbracket, \mathcal{G}^{(0)}\right)}$ is an isomorphism. Then $\Phi$ induces an isomorphism between the topological full groups by $\pi_U \mapsto \pi_{\Phi(U)}$ for $\pi_U \in \llbracket \mathcal{G} \rrbracket$. Let $g \in \mathcal{G}$ be given. As $\llbracket \mathcal{G} \rrbracket$ covers $\smash{\G\big(\llbracket \mathcal{G} \rrbracket, \mathcal{G}^{(0)} \big)}$ there is a full bisection $V$ containing $\Phi(g)$ such that $\pi_V \in \llbracket \G(\llbracket \mathcal{G} \rrbracket, \mathcal{G}^{(0)}) \rrbracket = \llbracket \mathcal{G} \rrbracket$. And then $\Phi^{-1}(V)$ is a full bisection in $\mathcal{G}$ containing~$g$ with~$\pi_{\Phi^{-1}(V)} \in \llbracket \mathcal{G} \rrbracket$. Hence $\llbracket \mathcal{G} \rrbracket$ covers $\mathcal{G}$.
\end{proof}

\section{The category of spatial groups}\label{sec:SpatG}
In this section we will study the groupoid of germs from a categorical point of view. By introducing suitable categories we will see that the assigment $(\Gamma, X) \mapsto \G(\Gamma,X)$ is indeed functorial. We will also see that certain equivariant maps between the spaces induce embeddings of the groupoids of germs.

\begin{definition}
The category of \emph{spatial groups}, denoted $\TopG$, consists of pairs $(\Gamma,X)$, where $X$ is a Boolean space and $\Gamma \leq \homeoc(X)$. A morphism in $\TopG$ from~$(\Gamma_1,X_1)$ to~$(\Gamma_2,X_2)$ is a local homeomorphism $\phi \colon X_1\to X_2$ satisfying $\phi \circ \Gamma_1\subseteq \Gamma_2\circ \phi$.
\end{definition}

We shall sometimes refer to a pair $(\Gamma, X)$ as a \emph{space-group pair}. Observe that an isomorphism in the category $\TopG$ is a homeomorphism $\phi$ such that $\phi \circ \Gamma_1\circ \phi^{-1}=\Gamma_2$. We call such an isomorphism a \emph{spatial isomorphism} (as it is a group isomorphism implemented by a homeomorphism).

\begin{definition}
The category $\Groupoid$ consists of ample effective Hausdorff groupoids, and the morphisms are étale homomorphisms.
\end{definition}

\begin{remark}
The choice of morphisms in $\TopG$ is done so that they induce étale homomorphisms between the groupoid of germs in a natural way. As for the morphisms in $\Groupoid$, there are several reasons for stipulating that they should be étale homomorphisms (rather than merely continuous groupoid homomorphisms). First of all, since all the structure maps in an étale groupoid are local homeomorphisms, it is reasonable to prescribe that maps between étale groupoids should be as well. Moreover, the image under an étale homomorphism is always an open étale subgroupoid in the codomain. An important consequence of this is that an injective étale homomorphism induce (diagonal preserving) injective $*$-homomorphisms between both the full and reduced groupoid $C^*$-algebras, respectively (and also between the Steinberg algebras), c.f.~\cite[page~113]{BNRSW} and~\cite[Proposition~1.9]{Phil}. Whereas the groupoid $C^*$-algebra construction is not functorial in general.
\end{remark}

It is straightforward to check that $\TopG$ and $\Groupoid$ indeed are categories. We will now define a functor from $\TopG$ to $\Groupoid$, which on objects is the groupoid of germs. Let~$\phi$ be a spatial morphism between two space-group pairs $(\Gamma_1,X_1)$ and $(\Gamma_2,X_2)$ in $\TopG$. Given~$[\gamma,x] \in \G(\Gamma_1,X_1)$, there is a $\gamma' \in \Gamma_2$ with $\phi \circ \gamma = \gamma' \circ \phi$. We then propose to define an étale homomorphism $\G(\phi)$ from $\G(\Gamma_1,X_1)$ to $\G(\Gamma_2,X_2)$ by setting~${\G(\phi)([\gamma,x]) = [\gamma',\phi(x)]}$.

\begin{proposition}\label{prop:functorial}
The mapping $\G(\phi)$ described above is a well-defined étale homomorphism, and $\G(-) \colon \TopG \to \Groupoid$ is a (covariant) functor.
\end{proposition}
\begin{proof}
Let $\phi \colon (\Gamma_1,X_1) \to (\Gamma_2,X_2)$ be a spatial morphism. We first verify that $\G(\phi)$ is well-defined. Given $[\gamma,x] \in \G(\Gamma_1,X_1)$, suppose $\gamma',  \gamma'' \in \Gamma_2$ satisfy \[\phi \circ \gamma = \gamma' \circ \phi = \gamma'' \circ \phi.\] Then $\gamma'$ and $\gamma''$ agree on $\phi(X_1)$, which is an open neighbourhood of $\phi(x)$, hence we have~${[\gamma',\phi(x)] = [\gamma'',\phi(x)]}$. So the choice of $\gamma'$ doesn't matter. As for the choice of $\gamma$, suppose $\tau \in \Gamma_1$ has the same germ over $x$ as $\gamma$, i.e.\  $\smash{\gamma_{|A} = \tau_{|A}}$ for some open neighbourhood~$A$ of~$x$ in~$X_1$. Let $\tau' \in \Gamma_2$ satisfy $\phi \circ \tau = \tau' \circ \phi$. Then \[\smash{\gamma' \circ \phi_{|A} = \phi \circ \gamma_{|A} = \phi \circ \tau_{|A} = \tau' \circ \phi_{|A}}.\]
This means that $\smash{\gamma'_{|\phi(A)} = \tau'_{|\phi(A)}}$, hence $[\gamma',\phi(x)] = [\tau',\phi(x)]$. This shows that $\G(\phi)$ is well-defined.

Observe that the restriction to the unit spaces is just $\G(\phi)^{(0)} = \phi \colon X_1 \to X_2$. From this we obtain
\[s\left(\G(\phi)([\gamma,x])\right) = \phi(x) = \G(\phi)\left(s([\gamma,x])\right), \] and 
\[r\left(\G(\phi)([\gamma,x])\right) = \gamma' \circ \phi(x) = \phi \circ \gamma(x) =  \G(\phi)\left(r([\gamma,x])\right). \]
This means that $\G(\phi)$ takes composable pairs to composable pairs. As for preserving the product itself, we verify that
\begin{align*}
\G(\phi)([\tau,\gamma(x)]) \cdot \G(\phi)([\gamma,x]) &= [\tau',\phi \gamma(x)] \cdot [\gamma',\phi(x)] = [\tau' \gamma', \phi(x)] \\
&= \G(\phi)([\tau \gamma,x]), \text{ since } \phi \tau \gamma = \phi \tau' \gamma'.
\end{align*}
As $\G(\phi)^{(0)} = \phi$ is a local homeomorphism, we have shown that $\G(\phi)$ is an étale homomorphism. Similar computations as above shows that $\G(-)$ sends identity morphisms to identity morphisms and preserves composition of morphisms.
\end{proof}

We record some consequences of this functoriality.

\begin{corollary}\label{corol_inj_groupoid}
Let $\phi \colon (X_1,\Gamma_1)\to(X_2,\Gamma_2)$ be a spatial morphism in $\TopG$. 
\begin{enumerate}
\item If $\phi$ is a spatial isomorphism, then $\G(\phi) \colon \G(\Gamma_1,X_1)\to \G(\Gamma_2,X_2)$ is an isomorphism of étale groupoids.
\item $\G(\phi)^{(0)} = \phi$, in particular $\G(\phi)$ maps $X_1$ onto $X_2$ if and only if $\phi$ is surjective.
\item If $\phi \colon X_1\to X_2$ is injective, then $\G(\phi) \colon \G(\Gamma_1,X_1)\to \G(\Gamma_2,X_2)$ is also injective.
\item If $\phi \colon X_1\to X_2$ is surjective and $\phi \circ \Gamma_1 = \Gamma_2 \circ \phi$, then $\G(\phi) \colon \G(\Gamma_1,X_1)\to \G(\Gamma_2,X_2)$ is also surjective.
\end{enumerate}
\end{corollary}
\begin{proof}
Statement (1) follows immediately from functoriality, and statement (2) was observed in the proof of Proposition~\ref{prop:functorial}.

(3) Assume that $\phi \colon X_1\to X_2$ is injective. Then clearly $\G(\phi)$ maps elements with distinct sources to distinct elements. So suppose \[[\gamma',\phi(x)] = \G(\phi)([\gamma,x]) = \G(\phi)([\tau,x]) = [\tau',\phi(x)]. \]
Then $\smash{\gamma'_{|A} = \tau'_{|A}}$ for some open neighbourhood $A$ of $\phi(x)$ in $X_2$. As $\phi \circ \gamma = \gamma' \circ \phi$ and~${\phi \circ \tau = \tau' \circ \phi}$ we have that $\phi \circ  \gamma$ and $\phi \circ \tau$ agree on $\phi^{-1}(A)$. The injectivity of~$\phi$ now implies that $\gamma$ and $\tau$ agree on $\phi^{-1}(A)$, which is an open neighbourhood of $x$, hence~$[\gamma,x] = [\tau,x]$ and $\G(\phi)$ is injective.

(4) Suppose $\phi \colon X_1\to X_2$ is surjective and that $\phi \circ \Gamma_1 = \Gamma_2 \circ \phi$. Given an element $[\tau, y]$ in $\G(\Gamma_2, X_2)$, pick any $x \in X_1$ with $\phi(x) = y$. By assumption there is some $\gamma \in \Gamma_1$ such that~$\phi \circ \gamma = \tau \circ \phi$, and then $\G(\phi)([\gamma, x]) = [\tau, y]$.
\end{proof}

\begin{remark}
It is natural to ask whether a spatial morphism $\phi \colon (X_1,\Gamma_1)\to(X_2,\Gamma_2)$ induces a (algebraic) group homomorphism from $\Gamma_1$ to $\Gamma_2$. This is not so clear. But at least if~$\phi \colon X_1\to X_2$ is injective and $\Gamma_2$ is locally closed, then one can define an injective group homomorphism $f_\phi \colon \Gamma_1 \to \Gamma_2$ in the following way. First observe that given~$\gamma \in \Gamma_1$, there is a~$\gamma_2 \in \Gamma_2$ with~$\phi \circ\gamma=\gamma_2\circ \phi$, and then $\gamma_2(\phi(X_1)) = \phi(X_1)$ and~$\smash{\supp((\gamma_2)_{|\phi(X_1)})= \phi(\supp(\gamma))}$. Given another $\gamma_3 \in \Gamma_2$ with $\phi \circ\gamma=\gamma_3\circ \phi$ we have \[\smash{(\gamma_2)_{|\phi(X_1)}=(\gamma_3)_{|\phi(X_1)}\in \homeoc(\phi(X_1))}.\] So we can define $f_\phi(\gamma) = \gamma'$ to be the homeomorphism $\gamma'$ on $X_2$ given by \[(\gamma')_{|\phi(X_1)}= (\gamma_2)_{|\phi(X_1)} \quad \text{and} \quad (\gamma')_{| X_2 \setminus \phi(X_1)} = \id_{X_2 \setminus \phi(X_1)}. \]
The homeomorphism $\gamma'$ belongs to $\Gamma_2$ because $\Gamma_2$ is locally closed. It is straightforward to check that $f_\phi$ is an injective group homomorphism, and also that $\supp(f_\phi(\gamma))= \phi(\supp(\gamma))$ for every $\gamma \in\Gamma_1$. If $\phi$ is a spatial isomorphism, then $f_\phi$ is a group isomorphism and~$f_\phi$ satisfies~$f_\phi(\gamma) = \phi \circ \gamma \circ \phi^{-1}$ for each $\gamma \in \Gamma_1$.
\end{remark} 

\begin{remark}
Viewing the functor $\G$ as a ``free'' functor turning a space-group pair into an effective ample Hausdorff groupoid (in the ``most efficient'' way), one could ask for a ``forgetful'' functor in the opposite direction. Proposition~\ref{prop_ample} suggests that this functor should be \[\smash{\llbracket - \rrbracket \colon  \Groupoid \to \TopG \quad \text{assigning} \quad \GG \mapsto \left(\llbracket \mathcal{G} \rrbracket,\mathcal{G}^{(0)}\right).} \]
The natural choice of mapping on the morphisms is for an étale homomorphism $\Phi \colon \mathcal{G} \to \mathcal{H}$ to let \[\smash{ \llbracket \Phi \rrbracket \coloneqq \Phi^{(0)} \colon \left(\llbracket \mathcal{G} \rrbracket, \mathcal{G}^{(0)} \right) \to \left(\llbracket \mathcal{H} \rrbracket, \mathcal{H}^{(0)} \right),} \]
i.e.\ restriction to the unit space. Unfortunately, this \emph{fails} to be a morphism in $\TopG$ in general. For injective étale homomorphisms though, the restriction to the unit spaces does yield an injective spatial morphism.
\end{remark}

\section{Spatial realization theorems}\label{sec:spatrel}

In this section we shall study reconstruction of topological spaces from subgroups of their homeomorphism group in the sense of the following definition. 

\begin{definition}\label{def:faithful}
A class $K$ of space-group pairs is called \emph{faithful} if every group isomorphism~$\Phi \colon \Gamma_1 \to \Gamma_2$, where $(\Gamma_1,X_1), (\Gamma_2,X_2) \in K$, is \emph{spatially implemented}, that is, there is a homeomorphism $\phi \colon X_1 \to X_2$ such that $\Phi(\gamma)=\phi \circ \gamma \circ \phi^{-1}$ for every $\gamma\in \Gamma_1$.
\end{definition}

We stress the fact that the isomorphisms $\Phi$ considered in the preceding definition are, a priori, \emph{abstract} group isomorphisms. They only ``see'' the algebraic structure of the $\Gamma_i$'s, not the actions on the underlying spaces. We may rephrase faithfulness to saying that ``every group isomorphism is a spatial isomorphism''. In relation to the previous section we obtain the following from Corollary \ref{corol_inj_groupoid}.

\begin{proposition}\label{faithGerm}
Suppose $K$ is a faithful class of space-group pairs from $\TopG$. If $(\Gamma_1,X_1)$ and $(\Gamma_2,X_2)$ belong to $K$ and $\Gamma_1$ is isomorphic to $\Gamma_2$ as abstract groups, then the groupoids of germs $\G(\Gamma_1,X_1)$ and  $\G(\Gamma_2,X_2)$ are isomorphic as topological groupoids.
\end{proposition}

In conjunction with Proposition~\ref{prop_germs} this will allow us to deduce that in many cases, the topological full group of an ample groupoid, considered as an abstract group, is a complete invariant for the isomorphism class of the groupoid. This will be done in the next section. The rest of this section will be devoted to proving two faithfulness results. The first one is a straightforward extension of Matui's spatial realization result~\cite[Theorem~3.5]{Mat} to our locally compact setting (Theorem~\ref{classF}). This result will not only apply to the topological full group, but also to any subgroup containing the commutator. The second result we present (Theorem~\ref{thm:KBfaithful}) has more relaxed assumptions on the ``mixing properties'' of the action, but we were not able to apply it to the commutator subgroup of the topological full group.

\subsection{The class $K^F$} 
We now present the main definition from~\cite[Section~3]{Mat}, adapted to our setting.

\begin{definition}\label{def_F}
We define the class \emph{$K^F$} to consist of all space-group pairs $(\Gamma,X) \in \TopG$ which satisfy the following conditions:
\begin{enumerate}
\item[(F1)] For any $x\in X$ and any clopen neighbourhood $A\subset X$ of $x$, there exists an involution $\alpha\in\Gamma$ such that $x\in \supp(\alpha)$ and $\supp(\alpha)\subseteq A$.
\item[(F2)] For any involution $\alpha\in\Gamma\setminus\{1\}$, and any non-empty clopen set $A\subseteq \supp(\alpha)$, there exists a $\beta\in\Gamma\setminus\{1\}$ such that $\supp(\beta)\subseteq A\cup\alpha(A)$ and $\alpha(x)=\beta(x)$ for every $x\in \supp(\beta)$.
\item[(F3)] For any non-empty clopen set $A\subseteq X$, there exists an $\alpha\in\Gamma$ such that $\supp(\alpha)\subseteq A$ and $\alpha^2\neq 1$.
\end{enumerate}
\end{definition}

\begin{remark}
In~\cite[Definition~3.1]{Mat} there is also a condition (F0), stipulating that the support of any involution should be clopen. This is already implicit in the definition above, since all supports of elements in $\Gamma$ are assumed to be compact and open. We also remark that Definition~\ref{def_F} does not impose any countability restrictions on the space $X$. However, condition~(F1) (and also (F3)) implies that $X$ cannot have isolated points.
\end{remark}

\begin{remark}
The notation $K^F$ to denote a class of space-group pairs is in the same style as Rubin uses in his paper~\cite{Rub}. Elsewhere in the literature, in particular~\cite{Mat} and~\cite{GPS}, groups $\Gamma$ with $(\Gamma, X) \in K^F$ are called groups of \emph{class F} (and $X$ is assumed to be a (compact) Cantor space).
\end{remark}
 
We now state a simple extension of Matui's Spatial Realization Theorem.  
 
\begin{theorem}[c.f.\ {\cite[Theorem~3.5]{Mat}}]\label{classF}
The class $K^F$ is faithful. 
\end{theorem}
\begin{proof} 
By closely inspecting the proof of~\cite[Theorem~3.5]{Mat} and the three lemmas preceding it, one finds that the compactness of the spaces is not needed until the proof of~\cite[Theorem~3.5]{Mat} itself. The lemmas  preceding it are completely algebraic. Furthermore,  the compactness is used only to guarantee that a certain intersection of supports become non-empty---by appealing to the finite intersection property. However, since all supports in our setting are already compact (by assumption) the conclusion that the intersection is non-empty still holds. The second countability is never needed. Therefore, Matui's proof remains valid.
\end{proof}

\begin{remark}
We remark that Matui's proof of~\cite[Theorem~3.5]{Mat} is similar to the approach used by Bezuglyi and Medynets in~\cite[Section~5]{BezMed}, wherein the authors prove a precursor of Matui's Isomorphism Theorem for Cantor minimal systems. Both of these build on Fremlin's book~\cite[Section~384]{Frem}.
\end{remark}

\subsection{The class $K^{LCC}$}
We now turn to obtaining the second spatial realization result, by providing another faithful class of space group-pairs. In comparison with $K^F$, we'll impose more restrictions on the spaces (second countability---resulting in locally compact Cantor spaces), but the conditions on the actions will be less ``localized'' in some sense. We will of course still need the groups $\Gamma$ to be very ``rich'' in order to recover the action on the space~$X$, but we do not focus solely on involutive group elements, as was the case for $K^F$.

Some of the (many) results from Rubins remarkable paper~\cite{Rub} will form the backbone of this spatial realization result. In that paper, Rubin exhibits the faithfulness of several general classes of space-group pairs. However, many of the classes considered there required quite different proofs. Arguably, the most commonly cited result from~\cite{Rub} in our context is~\cite[Corollary~3.5]{Rub}, but this spatial realization result is not strong enough to prove Theorem~\ref{intro:KLCCgroupoid}. We essentially end up reprove Rubin's result on $0$-dimensional spaces, but we obtain a slightly different statement. Also, our proof is a bit more straightforward (since we aim for a less general setting; namely perfect unit spaces of ample groupoids).

\subsubsection{Reconstructing the Boolean algebra $\mathcal{R}(X)$}
The main theorem from Section~$2$ of Rubin's paper (given below in Theorem~\ref{reconstructRX}) gives general conditions for when the abstract isomorphism class of a group $\Gamma \leq \homeo(X)$ determines the Boolean algebra $\mathcal{R}(X)$, and the induced action by $\Gamma$ on it.  We may view $\Gamma$ as a subgroup of $\aut(\mathcal{R}(X))$ by taking images of regular open sets in $\mathcal{R}(X)$ under the homeomorphisms in $\Gamma$. In~\cite[Section~3]{Rub}, Rubin defines several classes of space-group pairs and proves, in a case-by-case manner, that the space $X$ and the action by $\Gamma$ on it, can be recovered from the induced action of $\Gamma$ on $\mathcal{R}(X)$. Let us begin with some terminology (adapted from~\cite{Rub}).

\begin{definition}
Let $(\Gamma,X)$ be a space-group pair.
\begin{enumerate}
\item We say that $(\Gamma,X)$ is \emph{locally moving}  if for every non-empty open subset $A \subseteq X$ there exists $\gamma\in \Gamma\setminus\{1\}$ with $\supp(\gamma)\subseteq A$.
\item An open set $B \subseteq X$ is called \emph{flexible} if for every pair of open subsets $C_1,C_2\subseteq B$,  if there exists $\gamma \in \Gamma$ such that $\gamma(C_1)\cap C_2\neq\emptyset$, then there exists $\tau \in\Gamma$ such that $\tau(C_1)\cap C_2\neq \emptyset$ and $\supp(\tau)\subseteq B$.
\item We say that $(\Gamma,X)$ is \emph{locally flexible} if every non-empty open subset $A$ contains a non-empty open flexible subset $B \subseteq A$.
\end{enumerate} 
\end{definition}

\begin{remark}
Note that if $(\Gamma,X)$ is locally moving, then the space $X$ has no isolated points. 
\end{remark}

\begin{remark}
In~\cite{Rub}, ``locally moving'' goes by the name ``regionally disrigid'', whilst the former terminology is from a later paper of Rubin~\cite{Rub2}.
\end{remark}

We now state a special case of the main result from~\cite[Section~2]{Rub}.

\begin{theorem}[cf.\ {\cite[Theorem~0.2, Theorem~2.14(a)]{Rub}}] \label{reconstructRX}
Let $(\Gamma_1,X_1)$ and $(\Gamma_2,X_2)$ be in $\TopG$, and assume they are both locally moving and locally flexible. If $\Phi \colon \Gamma_1 \to \Gamma_2$ is an isomorphism of groups, then there exists a Boolean isomorphism $\psi \colon \mathcal{R}(X_1) \to \mathcal{R}(X_2)$ such that $\psi(g(A)) = \Phi(g)(\psi(A))$ for each $A \in \mathcal{R}(X_1)$ and $g \in \Gamma_1$.
\end{theorem}

If we think of $g$ and $\Phi(g)$ as elements in $\aut(\mathcal{R}(X_1))$ and $\aut(\mathcal{R}(X_2))$ respectively, then we can rewrite the conclusion in the preceding theorem as 
\[\Phi(g) = \psi \circ g \circ \psi^{-1}.  \] Thus, Theorem~\ref{reconstructRX} says that any group isomorphism between $\Gamma_1$ and $\Gamma_2$ is actually induced by an isomorphism of the Boolean algebras of regular open sets of the underlying spaces.

\begin{remark}
We remark that what Rubin proves in~\cite[Theorem~2.14(a)]{Rub} is a somewhat stronger statement than the one we gave above. First of all, the spaces need really only be Hausdorff (and perfect). Rubin shows that if $(\Gamma,X)$ is locally moving and locally flexible, then starting with $\Gamma$ alone, one can canonically reconstruct the Boolean algebra~$\mathcal{R}(X)$ (up to isomorphism) using only group theoretic constructions. Moreover, one obtains a natural action by $\Gamma$ on this Boolean algebra which is conjugate to the action by~$\Gamma$ on~$\mathcal{R}(X)$. The strategy of the proof is to model a regular set $A \in \mathcal{R}(X)$ by its rigid stabilizer~$Q(A) \coloneqq  \{\gamma \in \Gamma \ \vert \supp(\gamma) \subseteq A \}$, and then to describe the Boolean operations in $\mathcal{R}(X)$ in group theoretic terms, in terms of the subgroups $Q(A)$. Finally one shows that there are enough regular sets $A$ for which subgroups of the form $Q(A)$ can be detected inside $\Gamma$ in order to generate the whole of $\mathcal{R}(X)$.
\end{remark}

\subsubsection{Reconstructing the space $X$}
We now turn to reconstructing~$X$ (and the original action by~$\Gamma$) from its Boolean algebra of regular sets. The strategy is to first impose conditions making it possible to detect clopenness. And then characterize the compact open sets among the clopen sets, which in turn allow us to recover~$X$ from Stone duality.

\begin{definition}\label{recognizable}
Let $(\Gamma,X)$ be a space-group pair. A clopen set $A \subseteq X$ is said to be \emph{recognizable by $\Gamma$} if it satisfies:
\begin{enumerate}
\item For every $\gamma \in \Gamma$ with $\gamma(A) = A$ the homeomorphism $\tau$ given by 
\[\tau(x) =
\begin{cases}
\gamma(x) & x \in A,\\
x & \text{otherwise,}
\end{cases}\]
belongs to $\Gamma$.
\item For every $\gamma \in \Gamma$ with $\gamma(A) \cap A = \emptyset$ the involution $\alpha$ given by 
\[\alpha(x) =
\begin{cases}
\gamma(x) & x \in A,\\
\gamma^{-1}(x) & x \in \gamma(A), \\
x & \text{otherwise,}
\end{cases}\]
belongs to $\Gamma$.
\end{enumerate}
\end{definition}

We shall see later that in our setting of topological full groups, all clopen subsets of the unit space are recognizable. And whenever this is the case, it is possible to characterize when a regular set is closed (i.e.\ clopen) using the following Boolean algebra notion.

\begin{definition}
Let $(\Gamma,X)$ be a space-group pair, and let $A \in \mathcal{R}(X)$ be a regular open set. We say that $A$ is \emph{weakly clopen} if for every $\gamma \in \Gamma$ satisfying $\gamma\left(A \cap \gamma(A)\right) = A \cap \gamma(A)$, there exists an element $\rho \in \Gamma$ such that
\begin{enumerate}
\item $\rho(B) = \gamma(B)$ for each $B \in \mathcal{R}(X)$ with $B \subseteq A \cap \gamma(A)$,
\item $\rho(B) = B$ for each $B \in \mathcal{R}(X)$ with $B \subseteq \ \sim\left(A \cap \gamma(A)\right)$.
\end{enumerate}
\end{definition}

Note that the notion of being weakly clopen is formulated solely in terms of the action by~$\Gamma$ on the Boolean algebra $\mathcal{R}(X)$. And as the next result shows---under suitable hypotheses---being weakly clopen is the same as being clopen.

\begin{lemma}\label{clopenLemma}
Let $(\Gamma,X) \in \TopG$. Assume that every clopen subset of $X$ is recognizable by $\Gamma$, and that the $\Gamma$-orbit of each point contains at least $3$ points. Then a regular open set~$A \in \mathcal{R}(X)$ is clopen if and only if both $A$ and $\sim A$ are weakly clopen.
\end{lemma}
\begin{proof}
This is a special case of~\cite[Lemma~3.45]{Rub}, where the dense subset $R$ is taken to be all of $\mathcal{R}(X)$. The assumptions 3.V.1~(a), (b), (c) and 3.V.2~(a), (b) preceding~\cite[Lemma~3.45]{Rub} follow from those above. In particular, what Rubin calls ``recognizably clopen'' coincides with~(2) in Definition~\ref{recognizable}, and ``strongly recognizably clopen'' is slightly weaker than~(1) in Definition~\ref{recognizable} (together with~(2)).
\end{proof}

In order to invoke Stone duality for Boolean spaces we need to recover the generalized Boolean algebra of compact open sets. The previous lemma gives us the clopen sets, and from these we obtain the compact open ones as follows.

\begin{lemma}\label{lem:countCompact}
Let $X$ be a second countable Boolean space. Then $X$ is compact if and only if $\co(X)$ is countable.
\end{lemma}
\begin{proof}
If $X$ is compact, then $\co(X) = \ck(X)$, and any second countable space has countably many compact open subsets. 

Suppose $X$ is non-compact. Let $\{K_n \}_{n=1}^\infty$ be a countable basis for $X$ consisting of compact open sets. Now form the compact open sets $C_k = \cup_{n=1}^k K_n$. As $X$ is not compact, we must have~$C_k \neq X$ for each $k$. Also, $C_k \subseteq C_{k+1}$ and they cover $X$. By passing to a subsequence, if necessary, we may assume that $C_k \subsetneq C_{k+1}$ for each $k$. Finally, let~$D_k = C_{k+1} \setminus C_k$. Then the $D_k$'s are pairwise disjoint non-empty compact open sets. We claim that for each subset~$S$ of the natural numbers, the set $\cup_{k \in S} D_k$ is clopen. And then we have produced uncountably many distinct clopen sets. The claim follows from the fact that for each $C_m$, the intersection~$C_m \cap (\cup_{k \in S} D_k)$ is a finite intersection, hence closed, and that the $C_m$'s cover~$X$. 
\end{proof}

\begin{corollary}\label{compactOpenLemma}
Let $X$ be a second countable Boolean space, and let $A \in \co(X)$ be a clopen set. Then $A$ is compact if and only if the set $\{B \in \co(X) \ \vert \ B \subseteq A\}$ is countable.
\end{corollary}
\begin{proof}
The set $\{B \in \co(X) \ \vert \ B \subseteq A\}$ coincides with $\co(A)$ when viewing $A$ as a subspace of~$X$. The result now follows from Lemma~\ref{lem:countCompact}.
\end{proof}

This shows that in the generalized Boolean algebra $\co(X)$ compactness is characterized by having only countably many elements below. We are now ready to define the class $K^{LCC}$ and give the second spatial realization result of this section.

\begin{definition}\label{def_KB}
We define the class \emph{$K^{LCC}$} to consist of all space-group pairs $(\Gamma,X)$ in $\TopG$ which satisfy the following conditions:
\begin{enumerate}[label=(K\arabic*)]
\item $X$ is a locally compact Cantor space.
\item $(\Gamma,X)$ is locally moving.
\item $(\Gamma,X)$ is locally flexible.
\item Every clopen subset of $X$ is recognizable by $\Gamma$.
\item The $\Gamma$-orbit of each point contains at least $3$ points.
\end{enumerate}
\end{definition}

\begin{theorem}[c.f.\ {\cite[Theorem~3.50(a)]{Rub}}]\label{thm:KBfaithful}
The class $K^{LCC}$ is faithful.
\end{theorem}
\begin{proof}
Suppose we have two space-group pairs $(\Gamma_1,X_1)$, $(\Gamma_2,X_2) \in K^{LCC}$ and a group isomorphism $\Phi \colon \Gamma_1 \to \Gamma_2$. Invoking Theorem~\ref{reconstructRX} yields an isomorphism of Boolean algebras~${\psi \colon \mathcal{R}(X_1) \to \mathcal{R}(X_2)}$ such that $\psi(g(A)) = \Phi(g)(\psi(U))$ for each $A \in \mathcal{R}(X_1)$ and~$g \in \Gamma_1$. We first argue that $\psi(\co(X_1)) = \co(X_2)$, and then that $\psi(\ck(X_1)) = \ck(X_2)$.

First of all, note that both $\co(X_i)$ and $\ck(X_i)$ are invariant under $\Gamma_i$ ($i = 1,2$). Lemma~\ref{clopenLemma} characterizes clopenness of regular sets in $X_i$  solely in terms of the (induced) actions by $\Gamma_i$ on $\mathcal{R}(X_i)$. Since $\psi$ is an equivariant Boolean algebra isomorphism, it follows that $\psi(\co(X_1)) = \co(X_2)$. Next, Corollary~\ref{compactOpenLemma} characterizes compactness of a clopen set in terms of a countability condition in the generalized Boolean algebra $\co(X_i)$. Clearly, this is then also preserved by $\psi$. Consequently, $\psi$ restricts to an equivariant isomorphism of the generalized Boolean algebras $\ck(X_1)$ and $\ck(X_2)$.

By applying the Stone functor to the generalized Boolean algebra isomorphism \[\psi \colon \ck(X_1) \to \ck(X_2)\] we obtain a homeomorphism \[\mathbb{S}(\psi) \colon \mathbb{S}(\ck(X_2)) \to \mathbb{S}(\ck(X_1))\] of the spaces of ultrafilters. The induced actions by the groups $\Gamma_i$ on~$\mathbb{S}(\ck(X_i))$ is given by~${g \cdot \alpha = \{g(K) \ \vert \ K \in \alpha \}}$ for an ultrafilter $\alpha \in \mathbb{S}(\ck(X_i))$. Finally,  let $\phi \colon X_1 \to X_2$ be the homeomorphism given by the composition 
\[\begin{CD}
X_1 @>\Omega_{X_1}>> \mathbb{S}(\ck(X_1)) @>\mathbb{S}(\psi)^{-1}>> \mathbb{S}(\ck(X_2)) @>\Omega_{X_2}^{-1}>> X_2
\end{CD} \]
where $\Omega_{X_i}$ is the canonical homeomorphism mapping a point to its compact open neighbourhood ultrafilter. It is now easy to check that the original group isomorphism $\Phi$ is spatially implemented by $\phi$, i.e.\ that $\Phi(g) = \phi \circ g \circ \phi^{-1}$ for each $g \in \Gamma_1$. 
\end{proof}

\begin{remark}\label{rem:Medynets}
As mentioned in the introduction, Medynets has obtained a spatial realization result for full groups of group actions on the Cantor space~\cite{Med}. The arguments therein also apply to the topological full group, and could be adapted to the topological full group of the ample groupoids over locally compact Cantor spaces considered here. And then in turned be used to prove Theorem~\ref{intro:KLCCgroupoid} instead of using Theorem~\ref{thm:KBfaithful}. Medynets' starting point is a Boolean algebra reconstruction result of Fremlin~\cite[Theorem~384D]{Frem}. This result is very similar to Rubin's Boolean algebra reconstruction result; Theorem~\ref{reconstructRX}. Rubin requires the space-group pair to be locally moving and locally flexible, whereas Fremlin requires it to be locally moving in terms of involutions. Yet they both apply to the topological full group, since it is both (globally) flexible and has enough involutions to witness locally moving. Medynets then goes on to characterize the clopen sets among the regular open sets in an algebraic way and use this to show that the Boolean algebra isomorphism must preserve the Boolean subalgebra of clopen subsets and in turn give rise to a spatial isomorphism via Stone duality. This is exactly the same approach as we use here, via Rubin, but Medynets' characterization of the clopens~\cite[Lemma~2.5]{Med} looks (at least on the surface) a bit different from the one we give here in Lemma~\ref{clopenLemma}. Finally, we remark that Medynets' arguments does not seem to apply to the commutator subgroup either (c.f.\ Remark~\ref{rem:KLCCcommutator}).
\end{remark}

\section{Isomorphism theorems for ample groupoids}\label{sec:reggrpd}

In this section we shall apply the spatial realization results of the previous section to (subgroups of) the topological full group. As corollaries we are able to  reconstruct certain ample groupoids from their topological full group. The two faithful classes considered in the previous section allows us to lift an abstract group isomorphism of (subgroups of) the topological full groups to a spatial one. This in turn yields an isomorphism of the associated groupoids of germs (c.f.\ Corollary~\ref{corol_inj_groupoid}). In order to conclude that the groupoids themselves are isomorphic we need, by Proposition~\ref{prop_germs} and Corollary~\ref{cor:necessaryCover}, to assume that the subgroups in question cover the groupoids. As we saw in Lemma~\ref{ex_cover}, if every $\mathcal{G}$-orbit has length at least $2$, or respectively $3$, then $\llbracket \mathcal{G} \rrbracket$, or respectively any $\Gamma$ with $D(\llbracket \mathcal{G} \rrbracket) \leq \Gamma \leq \llbracket \mathcal{G} \rrbracket$, covers $\mathcal{G}$.

We first extract an isomorphism theorem from the faithfulness of the class $K^F$. For a general ample groupoid the only general condition we know to imply that $\smash{\left(\llbracket \mathcal{G} \rrbracket, \mathcal{G}^{(0)} \right)}$ belong to $K^F$ is minimality. So for general groupoids we obtain only a straightforward minor extension of~\cite[Theorem~3.9 \&~3.10]{Mat} in Theorem~\ref{thm:KFgroupoid} below. However, for the class of graph groupoids we will see in Section~\ref{sec:isogg} that we can weaken minimality quite a lot and still have the topological full group (and its commutator) in $K^F$, and thereby obtain a significantly more general result within the class of graph groupoids. It would therefore be interesting to find general conditions on a general ample groupoid $\GG$, weaker than minimality, ensuring that $\left(\llbracket \mathcal{G} \rrbracket,\GG^{(0)}\right)$ and $\left(\DD(\llbracket \mathcal{G} \rrbracket),\GG^{(0)}\right)$ belong to $K^F$.

\begin{proposition}[c.f.\ {\cite[Proposition~3.6]{Mat}}]\label{KFprop}
Let $\GG$ be an effective ample Hausdorff groupoid whose unit space has no isolated points. If $\GG$ is minimal and $\Gamma$ is any subgroup of~$\llbracket \mathcal{G} \rrbracket$ containing $\DD(\llbracket \mathcal{G} \rrbracket)$, then $\smash{\left(\Gamma, \mathcal{G}^{(0)} \right)} \in K^F$.
\end{proposition}
\begin{proof}
The proof of~\cite[Proposition~3.6]{Mat} goes through verbatim in this slightly more general setting. The proof makes heavy use of the minimality of $\GG$ and combine this with Lemma~\ref{lemma:bisectionInvolution} to find the desired elements in $\DD(\llbracket \mathcal{G} \rrbracket)$.
\end{proof}

\begin{theorem}\label{thm:KFgroupoid}
Let $\mathcal{G}_1, \mathcal{G}_2$ be effective ample minimal Hausdorff groupoids whose unit spaces have no isolated points. Suppose $\Gamma_1, \Gamma_2$ are subgroups with $\DD(\llbracket \mathcal{G}_i \rrbracket) \leq \Gamma_i \leq \llbracket \mathcal{G}_i \rrbracket$. If $\Gamma_1 \cong \Gamma_2$ as abstract groups, then $\mathcal{G}_1 \cong \mathcal{G}_2$ as topological groupoids. In particular, the following are equivalent:
\begin{enumerate}
\item $\mathcal{G}_1 \cong \mathcal{G}_2$ as topological groupoids.
\item $\llbracket \mathcal{G}_1 \rrbracket \cong \llbracket \mathcal{G}_2 \rrbracket$ as abstract groups.
\item $\DD(\llbracket \mathcal{G}_1 \rrbracket) \cong \DD(\llbracket \mathcal{G}_2 \rrbracket)$ as abstract groups.
\end{enumerate}
\end{theorem}
\begin{proof}
Clearly every $\mathcal{G}_i$-orbit is infinite, for $i = 1,2$. Thus the result follows from combining Proposition~\ref{KFprop}, Theorem~\ref{classF}, Proposition~\ref{faithGerm}, Lemma~\ref{ex_cover} and Proposition~\ref{prop_germs}.
\end{proof}

\begin{remark}
For transformation groupoids arising from minimal $\mathbb{Z}$-actions on locally compact Cantor spaces, a variant of this result appears in~\cite[Theorem~4.13~(vi)]{Mat4}. See also Remark~\ref{rem:matui}.
\end{remark}

\begin{remark}
In~\cite[Theorem~3.10]{Mat} the kernel of the so-called \emph{index map} also appears (as $\llbracket \mathcal{G} \rrbracket_0$). We could equally well have included it in Theorem~\ref{thm:KFgroupoid} since it is a distinguished subgroup lying between $\llbracket \mathcal{G} \rrbracket$ and $\DD(\llbracket \mathcal{G} \rrbracket)$.
\end{remark}

Our next goal is to analyze the conditions in the definition of the class $K^{LCC}$, when the space-group pair under consideration is the topological full group and the unit space of an ample groupoid. Unfortunately, the commutator subgroup $\DD(\llbracket \mathcal{G} \rrbracket)$ does not seem to belong to $K^{LCC}$, which is why we only consider $\llbracket \mathcal{G} \rrbracket$ itself (see Remark~\ref{rem:KLCCcommutator} below). We begin by showing that the groupoid-orbits coincide with the orbits of the action by the topological full group on the unit space.

\begin{lemma}\label{lemma:orbit}
Let $\mathcal{G}$ be an effective ample groupoid and let $x \in \mathcal{G}^{(0)}$. Then
\[\orb_\mathcal{G}(x) = \orb_{\llbracket \mathcal{G} \rrbracket \curvearrowright \mathcal{G}^{(0)}}(x).\]
\end{lemma}
\begin{proof}
From the definition of the topological full group it is obvious that the groupoid orbit~$\orb_\mathcal{G}(x)$ contains the orbit of the action $\orb_{\llbracket \mathcal{G} \rrbracket \curvearrowright \mathcal{G}^{(0)}}(x)$. For the reverse inclusion, suppose $y \in \orb_\mathcal{G}(x)$ is distinct from $x$, and let $\gamma \in \mathcal{G}$ be an arrow from $x$ to $y$. Applying Lemma~\ref{lemma:bisectionExistence} to $\gamma$ we obtain an element $\pi_U \in \llbracket \mathcal{G} \rrbracket$ with $\pi_U(x) = y$. Thus $y \in \orb_{\llbracket \mathcal{G} \rrbracket \curvearrowright \mathcal{G}^{(0)}}(x)$.
\end{proof} 

In other words, when the space group pair is $(\llbracket \mathcal{G} \rrbracket, \mathcal{G}^{(0)})$ condition (K5) of Definition~\ref{def_KB} is equivalent to saying that every $\mathcal{G}$-orbit has length at least $3$ (which, incidentally, implies that $\llbracket \mathcal{G} \rrbracket$ covers $\mathcal{G}$). Next we show that conditions (K3) and (K4) of Definition~\ref{def_KB} are always satisfied for topological full groups. In fact, $(\llbracket \mathcal{G} \rrbracket, \mathcal{G}^{(0)})$ is even ``globally flexible''.

\begin{lemma}
Let $\mathcal{G}$ be an effective ample groupoid. Then every open subset of $\mathcal{G}^{(0)}$ is flexible with respect to $\llbracket \mathcal{G} \rrbracket$. In particular, $\smash{\big(\llbracket \mathcal{G} \rrbracket, \mathcal{G}^{(0)}\big)}$ is locally flexible.
\end{lemma}
\begin{proof}
Let $A$ be a non-empty open subset of $\mathcal{G}^{(0)}$, and let $B_1, B_2$ be two open subsets of~$A$. We may assume that these are disjoint, for otherwise the identity homeomorphism trivially witnesses flexibility. Suppose $\pi_U \in \llbracket \mathcal{G} \rrbracket$ satisfies $\pi_U(B_1) \cap B_2 \neq \emptyset$. Then there is a~$g \in U$ with $s(g) \in B_1$ and $r(g) \in B_2$. Lemma~\ref{lemma:bisectionExistence} applied to $g$ and $B_1 \sqcup B_2$ produces  an element~$\pi_V \in \llbracket \mathcal{G} \rrbracket$ with $\supp(\pi_V) \subseteq B_1 \sqcup B_2 \subseteq A$ and  $\pi_V(B_1) \cap B_2 \neq \emptyset$. This shows that $A$ is flexible.
\end{proof}

\begin{lemma}
Let $\mathcal{G}$ be an effective ample groupoid. Then every clopen subset of $\mathcal{G}^{(0)}$ is recognizable by $\llbracket \mathcal{G} \rrbracket$.
\end{lemma}
\begin{proof}
Let $A \subseteq \mathcal{G}^{(0)}$ be clopen.

(1) Suppose $\pi_U \in \llbracket \mathcal{G} \rrbracket$ satisfies $\pi_U(A) = A$. Then $V = \smash{s_{\vert U}^{-1}(A)}  \subseteq U$ is a clopen bisection with $s(V) = r(V) = A$. Then $\tilde{V}$ as in Lemma~\ref{lem:extendBisection} is a full bisection with $\supp(\pi_{\tilde{V}}) \subseteq \supp(\pi_U)$, hence $\pi_{\tilde{V}} \in \llbracket \mathcal{G} \rrbracket$. The homeomorphism $\pi_{\tilde{V}}$ is the one from condition (1) of Definition~\ref{recognizable}.

(2) Suppose now that $\pi_U \in \llbracket \mathcal{G} \rrbracket$ satisfies $\pi_U(A) \cap A = \emptyset$. Again we set $V = \smash{s_{\vert U}^{-1}(A)}$. Then~$s(V) \cap r(V) = A \cap \pi_U(A) = \emptyset$. The full bisection $\hat{V}$ as in Lemma~\ref{lemma:bisectionInvolution} also has compact support since $\supp(\pi_{\hat{V}}) \subseteq \supp(\pi_U)$, and so $\pi_{\hat{V}} \in \llbracket \mathcal{G} \rrbracket$. The involution $\pi_{\hat{V}}$ is the one from condition (2) of Definition~\ref{recognizable}.
\end{proof}

It remains to consider condition~(K2) of Definition~\ref{def_KB}. Inspired by~\cite[Proposition~2.2]{Med} we introduce the the notion of a \emph{non-wandering} groupoid, in order to characterize when $\smash{\left(\llbracket \mathcal{G} \rrbracket, \mathcal{G}^{(0)}\right)}$ is locally moving in terms of the groupoid~$\mathcal{G}$.

\begin{definition}\label{def:wandering}
Let $\mathcal{G}$ be an ample groupoid. A subset $A \subseteq \mathcal{G}^{(0)}$ is called \emph{wandering} if~$\vert A \cap \orb_\mathcal{G}(x) \vert = 1$ for all $x \in A$. We say that $\mathcal{G}$ is \emph{non-wandering} if $\mathcal{G}^{(0)}$ has no non-empty clopen wandering subsets.
\end{definition}

In words, a non-wandering groupoid is one in which every clopen subset of the unit space meets some orbit at least twice. This may be viewed as a ``mixing condition'' which is far weaker than minimality. For if $\mathcal{G}$ is minimal, then in particular $\vert A \cap \orb_\mathcal{G}(x) \vert$ is infinite (from being dense) for each clopen neighbourhood $A$ of $x$.

\begin{proposition}\label{prop:LMChar}
Let $\mathcal{G}$ be an effective ample Hausdorff groupoid. Then the following are equivalent:
\begin{enumerate}
\item The space-group pair $\left(\llbracket \mathcal{G} \rrbracket, \mathcal{G}^{(0)}\right)$ is locally moving.
\item The groupoid $\mathcal{G}$ is non-wandering.
\end{enumerate}
\end{proposition}
\begin{proof}
Let $A$ be a non-empty clopen subset of $\mathcal{G}^{(0)}$. We will prove that $A$ meets some~\mbox{$\mathcal{G}$-orbit} twice (i.e.\ $A$ is not wandering) if and only if there is an element $\pi_U \in \llbracket \mathcal{G} \rrbracket \setminus \{1\}$ with $\supp(\pi_U) \subseteq A$. If $\emptyset \neq \supp(\pi_U) \subseteq A$, then, since both sets are clopen, there is an $x \in A$ with $x \neq \pi_U(x) \in A$. In other words, $\vert A \cap \orb_\mathcal{G}(x) \vert \geq 2$. Conversely, if $\vert A \cap \orb_\mathcal{G}(x) \vert \geq 2$ holds for some $x \in A$, then there is a  $g \in \mathcal{G} \setminus \mathcal{G}'$ such that $s(g)$ and $r(g)$ both belong to~$A$. Now Lemma~\ref{lemma:bisectionExistence} gives us a nontrivial group element in $\llbracket \mathcal{G} \rrbracket$ supported on $A$. As the clopens form a base for the topology on $\mathcal{G}^{(0)}$ we are done.
\end{proof}

Putting it all together, we arrive at the second main result of this section.

\begin{theorem}\label{thm:KLCCgroupoid}
Let $\mathcal{G}_1, \mathcal{G}_2$ be effective ample Hausdorff groupoids over locally compact Cantor spaces. Suppose that, for $i = 1,2$, $\mathcal{G}_i$ is non-wandering and that each $\mathcal{G}_i$-orbit has length at least $3$. Then any isomorphism between $\llbracket \mathcal{G}_1 \rrbracket$ and $\llbracket \mathcal{G}_2 \rrbracket$ is spatial. In particular, the following are equivalent:
\begin{enumerate}
\item $\mathcal{G}_1 \cong \mathcal{G}_2$ as topological groupoids.
\item $\llbracket \mathcal{G}_1 \rrbracket \cong \llbracket \mathcal{G}_2 \rrbracket$ as abstract groups.
\end{enumerate}
\end{theorem}

\begin{remark}\label{rem:KLCCcommutator}
It would be desirable to also obtain a spatial realization result for the commutator subgroup $\DD(\llbracket \mathcal{G} \rrbracket)$ in terms of the class $K^{LCC}$. Unfortunately we were not able to show that $\DD(\llbracket \mathcal{G} \rrbracket)$ satisfies condition (K4). This is also the reason why the arguments of~\cite{Med} do not apply to the commutator subgroup either. However, it might be that Theorem~\ref{thm:KLCCgroupoid} holds for the commutator subgroups as well.
\end{remark}

As mentioned above, non-wandering is a much weaker ``mixing property'' than minimality. Below we include two other ``mixing properties'' that lie between non-wandering and minimality.

\begin{definition}[c.f.\ {\cite[page~8]{Nek}}]
An ample groupoid $\GG$ is called \emph{locally minimal} if there exists a basis for $\mathcal{G}^{(0)}$ consisting of clopen sets $A$ such that $\GG_A$ is minimal.
\end{definition}

\begin{definition}
An ample groupoid $\GG$ is called \emph{densely minimal} if for every non-empty open subset $A$ of $\GG^{(0)}$ there exists a non-empty clopen subset $B \subseteq A$ such that $\GG_B$ is minimal.
\end{definition}

We clearly have the following implications for an ample groupoid:
\[\text{minimal} \implies \text{locally minimal} \implies \text{densely minimal} \implies \text{non-wandering}.\]
We will give examples of densely minimal groupoids which are not minimal in the next section (Examples~\ref{ex:1orbitCover} and~\ref{ex:2orbitCover}), as well as non-wandering groupoids which are not densely minimal (Remark~\ref{rem:notDM}).

\section{Graph groupoids}\label{sec:gg}

The rest of the paper will be focused on graph groupoids. This section recalls the relevant terminology for graphs and their associated groupoids (as they appear in the literature on graph algebras). We also record the characterizations of many properties of a graph groupoid in terms of the graph. This is fairly standard and may also be found in many other papers, e.g.~\cite{BCW}, \cite{KPRR}.

\subsection{Graph terminology}
By a \emph{graph} we shall always mean a directed graph, that is, a quadruple ${E=(E^0,E^1,r,s)}$, where $E^0, E^1$ are (non-empty) sets and $r,s \colon E^1 \to E^0$ are maps. The elements in $E^0$ and $E^1$ are called \emph{vertices} and \emph{edges}, respectively, while the maps $r$ and $s$ are called the \emph{range} and \emph{source} map\footnote{Although the notation collides with the range and source maps in a groupoid, both conventions are well established. In the sequel it will always be clear from context whether we mean the source/range of an edge in a graph or of an element in a groupoid.}, respectively. We say that $E$ is \emph{finite} if~$E^0$ and~$E^1$ both are finite sets, and similarly that $E$ is \emph{countable} if $E^0$ and $E^1$ are countable.

A \emph{path} in a graph $E$ is a sequence of edges $\mu=e_1 e_2 \ldots e_n$ such that $r(e_i)=s(e_{i+1})$ for~$1\leq i\leq n-1$. The \emph{length} of $\mu$ is $\left| \mu \right| \coloneqq n$. The set of paths of length $n$ is denoted~$E^n$. The vertices, $E^0$, are considered trivial paths of length $0$. The set of all finite paths is denoted $E^* \coloneqq \bigcup_{n=0}^\infty E^n$. The range and source maps extend to $E^*$ by setting~$r(\mu) \coloneqq r(e_n)$ and~$s(\mu) \coloneqq s(e_1)$. For $v \in E^0$, we set~$s(v) = r(v) = v$. Given another path $\lambda = f_1 \ldots f_m$ with $s(\lambda) = r(\mu)$ we denote the concatenated path $e_1 \ldots e_n f_1 \ldots f_m$ by $\mu \lambda$. In particular, we set~$s(\mu) \mu = \mu = \mu \ r(\mu)$ for each $\mu \in E^*$. Given two paths $\mu, \mu' \in E^*$ we write $\mu < \mu'$ if there exists a path $\lambda$ with $\left| \lambda \right| \geq 1$ such that $\mu' = \mu \lambda$. Writing $\mu \leq \mu'$ allows for $\mu = \mu'$. We say that $\mu$ and $\mu'$ are \emph{disjoint} if $\mu \nleq \mu'$ and $\mu' \nleq \mu$, i.e.\ neither is a subpath of the other.

A \emph{cycle} is a nontrivial path $\mu$ (i.e.\ $\left| \mu \right| \geq 1$) with $r(\mu) = s(\mu)$, and we say that $\mu$ is \emph{based} at $s(\mu)$. We also say that the vertex $s(\mu)$ \emph{supports} the cycle $\mu$. By a \emph{loop} we mean a cycle of length $1$. Beware that some authors use the term loop to denote what we here call cycles. When $\mu$ is a cycle and $k \in \mathbb{N}$, $\mu^k$ denotes the cycle $\mu \mu \ldots \mu$, where $\mu$ is repeated $k$ times. A cycle $\mu = e_1 \ldots e_n$ is called a \emph{return path} if $r(e_i) \neq r(\mu)$ for all $i < n$. This simply means that $\mu$ does not pass through $s(\mu)$ multiple times. An \emph{exit} for a path $\mu = e_1 \ldots e_n$ is an edge~$e$ such that $s(e) = s(e_i)$ and $e \neq e_i$ for some $1 \leq i \leq n$.

For $v,w\in E^0$ we set $vE^n \coloneqq \{ \mu \in E^n \mid s(\mu) = v\}$, $E^n w \coloneqq \{ \mu \in E^n \mid r(\mu) = w\}$ and~$vE^n w \coloneqq vE^n \cap E^n w.$ A vertex $v \in E^0$ is called a \emph{sink} if $vE^1 = \emptyset$, and a \emph{source} if $E^1 v = \emptyset$. Further, $v$ is called an \emph{infinite emitter} if $vE^1$ is an infinite set. The set of \emph{regular} vertices is $\smash{E^0_{\text{reg}} \coloneqq \{ v \in E^0 \mid 0 < | vE^1 | < \infty \}}$, and the set of \emph{singular} vertices is~$\smash{E^0_{\text{sing}} \coloneqq E^0\setminus E^0_{\text{reg}}}$. In other words, sinks and infinite emitters are singular vertices, while all other vertices are regular. We equip the vertex set $E^0$ with a preorder $\geq$ by definining~$v \geq w$ iff $vE^* w \neq \emptyset$, i.e.\ there is a path from $v$ to $w$. The graph $E$ is called \emph{strongly connected} if for each pair of vertices $v,w \in E^0$ we have $v \geq w$.

To close this subsection we describe three exit conditions on graphs that appear frequently in the graph algebra literature. They will play a central role in what follows. A graph $E$ is said to satisfy \emph{Condition~(L)} if every cycle in $E$ has an exit. The graph $E$ satisfies \emph{Condition~(K)} if for every vertex $v \in E^0$, either there is no return path based at $v$ or there are at least two distinct return paths based at $v$. We say that $E$ satisfies \emph{Condition~(I)} if for every vertex $v \in E^0$, there exists a vertex $w \in E^0$ supporting at least two distinct return paths and $v \geq w$. These conditions first appeared in~\cite{KPR},~\cite{KPRR} and~\cite{CK}, respectively. In general, Condition~(K) and~(I) both imply~(L), while~(K) and~(I) are not comparable. For graphs with finitely many vertices and no sinks, Condition~(I) is equivalent to Condition~(L).

\subsection{The boundary path space}
An \emph{infinite path} in a graph $E$ is an infinite sequence of edges $x = e_1 e_2 e_3 \ldots$ such that $r(e_i)=s(e_{i+1})$ for all $i \in \mathbb{N}$. We define~${s(x) \coloneqq s(e_1)}$ and~${\left| x \right| \coloneqq \infty}$. The set of all infinite paths in $E$ is denoted $E^\infty$. Given a finite path~${\mu = f_1 \ldots f_n}$ and an infinite path $x = e_1 e_2 e_3 \ldots \in E^\infty$ such that $r(\mu) = s(x)$ we denote the infinite path~$f_1 \ldots f_n e_1 e_2 e_3 \ldots$ by $\mu x$. For natural numbers $m < n$, we set $x_{[m,n]} \coloneqq e_m e_{m+1} \ldots e_n$, and we denote the infinite path $e_m e_{m+1} e_{m+2} \ldots$ by $x_{[m,\infty)}$. Given a cycle $\lambda \in E^*$ we denote the infinite path $\lambda \lambda \lambda \ldots$ by $\lambda^\infty$. An infinite path of the form $\mu \lambda^\infty$, where $\lambda$ is a cycle with~$s(\lambda) = r(\mu)$, is called \emph{eventually periodic}. An infinite path $e_1 e_2 \ldots \in E^\infty$ is \emph{wandering} if the set $\{ i \in \mathbb{N} \mid s(e_i) = v \}$ is finite for each $v \in E^0$. Note that there are no wandering infinite paths in a graph with finitely many vertices. We call a wandering infinite path~$e_1 e_2 \ldots \in E^\infty$ a \emph{semi-tail}\footnote{By comparison, a \emph{tail} is a wandering path with $s(e_i)E^1 = \{e_i\} = E^1r(e_i)$ for all $i$, c.f.~\cite{BPRS}.} if $s(e_i)E^1 = \{e_i\}$ for each~{$i \in \mathbb{N}$}. The graph $E$ is called \emph{cofinal} if for every vertex $v \in E^0$ and for every infinite path $e_1 e_2 \ldots \in E^\infty$, there exists~$n \in \mathbb{N}$ such that $v \geq s(e_n)$. 

The \emph{boundary path space} of $E$ is
\[\smash{\partial E \coloneqq E^\infty \cup \{\mu\in E^* \mid r(\mu)\in E^0_{\text{sing}}\}},\]
whose topology will be specified shortly. Note that if $v \in E^0$ is a singular vertex, then~$v$ belongs to~$\partial E$. For any vertex $v \in E^0$ we define~$v \partial E \coloneqq \{x \in \partial E \mid s(x) = v \}$ and similarly~$v E^\infty \coloneqq \{x \in E^\infty \mid s(x) = v \}$. The \emph{cylinder set} of a finite path $\mu \in E^*$ is~$Z(\mu) \coloneqq \{ \mu x \mid x \in r(\mu) \partial E \}$. Given a finite subset $F \subseteq r(\mu)E^1$, we  define the ``punctured'' cylinder set~$Z(\mu \setminus F) \coloneqq Z(\mu) \setminus \left( \bigcup_{e\in F} Z(\mu e) \right)$. Note that two finite paths are disjoint if and only if their cylinder sets are disjoint sets. A basis for the topology on the boundary path space $\partial E$ is given by $\left\{Z(\mu \setminus F)  \mid \mu \in E^*, F \subseteq_{\text{finite}} r(\mu)E^1 \right\}$, c.f.~\cite{Web}. Each basic set~$Z(\mu \setminus F)$ is compact open and these separate points, so $\partial E$ is a Boolean space. Moreover, each open set in $\partial E$ is a disjoint union of basic sets $Z(\mu \setminus F)$ (\cite[Lemma~2.1]{BCW}).  The boundary path space $\partial E$ is second countable exactly when $E$ is countable, and it is compact if and only if $E^0$ is finite. When it comes to (topologically) isolated points, these are classified as follows.

\begin{proposition}[{\cite[Proposition~3.1]{CW}}]\label{prop:DEperfect}
Let $E$ be a graph.
\begin{enumerate}
\item If $v \in E^0$ is a sink, then any finite path $\mu \in E^*$ with $r(\mu) = v$ is an isolated point in $\partial E$.
\item If $x = \mu \lambda^\infty \in E^\infty$ is eventually periodic, then $x$ is an isolated point if and only if the cycle $\lambda$ has no exit.
\item If $x = e_1 e_2 \ldots \in E^\infty$ is wandering, then $x$ is an isolated point if and only if for some $n \in \mathbb{N}$, $e_n e_{n+1} \ldots$ is a semi-tail.
\end{enumerate}
These are the only isolated points in $\partial E$.
\end{proposition}

For each $n \in \mathbb{N}$ we set $\partial E^{\geq n} \coloneqq \{ x \in \partial E  \mid \left| x \right| \geq n \}$ and $\partial E^{n} \coloneqq \{ x \in \partial E  \mid \left| x \right| = n \}$. Each of the sets $\partial E^{\geq n}$ is an open subset of $\partial E$. The \emph{shift map} on $E$ is the map $\sigma_E \colon \partial E^{\geq 1} \to \partial E$ given by $\sigma_E(e_1 e_2 e_3 \ldots) = e_2 e_3 e_4 \ldots$ for $e_1 e_2 e_3 \ldots \in \partial E^{\geq 2}$ and $\osh(e) = r(e)$ for $e \in \partial E^{1}$. In other words, $\sigma_E(x) = x_{[2, \infty)}$. We have that
\[ \sigma_E \left( \partial E^{\geq 1} \right) = \{ x \in \partial E \mid E^1 s(x) \neq \emptyset \} = \partial E \setminus \left( \cup_{E^1 v \neq \emptyset} Z(v) \right), \]
which is an open set, and we see that $\osh$ is surjective if and only if $E$ has no sources. We let $\osh^n \colon \partial E^{\geq n}\to\partial E$ be the $n$-fold composition of $\osh[E]$ with itself, and we set $\osh^0 = \id_{\OSS}$. Each $\osh^n$ is then a local homeomorphism between open subsets of $\partial E$. Note that an infinite path $x \in E^\infty$ is eventually periodic if and only if there are distinct numbers $m,n \in \mathbb{N}_0$ such that $\osh^m(x) = \osh^n(x)$.

\subsection{Graph groupoids and their properties}
The \emph{graph groupoid} of a graph $E$ is the (generalized) \mbox{Renault-Deaconu} groupoid~(\cite{Dea}, \cite{Ren3}) of the dynamical system $(\partial E, \osh)$, that is  
\[\mathcal{G}_E \coloneqq \{(x,m-n,y) \mid m,n \in \mathbb{N}_0, x \in \partial E^{\geq m}, y \in \partial E^{\geq n}, \osh^m(x) = \osh^n(y)    \} \]
as a set. The groupoid structure is given by $(x,k,y) \cdot (y,l,z) \coloneqq (x,k+l,z)$ (and undefined otherwise), $(x,k,y)^{-1} \coloneqq (y,-k,x)$. The unit space is $\mathcal{G}^{(0)}_E = \{ (x,0,x) \mid x \in \partial E \}$, which we will identify with $\partial E$ via $(x,0,x) \leftrightarrow x$. Then $s(x,k,y) = y$ and $r(x,k,y) = x$. We equip~$\mathcal{G}_E$ with the topology generated by the basic sets 
\[Z(U,m,n,V) \coloneqq \{(x,m-n,y) \mid x\in U, y\in V, \osh[E]^m(x)=\osh[E]^n(y)\},\]
where $U \subseteq \partial E^{\geq m}$ and $V \subseteq \partial E^{\geq n}$ are open sets such that $\left(\sigma_E^m \right)_{| U}$ and $\left(\sigma_E^n \right)_{| V}$ are injective, and $\osh[E]^m(U)=\osh[E]^n(V)$. This makes $\mathcal{G}_E$ an étale groupoid, and the identification of the unit space with $\partial E$ is compatible with the topology on $\partial E$. Note however, that this topology on~$\mathcal{G}_E$ is finer than the relative topology induced from $\partial E \times \mathbb{Z} \times \partial E$. According to~\cite[page~394]{BCW} the family
\begin{equation}\label{eq:basis}
\left \{ Z(U, \left| \mu \right|,\left| \lambda \right|, V) \mid \osh^{\left| \mu \right|}(U) = \osh^{\left| \lambda \right|}(V) \right \},
\end{equation}
parametrized over all  $\mu, \lambda \in E^*$ with $r(\mu) = r(\lambda)$, $U \subseteq Z(\mu)$ compact open, $V \subseteq Z(\lambda)$ compact open, is also a basis for the same topology. Each set $Z(U, \left| \mu \right|,\left| \lambda \right|, V)$ is a compact open bisection, and they separate the elements of $\mathcal{G}_E$, so $\mathcal{G}_E$ is an ample Hausdorff groupoid. The family in \eqref{eq:basis} is countable precisely when $E$ is countable, and so the graph groupoid $\mathcal{G}_E$ is second countable exactly when $E$ is countable.

For a boundary path $x \in \partial E$, the isotropy group of $(x,0,x) \in \mathcal{G}_E^{(0)}$ is nontrivial if and only if $x$ is eventually periodic (and infinite). For graph groupoids, effectiveness coincides with topological principality (even without assuming second countability), which in turn is well-known to coincide with the graph satisfying Condition~(L).

\begin{proposition}[c.f.\ {\cite[Proposition~2.3]{BCW}}]
Let $E$ be a graph. The following are equivalent:
\begin{enumerate}
\item The groupoid $\mathcal{G}_E$ is effective.
\item The groupoid $\mathcal{G}_E$ is topologically principal.
\item The set of infinite paths which are not eventually periodic form a dense subset of the boundary path space $\partial E$.
\item The graph $E$ satisfies Condition (L).
\end{enumerate}
\end{proposition}
\begin{proof}
The equivalence of (2), (3) and (4) is proved (for countable graphs) in~\cite[Proposition~2.3]{BCW}. The proof does not rely on the countability of the graph. As it is always the case that (2) implies (1) (c.f.\ Remark~\ref{rem:effective}), we only have to show that (1) implies (4). To that end, assume that $E$ does not satisfy Condition~(L). Then there is a cycle $\lambda \in E^*$ with no exit, and $\lambda^\infty$ is an isolated point in $\partial E$. But then the bisection \[Z\left(Z\left(\lambda^2\right), \vert \lambda \vert^2, \vert \lambda \vert, Z(\lambda)\right) = \{ \left(\lambda^\infty, \vert \lambda \vert, \lambda^\infty \right) \}\] is an open subset of $\mathcal{G}_E \setminus \mathcal{G}_E^{(0)}$, and hence $\mathcal{G}_E$ is not effective.
\end{proof}

We end this subsection by giving a characterization of minimality for graph groupoids. Let $E$ be a graph. Two infinite paths $x, y \in E^\infty$ are called \emph{tail equivalent} if there are natural numbers $k,l$ such that $x_{[k, \infty)} = y_{[l, \infty)}$. Similarly, two finite paths $\mu, \lambda \in E^*$ are \emph{tail equivalent} if $r(\mu) = r(\lambda)$. From the definition of $\mathcal{G}_E$ one sees that two boundary paths belong to the same $\mathcal{G}_E$-orbit if and only if they are tail equivalent. By combining~\cite[Theorem~5.1]{BCFS} with~\cite[Corollary~2.15]{DT} we arrive at the following result---of which we provide a self-contained proof.

\begin{proposition}
Let $E$ be a graph. Then the following are equivalent:
\begin{enumerate}
\item The groupoid $\mathcal{G}_E$ is minimal.
\item  The graph $E$ is cofinal, and for each $v \in E^0$ and $\smash{w \in E^0_\text{sing}}$, we have $v \geq w$.
\end{enumerate}
\end{proposition}
\begin{proof}
If $E$ has a sink $\smash{w \in E^0_{\text{sing}}}$, then one immediately deduces from both statements that~$E$ cannot have any other singular vertices, nor any infinite paths. Consequently \[\partial E = \orb_{\mathcal{G}_E}(w) = \{ \mu \in E^* \mid r(\mu) = w \},\]
and this entails that $\mathcal{G}_E$ is a discrete transitive groupoid. Now, \emph{(1)} and \emph{(2)} are clearly equivalent in this case.

For the remainder of the proof we assume that $E$ has no sinks. Assume that \emph{(2)} holds. Let~$x \in E^\infty$ and let $\lambda \in E^*$. By cofinality, there is a path $\lambda'$ from $r(\lambda)$ to $s(x_n)$ for some~$n \in \mathbb{N}$. The infinite path $\lambda \lambda' x_n x_{n+1} \ldots$ then belongs to both $Z(\lambda)$ and $\smash{\orb_{\mathcal{G}_E}(x)}$. Hence the latter is dense in $\partial E$ (since every open set contains a cylinder set when there are no sinks). Next, suppose $\mu \in \partial E \cap E^*$ with $r(\mu)$ an infinite emitter. By assumption there is a path $\lambda''$ from $r(\lambda)$ to $r(\mu)$, and then $\lambda \lambda'' \in Z(\lambda) \cap \orb_{\mathcal{G}_E}(\mu)$. This shows that $\mathcal{G}_E$ is minimal.

Assume now that $\mathcal{G}_E$ is minimal. To see that $E$ is cofinal, let $x \in E^\infty$ and $v \in E^0$ be given. By minimality there is a $y \in E^\infty$ tail equivalent to $x$ such that $y \in Z(v)$. This implies that~$v$ can reach $x$. As for the second part of \emph{(2)}, let $v \in E^0$ and $w \in \smash{E^0_{\text{sing}}}$ be given. Again by minimality there is a $\lambda \in E^* \cap Z(v)$ tail equivalent to $w$, but this is just a path from $v$ to $w$, so $v \geq w$.
\end{proof}

\begin{remark}
The notion of cofinality is slightly weaker than strong connectedness. But for finite graphs with no sinks and no sources, cofinality coincides with strong connectedness. In fact, this is also true for infinite graphs which additionaly have no \emph{semi-heads} (the direction-reversed notion of a semi-tail). We also remark that for cofinal graphs, Condition~(L) is equivalent to Condition~(K).
\end{remark}

\section{Topological full groups of graph groupoids}\label{sec:tfgg}
 
We are now going to describe the elements in the topological full group of a graph groupoid. Some examples will be given at the end of the section. We begin by specifying yet another (equivalent) basis for $\mathcal{G}_E$, which in turn will allow us to describe  bisections combinatorially in terms of the graph.
 
 For two finite paths $\mu, \lambda \in E^*$ with $r(\mu) = r(\lambda) = v$ we define
\[Z(\mu, \lambda) \coloneqq Z\left(Z(\mu),\left| \mu \right|,\left| \lambda \right|,Z(\lambda) \right).\]
More generally, given a finite subset $F \subseteq v E^1$ as well, we define
\[Z(\mu, F, \lambda) \coloneqq Z\left(Z(\mu \setminus F),\left| \mu \right|,\left| \lambda \right|,Z(\lambda \setminus F) \right).\]
Each $Z(\mu, F, \lambda)$ is a compact open bisection in $\mathcal{G}_E$, and we will see shortly that they also form a basis. Observe that if $\smash{v \in E^0_{\text{reg}}}$, then $Z(\mu,F,\lambda)=\bigsqcup_{e \in v E^1\setminus F}Z(\mu e,\lambda e)$, and that this is a finite union.

 \begin{lemma}\label{lem_disj}
 Let $E$ be a graph. Let $\mu, \mu',\lambda,\lambda' \in E^*$ with $r(\mu)=r(\lambda)=v$, $r(\mu')=r(\lambda')=v'$ and let $F \subseteq_{\text{finite}} v E^1$, $F' \subseteq_{\text{finite}} v'E^1$. Then $Z(\mu,F,\lambda) \bigcap Z(\mu',F',\lambda')$ equals either
 \begin{enumerate}
 \item $\emptyset$, or
 \item $Z(\mu,F,\lambda)$, or
 \item $Z(\mu',F',\lambda')$, or
 \item $Z(\mu,F\cup F',\lambda)$, in which case $\mu=\mu'$, $\lambda=\lambda'$ and \[Z(\mu,F,\lambda) \cup Z(\mu',F',\lambda') = Z(\mu,F \cap F',\lambda).\]
 \end{enumerate}  
 \end{lemma} 
 \begin{proof}
Suppose $Z(\mu,F,\lambda)\cap Z(\mu',F',\lambda') \neq \emptyset$. Then we must have $\left| \mu \right| - \left| \lambda \right| = \left| \mu' \right| - \left| \lambda' \right|$, $Z(\mu \setminus F)\cap Z(\mu' \setminus F') \neq \emptyset$ and $Z(\lambda \setminus F)\cap Z(\lambda' \setminus F') \neq \emptyset$. Since
 \[ Z(\mu \setminus F) \bigcap Z(\mu' \setminus F')= \left\lbrace \begin{array}{ll}
Z(\mu\setminus (F\cup F')) & \text{if }\mu=\mu', \\
Z(\mu \setminus F) & \text{if }\mu' < \mu \text{ and } \mu_{ \left| \mu' \right| +1} \notin F', \\
Z(\mu' \setminus F') & \text{if }\mu <\mu' \text{ and } \mu'_{ \left| \mu \right| +1} \notin F, \\
\emptyset & \text{otherwise},\end{array} \right.  \]
we may suppose without loss of generality that $\mu \leq \mu'$. The equality $\left| \mu \right| - \left| \lambda \right| = \left| \mu' \right| - \left| \lambda' \right|$ then forces $\lambda \leq \lambda'$ as well. If $\mu = \mu'$, then we must also have $\lambda = \lambda'$ and it is easy to see that \emph{(4)} holds in this case.

Next, suppose $\mu < \mu'$, which forces $\lambda < \lambda'$. As the intersections above are non-empty we have $Z(\mu' \setminus F') \subseteq Z(\mu \setminus F)$ and $Z(\lambda' \setminus F') \subseteq Z(\lambda \setminus F)$. It follows from this that~$Z(\mu',F',\lambda') \subseteq Z(\mu,F,\lambda)$, and we are done.
 \end{proof}
 
 \begin{lemma}
The family \[\left\{ Z(\mu, F, \lambda) \mid \mu, \lambda \in E^*, r(\mu) = r(\lambda), F \subseteq_{\text{finite}} r(\mu)E^1 \right\}\] forms a basis for the topology on $\mathcal{G}_E$.
\end{lemma}
\begin{proof}
It suffices to write each basic set $Z(U, \left| \mu \right|,\left| \lambda \right|, V)$, where $\mu, \lambda \in E^*$ with $r(\mu) = r(\lambda)$, $U \subseteq Z(\mu)$ compact open, $V \subseteq Z(\lambda)$ compact open and $\osh^{\left| \mu \right|}(U) = \osh^{\left| \lambda \right|}(V)$, as a union of $Z(\mu', F', \lambda')$'s. Given such a basic set $Z(U, \left| \mu \right|,\left| \lambda \right|, V)$, we can then write 
\[ \osh^{\left| \mu \right|}(U) = \osh^{\left| \lambda \right|}(V) = \bigsqcup_{i=1}^k Z(\eta_i \setminus F_i),  \]
for some $\eta_i \in E^*$, $F_i \subseteq_{\text{finite}} r(\eta_i)E^1$, since the former two are compact open subsets of $\partial E$. It follows that 
\[U = \bigsqcup_{i=1}^k Z(\mu \eta_i \setminus F_i) \quad \text{and} \quad V = \bigsqcup_{i=1}^k Z( \lambda \eta_i \setminus F_i). \]
Hence 
\[Z(U, \left| \mu \right|,\left| \lambda \right|, V) = \bigsqcup_{i=1}^k Z(\mu \eta_i, F_i, \lambda \eta_i). \]
\end{proof}
 
Using the basis above, we may concretely describe the bisections in $\mathcal{G}_E$ as follows.
 
 \begin{lemma}\label{lem_bis}
 Let $E$ be graph, and let $U \subseteq \GG_E$ be a compact open bisection with $s(U) = r(U)$. Then $U$ is of the form 
\[U = \bigsqcup_{i=1}^k Z(\mu_i, F_i, \lambda_i), \]
where $\mu_i, \lambda_i \in E^*$ with $r(\mu_i) = r(\lambda_i)$, $F_i \subseteq_{\text{finite}} r(\mu_i) E^1$ and 
\[s(U) =\bigsqcup_{i=1}^k Z(\lambda_i \setminus F_i) = \bigsqcup_{i=1}^k Z(\mu_i \setminus F_i). \]
 \end{lemma} 
\begin{proof}
Since $U$ is a compact open subset of $\mathcal{G}_E$ we may, by the preceding two lemmas, write~$U$ as a finite disjoint union of basic sets $Z(\mu, F, \lambda)$'s, say $U = \bigsqcup_{i=1}^k Z(\mu_i, F_i, \lambda_i)$. As~$r$ and~$s$ are injective on $U$ they preserve disjoint unions, so we have
\begin{align*}
&s(U) = s\left(\bigsqcup_{i=1}^k Z(\mu_i, F_i, \lambda_i) \right) = \bigsqcup_{i=1}^k s\left(Z(\mu_i, F_i, \lambda_i)\right) = \bigsqcup_{i=1}^k Z(\lambda_i \setminus F_i) \\
= \ &r(U) = r\left(\bigsqcup_{i=1}^k Z(\mu_i, F_i, \lambda_i) \right) = \bigsqcup_{i=1}^k r\left(Z(\mu_i, F_i, \lambda_i)\right) = \bigsqcup_{i=1}^k Z(\mu_i \setminus F_i).
\end{align*}
 \end{proof}  
 
In conjunction with Lemma~\ref{lem:extendBisection} we get that the elements in $\llbracket \mathcal{G}_E \rrbracket$ for an effective graph groupoid (i.e.\ the graph $E$ satisfying Condition (L)) may be described as follows, in terms of~$E$.

\begin{proposition}\label{prop:bis}
Let $E$ be a graph satisfying Condition (L). If $\pi_U \in \llbracket \mathcal{G}_E \rrbracket$, then the full bisection $U$ can be written as
\[U = \left( \bigsqcup_{i=1}^k Z(\mu_i, F_i, \lambda_i) \right) \bigsqcup \left( \partial E \setminus \supp(\pi_U) \right),  \]
where $\mu_i, \lambda_i \in E^*$ with $r(\mu_i) = r(\lambda_i)$, $F_i \subsetneq_{\text{finite}} r(\mu_i) E^1$ and 
\[ \supp(\pi_U) = \bigsqcup_{i=1}^k Z(\lambda_i \setminus F_i) = \bigsqcup_{i=1}^k Z(\mu_i \setminus F_i). \]
Moreover, $\mu_1, \ldots \mu_k$ are pairwise disjoint, $\lambda_1, \ldots \lambda_k$ are pairwise disjoint, and $\mu_i \neq \lambda_i$ for each $i$.
The homeomorphism $\pi_U \colon \partial E \to \partial E$ is given by $x = \lambda_i z \longmapsto \mu_i z$ for $x \in Z(\lambda_i \setminus F_i)$ and $x \longmapsto x$ otherwise.
\end{proposition}

\begin{remark}
The elements in $\llbracket \mathcal{G}_E \rrbracket$ may alternatively be described in more dynamical terms via the orbits by the shift map. From~\cite[Proposition~3.3]{BCW} one deduces that a homeomorphism $\alpha \in \homeo(\partial E)$ belongs to $\llbracket \mathcal{G}_E \rrbracket$ if and only if there are compactly supported continuous functions $m,n \colon \partial E \to \mathbb{N}_0$ such that $\osh^{m(x)}(\alpha(x)) = \osh^{n(x)}(x)$. This parallels Matui's definition for locally compact Cantor minimal systems mentioned in Remark~\ref{rem:matui}, and Matsumoto's definition for one-sided shifts of finite type in~\cite{Mats}. 
\end{remark}

Having completely described the topological full group of a graph groupoid, we provide an example to show that the assumption on the orbits in Lemma~\ref{ex_cover} is not a necessary condition. On the other hand, we also give an example to show that the statement is generally false without said assumption. These examples also provide examples of densely minimal groupoids which are not minimal.

\begin{example}\label{ex:1orbitCover}
Consider the following graph:
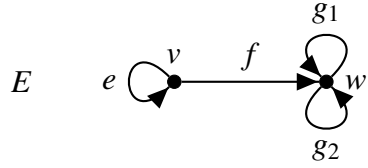
\[ \begin{tikzpicture}[vertex/.style={circle, draw = black, fill = black, inner sep=0pt,minimum size=5pt}]

\node at (-2,0) {$E$};
\node[vertex] (a) at (0,0) [label=above:$v$] {};
\node[vertex] (b) at (2,0) [label=right:$w$] {};

\path (a) 	edge[thick, loop, min distance = 10mm, looseness = 10, out = 135, in = 225,decoration={markings, mark=at position 0.99 with {\arrow{triangle 45}}}, postaction={decorate}] node[left] {$e$} (a)
		edge[thick, decoration={markings, mark=at position 0.99 with {\arrow{triangle 45}}}, postaction={decorate} ] node[above] {$f$} (b)
	(b) 	edge[thick,loop, min distance = 10mm, looseness = 20, out = 225, in = 315, decoration={markings, mark=at position 0.99 with {\arrow{triangle 45}}}, postaction={decorate}] node[below] {$g_2$} (b)
	(b) 	edge[thick,loop, min distance = 10mm, looseness = 20, out = 45, in = 135, decoration={markings, mark=at position 0.99 with {\arrow{triangle 45}}}, postaction={decorate}] node[above] {$g_1$} (b);
\end{tikzpicture} \]
The graph $E$ satisfies condition~(L), but is not cofinal, so $\mathcal{G}_E$ is effective, but not minimal. We claim that $\mathcal{G}_E$ is densely minimal. To see this, note that any non-empty open subset of $E^\infty$ must contain a cylinder set $Z(\mu)$ where $r(\mu) = w$. And the restriction of $\mathcal{G}_E$ to~$Z(\mu)$ is minimal. As for covering, observe that the orbit of $e^\infty \in \partial E$ has length $1$, i.e.~$\smash{\orb_{\mathcal{G}_E}(e^\infty)} = \{e^\infty\}$. However, the topological full group $\llbracket \mathcal{G}_E \rrbracket$ still covers $\mathcal{G}_E$. For instance, the isotropy element $(e^\infty,1,e^\infty)$ belongs to the  full bisection
\[ U = Z(e^2,e) \bigsqcup Z(ef,g_1g_2) \bigsqcup Z(g_1,g_1g_1) \bigsqcup Z(f,f) \bigsqcup Z(g_2,g_2). \]
Similar full bisections can be found for $(e^\infty,k,e^\infty)$ where $k$ is any integer.
\end{example}

\begin{example}\label{ex:2orbitCover}
Consider the following graph:
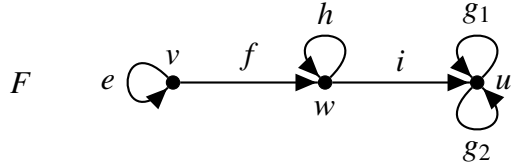
\[ \begin{tikzpicture}[vertex/.style={circle, draw = black, fill = black, inner sep=0pt,minimum size=5pt}]

\node at (-2,0) {$F$};
\node[vertex] (a) at (0,0) [label=above:$v$] {};
\node[vertex] (c) at (2,0) [label=below:$w$] {};
\node[vertex] (b) at (4,0) [label=right:$u$] {};

\path (a)	edge[thick, loop, min distance = 10mm, looseness = 10, out = 135, in = 225,decoration={markings, mark=at position 0.99 with {\arrow{triangle 45}}}, postaction={decorate}] node[left] {$e$} (a)
		edge[thick, decoration={markings, mark=at position 0.99 with {\arrow{triangle 45}}}, postaction={decorate} ] node[above] {$f$} (c)
	(b) 	edge[thick,loop, min distance = 10mm, looseness = 20, out = 225, in = 315, decoration={markings, mark=at position 0.99 with {\arrow{triangle 45}}}, postaction={decorate}] node[below] {$g_2$} (b)
	(b) 	edge[thick,loop, min distance = 10mm, looseness = 20, out = 45, in = 135, decoration={markings, mark=at position 0.99 with {\arrow{triangle 45}}}, postaction={decorate}] node[above] {$g_1$} (b)
	(c)	edge[thick, decoration={markings, mark=at position 0.99 with {\arrow{triangle 45}}}, postaction={decorate} ] node[above] {$i$} (b)
	edge[thick,loop, min distance = 10mm, looseness = 20, out = 45, in = 135, decoration={markings, mark=at position 0.99 with {\arrow{triangle 45}}}, postaction={decorate}] node[above] {$h$} (c);
\end{tikzpicture} \]
As in the previous example, $e^\infty \in \partial F$ has a singleton orbit. However, in contrast to the previous example, $\llbracket \mathcal{G}_F \rrbracket$ does \emph{not} cover $\mathcal{G}_F$. For there is no full bisection containing the element $(e^\infty,1,e^\infty)$. If $U$ is a bisection containing $(e^\infty,1,e^\infty)$, then $U$ must contain a bisection of the form $Z(e^{k+1}, e^k)$. Now since $Z(e^k) = Z(e^{k+1}) \sqcup Z(e^k f)$, it will be impossible to enlarge $U$ to a full bisection. By adding disjoint $Z(\mu, \lambda)$'s to write $U$ as in Proposition~\ref{prop:bis} one will always have one more $\mu$ ending in $w$ than $\lambda$'s. See also~\cite[Example~3.5]{BrixS} for the same phenomenon in a restricted transformation groupoid.
\end{example}

\section{Isomorphism theorems for graph groupoids}\label{sec:isogg}

In this section we will pursue specialized isomorphism theorems for graph groupoids. We will determine exactly when the topological full group of a graph groupoid belongs to $K^F$, and the conditions for this turn out to be weaker than minimality. We  will also determine, in terms of the graph, exactly when it belongs to $K^{LCC}$. From this we obtain two isomorphism theorems for graph groupoids.

\subsection{The class $K^F$}\label{subs:KF}
We are now going to give necessary and sufficient conditions for when~$(\Gamma, \partial E)$ belongs to $K^F$---for a graph $E$, and a subgroup $\Gamma \leq \llbracket \mathcal{G}_E \rrbracket$ containing $\DD(\llbracket \mathcal{G}_E \rrbracket)$. Of the three conditions (F1), (F2) and (F3) in Definition~\ref{def_F}, (F1) is the ``hardest'' one to satisfy. This is essentially because we need to produce elements in the topological full group with support containing a given point $x \in \partial E$, but also contained in a given neighbourhood of $x$. In the other two conditions we can get away with simply choosing a ``small enough'' support. As both conditions (F1) and (F3) fails in the presence of isolated points, we will only consider graphs that have no sinks, no semi-tails, and satisfy Condition~(L). We will see that Condition~(K) will be necessary for (F1) to hold for periodic\footnote{That is, $x = \lambda^\infty$ for some cycle $\lambda \in E^*$.} points. The two conditions in Definition~\ref{def:adhoc} below are needed to ensure that (F1) holds for wandering infinite paths, and for finite boundary paths, respectively. For notational convenience we make the following ad-hoc definitions.

\begin{definition}\label{def:adhoc}
Let $E$ be a graph. 
\begin{enumerate}
\item We say that $E$ satisfies \emph{Condition~(W)} if for every wandering infinite path $x \in E^\infty$, we have $\vert s(x) E^* r(x_n) \vert \geq 2$ for some $n \in \mathbb{N}$.
\item We say that $E$ satisfies \emph{Condition~($\infty$)} if for every infinite emitter $v \in E^0$, the set $\{ e \in v E^1 \mid r(e) \geq v \}$ is infinite.
\end{enumerate}
\end{definition}

The three conditions (K), (W) and ($\infty$) can be thought of as strengthenings of each of the three criteria for the boundary path space $\partial E$ being perfect (Proposition~\ref{prop:DEperfect}). The latter three criteria can informally be described as ``can exit'', whereas the former three can be described as ``can exit \emph{and} return''. More specifically, Condition~(L) means that one can exit every cycle, whereas Condition~(K) means that one can also return back to the same cycle. That~$E$ has no semi-tails means that every wandering infinite path has an exit, and Condition~(W) means that one can return to the same infinite path again. That~$E$ has no sinks can be reformulated as saying that every singular vertex has an exit (and hence infinitely many), whereas Condition~($\infty$) says that one can also return to the same vertex (from infinitely many of these exits). Note that Condition~($\infty$) holds in particular if every infinite emitter supports infinitely many loops. Also note that if $\vert s(x) E^* r(x_n) \vert \geq 2$ for some~$n \in \mathbb{N}$, then the same is true for each $m \geq n$. We now make two elementary observations needed in the proof of the next proposition.

\begin{lemma}\label{lem:cycles}
Let $E$ be a graph.
\begin{enumerate}
\item If $\mu \in E^*$ is a cycle and $E$ satisfies Condtion~(K), then there are infinitely many cycles $\lambda_1, \lambda_2, \ldots$ based at $s(\mu)$ such that $\mu, \lambda_1, \lambda_2, \ldots$ are mutually disjoint.
\item If $x = x_1 x_2 \ldots \in E^\infty$ is a wandering infinite path and $E$ satisfies Condition~(W), then for each $N \in \mathbb{N}$ there is an $n \in \mathbb{N}$ and paths $\mu_1, \ldots, \mu_N$ from $s(x)$ to $r(x_n)$ such that $x_{[1,n]}, \mu_1, \ldots, \mu_N$ are mutually disjoint.
\end{enumerate}
\end{lemma}
\begin{proof}
For the first part, let $\tau_1$ and $\tau_2$ be two distinct return paths based at $s(\mu)$. As distinct return paths are disjoint we must have that $\mu$ is disjoint from one of them, say $\tau_1$. And then the cycles $\mu, \tau_1 \mu, \tau_1^2 \mu, \tau_1^3 \mu, \ldots$ are all disjoint.

We argue inductively for the second part. Let $n_1 \in \mathbb{N}$ be such that $\vert s(x) E^* r(x_{n_1}) \vert \geq 2$, and put $v = r(x_{n_1})$. Since $x$ is wandering we can let $m_1 \geq n_1$ be the largest index such that~$r(x_{m_1}) = v$. So that $x$ never returns to $v$ after the $m_1$'th edge. Let $\mu$ be a path in~$s(x) E^* r(x_{m_1})$ distinct from $x_{[1,m_1]}$. If $x_{[1,m_1]}$ and $\mu$ are disjoint, then we are done with the base case. If not, then either $x_{[1,m_1]} < \mu$ or $x_{[1,m_1]} > \mu$. In the former case we have that~$\mu = x_{[1,m_1]} \rho$, where $\rho$ is a cycle based at $v$. As $x$ does not return to $v$ again we must have that $x_{[m_1 +1, m_1 + \vert \rho \vert]} \neq \rho$, and then $x_{[1, m_1 + \vert \rho \vert]}$ is disjoint from the path \[\mu_1 \coloneqq \mu x_{[m_1 +1, m_1 + \vert \rho \vert]} = x_{[1,m_1]} \rho x_{[m_1 +1, m_1 + \vert \rho \vert]}.\] If the latter is the case, then $\mu = x_{[1,k]}$ for some $k < m_1$ and $x_{[k+1, m_1]}$ is a cycle. And then the previous argument applied to $x_{[1,m_1]}$ and $\mu' = x_{[1,k]} x_{[k+1, m_1]} x_{[k+1, m_1]}$ shows that the statement holds for $N=1$.

Applying the above to the tail $x_{[m_1 + 1, \infty)}$, which is again a wandering infinite path, we get an index $m_2 > m_1$ and a path $\mu_2$ from $r(x_{m_1})$ to $r(x_{m_2})$ disjoint from $x_{[m_1 + 1, m_2]}$. By concatenating $x_{[1,m_1]}$ and $\mu_1$ with $x_{[m_1 +1,m_2]}$ and $\mu_2$ we obtain three paths from $s(x)$ to $r(x_{m_2})$ that are mutually disjoint, as well as disjoint from $x_{[1,m_2]}$. By continuing in this manner one sees that the result is true for all $N \in \mathbb{N}$.
\end{proof}

\begin{proposition}\label{prop:KFgraph}
Let $E$ be a graph with no sinks and let $\Gamma \leq \llbracket \mathcal{G}_E \rrbracket$ be a subgroup containing~$\DD(\llbracket \mathcal{G}_E \rrbracket)$. Then $(\Gamma, \partial E)$ belongs to $K^F$ if and only if $E$ satisfies Condition~(K), (W) and~($\infty$).
\end{proposition}
\begin{proof}
This proof is inspired by Matui's proof of~\cite[Proposition~3.6]{Mat}. We employ similar tricks in this more concrete, yet non-minimal context. We will first show that (F2) and (F3) holds when $E$ satisfies Condition~(K) and~(W). And then we will show, in turn, that all three conditions are necessary and sufficient for (F1) to hold at certain boundary paths.

Suppose $E$ satisfies Condition~(K) and (W) (in addition to having no sinks). We verify~(F3) first. Let $A$ be any non-empty clopen subset of $\partial E$. There is then a path $\eta$ such that~${Z(\eta) \subseteq A}$. Now there are two possibilities. Either $r(\eta)$ connects to a cycle, or~$r(\eta) E^\infty$ consists only of wandering paths. In the first case we may assume, by extending~$\eta$, that~$r(\eta)$ supports a cycle. By Lemma~\ref{lem:cycles} we can find three disjoint cycles~$\lambda_1, \lambda_2, \lambda_3$ based at~$r(\eta)$. Define~${V = Z(\eta \lambda_1, \eta \lambda_2)}$, $W = Z(\eta \lambda_2, \eta \lambda_3)$ and $\alpha = [\pi_{\hat{V}}, \pi_{\hat{W}}]$ (as in Lemma~\ref{lemma:bisectionInvolution}). Then~${\alpha \in \Gamma \setminus \{1\}}$ has order $3$ and $\supp(\alpha) \subseteq Z(\eta) \subseteq A$. In the case that $r(\eta) E^\infty$ consists only of wandering paths we may find, again by Lemma~\ref{lem:cycles}, three disjoint paths $\lambda_1, \lambda_2, \lambda_3$ starting at $r(\eta)$, and such that $r(\lambda_1) = r(\lambda_2) = r(\lambda_3)$. Defining $\alpha$ as above shows that (F3) holds in this case as well.

Next we verify (F2). To that end, let $\alpha \in \Gamma \setminus \{1\}$ with $\alpha^2 = 1$ and $\emptyset \subsetneq A \subseteq \supp(\alpha)$ a clopen be given. We have $\alpha = \pi_U$ with  
\[U = \left( \bigsqcup_{i=1}^k Z(\mu_i, F_i, \lambda_i) \right) \bigsqcup \left( \mathcal{G}_E^{(0)} \setminus \supp(\pi_U) \right) \]
as in Proposition~\ref{prop:bis}. Arguing as above, we can find a finite path $\eta$ and an index $1 \leq j \leq k$ such that~${Z(\eta) \subseteq A \cap Z(\lambda_j \setminus F_j)}$, as well as two disjoint paths $\tau_1, \tau_2$ satisfying~${s(\tau_1) = s(\tau_2) = r(\eta)}$ and~${r(\tau_1) = r(\tau_2)}$. As $\lambda_j \leq \eta$ we can write $\eta = \lambda_j \rho$ for some path $\rho$ whose first edge does not belong to $F_{j}$. Define the bisections
\[V = Z(\lambda_{j} \rho \tau_1, \lambda_{j} \rho \tau_2) \bigsqcup Z(\mu_{j} \rho \tau_1, \mu_{j} \rho \tau_2) \] and
\[W = Z(\mu_{j} \rho \tau_1, \lambda_{j} \rho \tau_1). \]
Put $\beta = [\pi_{\hat{V}}, \pi_{\hat{W}}]$. As $\alpha$ is an involution we have that $\alpha(\lambda_{j} z) = \mu_{j} z$ for $\lambda_{j} z \in Z(\lambda_{j} \setminus F_{j})$ and vice versa. Now observe that $\beta \in \Gamma$,
\begin{align*}
\supp(\beta) &= Z(\lambda_j \rho \tau_1) \sqcup Z(\lambda_j \rho \tau_2) \sqcup Z(\mu_j \rho \tau_1) \sqcup Z(\mu_j \rho \tau_2) \\
& \subseteq Z(\eta) \cup \alpha(Z(\eta)) \subseteq A \cup \alpha(A),
\end{align*}
and that $\alpha$ and $\beta$ agree on $\supp(\beta)$ (as they both swap the inital paths $\lambda_j$ and $\mu_j$). 

Assume now that $E$ merely has no sinks, no semi-tails and satisfies Condition~(L). We will show that (F1) holds if and only if $E$ satisfies Condition~(K), (W) and ($\infty$). Let $x \in \partial E$ and $A$ a clopen neighbourhood of $x$ be given. We further divide this part into three cases, each one yielding the necessity of one of the three conditions.

\textbf{Condtion~(K)}: Assume $E$ satisfies Condtion~(K), and suppose $x = x_1 x_2 \ldots \in E^\infty$ is an infinite non-wandering path. For $m \in \mathbb{N}$ large enough, we have that $Z(x_{[1,m]}) \subseteq A$. As $x$ contains infinitely many cycles we can, by possibly choosing $m$ larger, assume that~$x_{[m+1,n]}$ is a return path at $r(x_m)$ for some $n > m$. Using Lemma~\ref{lem:cycles} we can find three mutually disjoint cycles $\lambda_1, \lambda_2, \lambda_3$ based at $r(x_m)$ which are also disjoint from $x_{[m+1,n]}$. Let $\mu_i = x_{[1,m]} \lambda_i$ for~$i = 1,2,3$ and let $\mu_4 = x_{[1,n]}$. Define 
\[V = Z(\mu_1, \mu_2) \bigsqcup Z(\mu_3, \mu_4) \] and
\[W = Z(\mu_1, \mu_3). \]
Then $\alpha = [\pi_{\hat{V}}, \pi_{\hat{W}}] \in \Gamma$ satisfies $\supp(\alpha) = \bigsqcup_{i=1}^4 Z(\mu_i) \subseteq Z(x_{[1,m]}) \subseteq A$, $\alpha^2 = 1$ and~$x \in Z(\mu_4) \subseteq \supp(\alpha)$ as desired.

To see that Condition~(K) is necessary, suppose that $E$ does not satisfy it. Then there is a vertex $v \in E^0$ supporting a unique return path, say $\tau$. We may assume that $\tau$ has an exit~$f$ with~$s(f) = v$. Consider $x = \tau^\infty$ and its neighbourhood $A = Z(\tau)$. We claim that (F1) fails for this pair. To see this, suppose $\pi_U \in \llbracket \mathcal{G}_E \rrbracket$ satisfies $\tau^\infty \in \supp(\pi_U) \subseteq Z(\tau)$. By Proposition~\ref{prop:bis} we can find $Z(\mu, \lambda) \subseteq U$ with $r(\mu) = r(\lambda)$, $\mu \neq \lambda$ and $\tau^\infty \in Z(\lambda)$, which means that $\lambda \leq \tau^k$ for some $k \geq 1$. By possibly extending $\mu$ and $\lambda$ we may assume that~$\lambda = \tau^k$. We also have $Z(\mu) \subseteq Z(\tau)$, i.e.\ $\tau \leq \mu$, and $r(\mu) = r(\lambda) = v$. But since $\tau$ is the only return path based at $v$ we must have $\mu = \tau^l$ for some $l \neq k$ as $\mu \neq \lambda$. Let $z \in r(f) \partial E$. Then $(\pi_U)^2\left(\tau^{2k}fz\right) = \tau^{2l}fz \neq \tau^{2k}fz$, hence $\pi_U$ is not an involution, and therefore $(\Gamma, \partial E)$ does not satisfy (F1).

\textbf{Condtion~(W)}: Assume $E$ satisfies Condtion~(W), and suppose $x = x_1 x_2 \ldots \in E^\infty$ is an infinite wandering path. Choose $m$ large enough so that $Z(x_{[1,m]}) \subseteq A$. By Lemma~\ref{lem:cycles} there is an $n \geq m$ and three paths $\lambda_1, \lambda_2, \lambda_3$ from $s(x)$ to $r(x_n)$ such that $\lambda_1, \lambda_2, \lambda_3, x_{[1,n]}$ are mutually disjoint. Setting $\mu_i = x_{[1,m]} \lambda_i$ for $i = 1,2,3$ and $\mu_4 = x_{[1,n]}$, and defining $\alpha$ in the same way as in the case of Condition~(K) above gives the desired element in $\Gamma$. 

To see that Condition~(W) is necessary, suppose that there is an infinite wandering path~$x = x_1 x_2 \ldots \in E^\infty$ such that $\vert s(x) E^* s(x_n) \vert = 1$ for all $n \in \mathbb{N}$. We claim that (F1) fails for $A = Z(x_1)$. Indeed, suppose $\pi_U \in \llbracket \mathcal{G}_E \rrbracket$ satisfies $x \in \supp(\pi_U) \subseteq Z(x_1)$. By Proposition~\ref{prop:bis} we can find $Z(\mu, \lambda) \subseteq U$ with $r(\mu) = r(\lambda)$, $\mu \neq \lambda$ and $x \in Z(\lambda)$, which implies that $\lambda = x_{[1,m]}$ for some $m \geq 1$. But as $Z(\mu) \subseteq Z(x_1)$ we have that $s(\mu) = s(x)$ and $r(\mu) = r(x_m)$. It now follows that $\mu = \lambda$ since $\vert s(x) E^* s(x_m) \vert = 1$. This contradiction shows that there is not even an element $\pi_U \in \llbracket \mathcal{G}_E \rrbracket$ such that $x \in \supp(\pi_U) \subseteq Z(x_1)$.

\textbf{Condtion~($\infty$)}: Assume $E$ satisfies Condtion~($\infty$), and suppose $x = x_1 \ldots x_m \in E^*$ is a finite boundary path. Then for some $F \subseteq_{\text{finite}} r(x)E^1$ we have $Z(x \setminus F) \subseteq A$. By Condition~($\infty$) we can find three distinct edges $e_1, e_2, e_3 \in r(x)E^1 \setminus F$, and three (necessarily disjoint) cycles $\tau_1, \tau_2,  \tau_3$ based at $r(x)$ such that $e_i \leq \tau_i$ for $i = 1,2,3$. Let $F' = F \sqcup \{e_1, e_2, e_3\}$. Now define 
\[V = Z(x \tau_1,F', x) \bigsqcup Z(x \tau_2,F', x \tau_3) \] and
\[W = Z(x \tau_1,F', x \tau_2). \]
Then $\alpha = [\pi_{\hat{V}}, \pi_{\hat{W}}] \in \Gamma$ satisfies $\supp(\alpha) = Z(x \setminus F') \bigsqcup_{i=1}^3 Z(x \tau_i \setminus F') \subseteq Z(x \setminus F) \subseteq A$, $\alpha^2 = 1$ and $x \in Z(x \setminus F') \subseteq \supp(\alpha)$.

Finally, if $E$ does not satisfy Condition~($\infty$), then there is an infinite emitter $v \in E^0$ such that the set $F = \{ e \in v E^1 \mid r(e) \geq v \}$ is finite. And then (F1) fails for $x = v$ and~$A = Z(v \setminus F)$ as there is no element $\pi_U \in \llbracket \mathcal{G}_E \rrbracket$ whose support is contained in $Z(v \setminus F)$ and contains $v$.
The argument for this is essentially the same as in the necessity of Condition~(W) above.
\end{proof}

\begin{remark}
From Proposition~\ref{prop:KFgraph} we see that for a graph groupoid $\mathcal{G}_E$, the topological full group $\llbracket \mathcal{G}_E \rrbracket$ (on the boundary path space $\partial E$) belongs to the class $K^F$ if and only if its commutator subgroup $\DD(\llbracket \mathcal{G}_E \rrbracket)$ does. This is not something one would expect in general from the definition of $K^F$. It is clear that (F1) and (F3) in Definition~\ref{def_F} passes to supergroups, but (F2) need not do so. It is even more peculiar that the properties (F1), (F2) and~(F3) pass down to the commutator from $\llbracket \mathcal{G}_E \rrbracket$. This phenomenon might be an artifact of the combinatorial nature of the topological full group of a graph groupoid, and so it might also hold for other concrete classes of groupoids.
\end{remark}

\subsection{The class $K^{LCC}$}
Our next objective is to perform a similar analysis of when the space-group pair $(\llbracket \mathcal{G}_E \rrbracket, \partial E)$ for a graph $E$ belongs to $K^{LCC}$. In this case the ``mixing conditions'' will be weaker than for $K^F$ (c.f.\ Proposition~\ref{prop:KFgraph}), but we are only able to prove membership for the topological full group itself---no proper subgroups. As in the case of $K^F$ we need to stipulate that the boundary path space $\partial E$ has no isolated points (c.f.\ condition~(K1) in Definition~\ref{def_KB}), but also that the graphs are countable (this also for condition (K1)). By the results in Section~\ref{sec:reggrpd} we only have to determine when $\mathcal{G}_E$ is non-wandering, and when all orbits have length at least $3$. We shall soon see that the former property is characterized by excluding certain ``tree-like'' components in the graph $E$, which we make precise in the following defintion.

\begin{definition}\label{condT}
We say that a graph $E$ satisfies \emph{Condition~(T)} if for every vertex $v \in E^0$, there exists a vertex $w \in E^0$ such that $\vert v E^* w \vert \geq 2$. 
\end{definition}

Note that Condition~(T) implies that there are no sinks and no semi-tails. It does not, however, imply Condition~(L) as one can traverse a cycle twice to get two different paths. As long as there are no sinks, Condition~(W) implies Condition~(T). Condition~(T) is a fairly weak condition; it is in fact satisfied by all graphs that have finitely many vertices and no sinks, and more generally by any graph in which every vertex connects to a cycle. The archetypical example of graphs not satisfying Condition~(T) are trees, or more generally graphs containing such components.

As for when $\mathcal{G}_E$ can have orbits of length $1$ or $2$ one finds, by merely exhausting all possibilites, that this happens exactly if one or more of the following kinds of vertices are present in the graph~$E$. 

\begin{definition}\label{def:degenerate}
Let $E$ be a graph. We say that a vertex $v \in E^0$ is \emph{degenerate} if it is one of the following types:
\begin{enumerate}
\item \textbf{``1-loop-source''}:  $E^1 v = \{e\}$ where $e$ is a loop.
\item \textbf{``1 source to 1-loop-source''}: $E^1 v = \{e,f \}$ where $e$ is a loop and $s(f)$ is a source.
\item \textbf{``2-loop-source''}: There is another vertex $w \in E^0$ distinct from $v$ such that $E^1 v = \{e\} = w E^1 v$ and $E^1 w = \{f\} = v E^1 w$.
\item \textbf{``Infinite source''}: $v E^1$ is infinite and $E^1 v$ is empty. 
\item \textbf{``1 source to singular''}: $v$ is singular and $E^1 v = \{f\}$ where $s(f)$ is a source.
\item \textbf{``Stranded''}: $vE^1$ and $E^1 v$ are both empty.
\end{enumerate}
\end{definition}

\begin{proposition}\label{prop:KLCCgraph}
Let $E$ be a graph.
\begin{enumerate}
\item $\mathcal{G}_E$ is non-wandering if and only if $E$ satisfies Condition~(L) and (T).
\item $\left| \orb_{\mathcal{G}_E}(x) \right| \geq 3$ for all $x \in \partial E$ if and only if $E$ has no degenerate vertices.
\end{enumerate}
\end{proposition}
\begin{proof}
We prove part \emph{(1)} first. We may assume that $E$ has no sinks, as this is implied by both of the statements in \emph{(1)}. Suppose $E$ satisfies Condition~(L) and (T). Let $A$ be a non-empty clopen subset of $\partial E$. Then there is a path $\mu \in E^*$ such that $Z(\mu) \subseteq A$. Suppose first that~$r(\mu)$ connects to a cycle. Let $\lambda$ be such a cycle and let $\rho$ be a path from $r(\mu)$ to~$s(\lambda)$. We may assume that $\lambda$ has an exit $f$ with $s(f) = s(\lambda)$. Let $x \in r(f)E^\infty$. Then $\mu \rho f x$ and~$\mu \rho \lambda f x$ are two distinct tail-equivalent boundary paths in $A$. If, on the other hand, $r(\mu)$ does not connect to a cycle, then $r(\mu) E^\infty$ consists only of wandering paths that visit each vertex at most once. Let $w \in E^0$ be a vertex such that there are two distinct paths $\rho_1, \rho_2$ from $r(\mu)$ to $w$. Again letting $x \in wE^\infty$ be arbitrary we have that $\mu \rho_1 x$ and $\mu \rho_2 x$ are two distinct tail-equivalent boundary paths in $A$. Hence $A$ is not wandering.

To see that Condition~(L) and (T) are both necessary, note first that if $E$ does not satisfy Condition~(L), then $\partial E$ has an isolated point, and a clopen singleton is surely wandering. Assume instead that $E$ fails to satisfy Condition~(T), and let $v \in E^0$ be a vertex such that there is either no path or a unique path from $v$ to any other vertex in $E$. We claim that the cylinder set $Z(v)$ is wandering. We first consider a finite boundary path $\mu$ beginning in $v$ (if such a path exists). Then $r(\mu)$ is a singular vertex and \[ \orb_{\mathcal{G}_E}(\mu) \cap Z(v) = \{ \lambda \in E^* \mid s(\lambda) = v, \ r(\lambda) = r(\mu) \} = v E^* r(\mu) = \{\mu\}, \] as desired. Similarly, if $x \in v E^\infty$ and $y \in \orb_{\mathcal{G}_E}(x) \cap Z(v)$, then there are $k,l \in \mathbb{N}$ such that $x_{[k, \infty)} = y_{[l, \infty)}$. In particular $x_{[1, k-1]}$ and $y_{[1, l-1]}$ are finite paths from $v$ to $s(x_k) = s(y_l)$, hence these are equal and it follows then that $x = y$. Thus $\orb_{\mathcal{G}_E}(x) \cap Z(v) = \{x\}$. This proves the first part of the proposition.

For part \emph{(2)}, simply note that an orbit of length~$1$ can only occur if there are degenerate vertices of type (1), (4), or (6) as in Definition~\ref{def:degenerate} (the corresponding orbits of length~$1$ being $\{e^\infty\}$, $\{v\}$, $\{v\}$, respectively). And that an orbit of length $2$ can only occur if there are degenerate vertices of type (2), (3), or (5) (the corresponding orbits of length~$2$ being~$\{e^\infty, f e^\infty \}$, $\{ (ef)^\infty,  (fe)^\infty\}$, $\{v, f\}$, respectively).
\end{proof}

\begin{remark}\label{rem:notDM}
By an argument as in Example~\ref{ex:1orbitCover} one deduces that if a graph $E$ satisfies Condition~(I), then the graph groupoid $\mathcal{G}_E$ is densely minimal. However, statement \emph{(1)} in Proposition~\ref{prop:KLCCgraph} is strictly weaker than $\mathcal{G}_E$ being densely minimal. It is easy to cook up examples of infinite graphs satisfying Condition~(L) and (T), but whose graph groupoids are not densely minimal. One such example is:
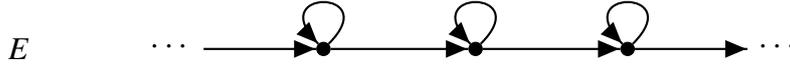
\[ \begin{tikzpicture}[vertex/.style={circle, draw = black, fill = black, inner sep=0pt,minimum size=5pt}]

\node at (-4,0) {$E$};
\node[vertex] (a) at (0,0)  {};
\node[vertex] (c) at (2,0)  {};
\node[vertex] (b) at (4,0) {};
\node (e) at (6,0) {$\cdots$};
\node (f) at (-2,0) {$\cdots$};

\path (f)	edge[thick, decoration={markings, mark=at position 0.99 with {\arrow{triangle 45}}}, postaction={decorate} ]  (a)
	 (a) edge[thick,loop, min distance = 10mm, looseness = 20, out = 45, in = 135, decoration={markings, mark=at position 0.99 with {\arrow{triangle 45}}}, postaction={decorate}]  (a)
		edge[thick, decoration={markings, mark=at position 0.99 with {\arrow{triangle 45}}}, postaction={decorate} ]  (c)
	(b) 	edge[thick,loop, min distance = 10mm, looseness = 20, out = 45, in = 135, decoration={markings, mark=at position 0.99 with {\arrow{triangle 45}}}, postaction={decorate}]  (b) edge[thick, decoration={markings, mark=at position 0.99 with {\arrow{triangle 45}}}, postaction={decorate} ]  (e)
	(c)	edge[thick, decoration={markings, mark=at position 0.99 with {\arrow{triangle 45}}}, postaction={decorate} ]  (b)
	edge[thick,loop, min distance = 10mm, looseness = 20, out = 45, in = 135, decoration={markings, mark=at position 0.99 with {\arrow{triangle 45}}}, postaction={decorate}] (c);
\end{tikzpicture} \]
\end{remark}

\subsection{Isomorphism theorems}
Recall that all orbits having length at least $3$ is sufficient for the commutator subgroup of the topological full group to cover the groupoid (Lemma~\ref{ex_cover}). This in turn means that the groupoid can be recovered as the groupoid of germs of any subgroup between the topological full group and its commutator. Combined with Propositions~\ref{prop:KFgraph} and \ref{prop:KLCCgraph} we will obtain the two isomorphism results for graph groupoids. We begin by first observing that the conditions on the graph for membership in $K^F$ actually implies that all orbits are infinite.

\begin{lemma}\label{KForbits}
Let $E$ be a graph with no sinks and suppose $E$ satisfies Condition~(K) and~($\infty$). Then $\orb_{\mathcal{G}_E}(x)$ is infinite for each $x \in \partial E$. In particular, $E$ has no degenerate vertices. 
\end{lemma}
\begin{proof}
We first consider the $\mathcal{G}_E$-orbits of finite boundary paths. Suppose $v \in E^0$ is an infinite emitter. Condition~($\infty$) implies that there are infinitely many distinct return paths at~$v$, hence $\orb_{\mathcal{G}_E}(\mu)$ is infinite for each $\mu \in \partial E \cap E^*$.

Next, let $x \in E^\infty$ be an infinite path. If $x$ is eventually periodic, then $x = \mu \lambda^\infty$ for some finite path $\mu$ and some cycle $\lambda$. Lemma~\ref{lem:cycles} gives a sequence of mutually disjoint cycles~$\tau_1, \tau_2, \ldots$ based at $s(\lambda)$. And then $\{ \tau_1 \lambda^\infty, \tau_2 \lambda^\infty, \ldots \}$ is an infinite subset of $\orb_{\mathcal{G}_E}(x)$. If $x$ is not eventually periodic, then $\{ x, x_{[2, \infty]}, x_{[3, \infty]}, \ldots \}$ is an infinite subset of $\orb_{\mathcal{G}_E}(x)$.
\end{proof}

In terms of the class $K^F$ we obtain the following isomorphism result, which relaxes the assumptions in Theorem~\ref{thm:KFgroupoid} considerably for graph groupoids.

\begin{theorem}\label{thm:KFrigid}
Let $E$ and $F$ be graphs with no sinks, and suppose they both satisfy Condition~(K), (W) and ($\infty$). Suppose $\Gamma, \Lambda$ are subgroups with $\DD(\llbracket \mathcal{G}_{E} \rrbracket) \leq \Gamma \leq \llbracket \mathcal{G}_{E} \rrbracket$ and $\DD(\llbracket \mathcal{G}_{F} \rrbracket) \leq \Lambda \leq \llbracket \mathcal{G}_{F} \rrbracket$. If $\Gamma \cong \Lambda$ as abstract groups, then $\mathcal{G}_{E} \cong \mathcal{G}_{F}$ as topological groupoids. In particular, the following are equivalent:
\begin{enumerate}
\item $\mathcal{G}_{E} \cong \mathcal{G}_{F}$ as topological groupoids.
\item $\llbracket \mathcal{G}_{E} \rrbracket \cong \llbracket \mathcal{G}_{F} \rrbracket$ as abstract groups.
\item $\DD(\llbracket \mathcal{G}_{E} \rrbracket) \cong \DD(\llbracket \mathcal{G}_{F} \rrbracket)$ as abstract groups.
\end{enumerate}
\end{theorem}
\begin{proof}
Combine Proposition~\ref{prop:KFgraph}, Theorem~\ref{classF}, Proposition~\ref{faithGerm}, Lemma~\ref{KForbits}, Lemma~\ref{ex_cover} and Proposition~\ref{prop_germs}.
\end{proof}

The preceding result covers---in particular---all finite graphs that have no sinks and satisfy Condition~(K). As for an isomorphism result in terms of $K^{LCC}$, we combine Proposition~\ref{prop:KLCCgraph} with Theorem~\ref{thm:KLCCgroupoid} to get the following result.

\begin{theorem}\label{thm:KLCCrigid}
Let $E$ and $F$ be countable graphs satisfying Condition~(L) and (T), and having no degenerate vertices. Then the following are equivalent:
\begin{enumerate}
\item $\mathcal{G}_{E} \cong \mathcal{G}_{F}$ as topological groupoids.
\item $\llbracket \mathcal{G}_{E} \rrbracket \cong \llbracket \mathcal{G}_{F} \rrbracket$ as abstract groups.
\end{enumerate}
\end{theorem}

This result covers---in particular---all finite graphs that have no degenerate vertices nor sinks, and which satisfy Condition~(L).

\begin{remark}
In~\cite{Mats2}, Matsumoto established a version of Theorem~\ref{thm:KLCCrigid} for finite graphs which are strongly connected (and satisfy Condition~(L), or equivalently Condition~(K)). At about the same time, Matui announced~\cite{Mat}, and his Isomorphism Theorem therein applies to the enlarged class of graphs which have finitely many vertices, countably many edges, no sinks, are cofinal, satisfy Condition~(L) and for which every vertex can reach every infinite emitter.
\end{remark}

Combining Theorem~\ref{thm:KLCCrigid} with~\cite[Theorem~5.1]{BCW} and~\cite[Corollary~4.2]{CR} we obtain the rigidity result in Corollary~\ref{cor:BCW} below, which ties in many of the mathematical structures associated to (directed) graphs. For background on graph $C^*$-algebras, see~\cite{Rae}\footnote{Beware that the convention for paths in graphs in Raeburn's book is opposite of the one used in this paper.}, and for Leavitt path algebras, see~\cite{AAM}.

\begin{corollary}\label{cor:BCW}
Let $E$ and $F$ be countable graphs satisfying Condition~(L) and (T), and having no degenerate vertices. Let $R$ be an integral domain. Then the following are equivalent:
\begin{enumerate}
\item The graph groupoids $\mathcal{G}_{E}$ and $\mathcal{G}_{F}$ are isomorphic as topological groupoids.
\item There is an isomorphism of the graph $C^*$-algebras $C^*(E)$ and $C^*(F)$ which maps the diagonal $\mathcal{D}(E)$  onto $\mathcal{D}(F)$.   
\item There is an isomorphism of the Leavitt path algebras $L_R(E)$ and $L_R(F)$ which maps the diagonal $D_R(E)$ onto $D_R(F)$.
\item The pseudogroups $\mathcal{P}_E$ and $\mathcal{P}_F$ are spatially isomorphic.
\item The graphs $E$ and $F$ are (continuously) orbit equivalent.
\item The topological full groups $\llbracket \mathcal{G}_{E} \rrbracket$ and $\llbracket \mathcal{G}_{F} \rrbracket$ are isomorphic as abstract groups.
\end{enumerate}
\end{corollary}

\begin{remark}
Statement~(5) in Corollary~\ref{cor:BCW} coincides with Li's notion of \emph{continuous orbit equivalence} for the partial dynamical systems associated to the graphs, c.f.~\cite{Li}.
\end{remark}

\begin{remark}
We remark that in Corollary~\ref{cor:BCW} statements (1), (2) and (3) are always equivalent, statements (4) and (5) are always equivalent and they are implied by (1), (2) and (3). Furthermore, if the graphs satisfy Condition~(L), then statements (1)--(5) are equivalent. Additionally, the equivalence of (1) and (2) has recently been shown in greater generality~\cite{CRST}. The same is true for (1) and (3) by recent work of Steinberg~\cite{Stein}, even with weaker assumptions on the ring $R$.
\end{remark}

\section{Embedding theorems}\label{sec:embed}
In this final section we will show that several classes of groupoids embed into a certain fixed graph groupoid---namely the groupoid of the graph that consists of a single vertex and two edges. This class includes graph groupoids and AF-groupoids. We will also discuss the induced embeddings of the associated graph algebras and the topological full groups.

\subsection{Embedding graph groupoids}
Let $E_2$ denote the graph with a single vertex $v$, and two edges $a$ and $b$:
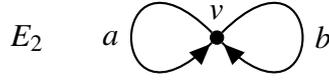
\[ \begin{tikzpicture}[vertex/.style={circle, draw = black, fill = black, inner sep=0pt,minimum size=5pt}]
\node at (-2.5,0) {$E_2 $};
\node[vertex] (a) at (0,0) [label=above:$v$] {};

\path (a) 	edge[thick, loop, min distance = 20mm, looseness = 10, out = 135, in = 225,decoration={markings, mark=at position 0.99 with {\arrow{triangle 45}}}, postaction={decorate}] node[left] {$a$} (a)
edge[thick, loop, min distance = 20mm, looseness = 10, out = 45, in = 315,decoration={markings, mark=at position 0.99 with {\arrow{triangle 45}}}, postaction={decorate}] node[right] {$b$} (a);
\end{tikzpicture} \]

In~\cite{BS}, Brownlowe and Sørensen proved an algebraic analog of Kirchberg's Embedding Theorem (see~\cite{KP}) for Leavitt path algebras. They showed that for any countable graph $E$, and for any commutative unital ring $R$, the Leavitt path algebra $L_R(E)$ embeds (unitally, whenever it makes sense) into $L_R(E_2)$. By inspecting their proof one finds that this embedding is also \emph{diagonal-preserving}, i.e.\ that the canonical diagonal $D_R(E)$ is mapped into $D_R(E_2)$. A special case of Kirchberg's Embedding Theorem is that any graph $C^*$-algebra, $C^*(E)$, embeds into the Cuntz algebra $\mathcal{O}_2$, which is canonically isomorphic to the graph $C^*$-algebra $C^*(E_2)$ (and the groupoid $C^*$-algebra $C^*_r\left(\mathcal{G}_{E_2}\right)$). We denote the canonical diagonal subalgebra in $\mathcal{O}_2$ by $\mathcal{D}_2$. A priori, Kirchberg's embedding is of an analytic nature, but Brownlowe and Sørensen's results shows that in the case of graph $C^*$-algebras, algebraic embeddings exist. Both graph $C^*$-algebras and Leavitt path algebras have the same underlying groupoid models (being canonically isomorphic to the groupoid $C^*$-algebra, and the Steinberg $R$-algebra ($A_R(\mathcal{G}_E)$) of $\mathcal{G}_E$, respectively). Generally, isomorphisms of the graph groupoids correspond to diagonal preserving isomorphisms of the algebras. Thus, one could wonder whether there is an embedding of the underlying graph groupoids. We will show that this is indeed the case, modulo topological obstructions. Our proof is inspired by~{\cite[Proposition~5.1]{BS}} (and the examples following it).

\begin{lemma}\label{lem:emb}
Let $E$ be a countable graph with no sinks, no semi-tails, and suppose $E$ satisfies Condition~(L). Then there exists an injective local homeomorphism $\phi \colon \partial E \to E^\infty_2$ such that \[\phi \circ \llbracket \mathcal{G}_E \rrbracket \subseteq \llbracket \mathcal{G}_{E_2} \rrbracket \circ \phi.\]
If $E^0$ is finite, then $\phi$ is surjective (hence a homeomorphism), and if $E^0$ is infinite, then~$\phi(\partial E) = E_2^\infty \setminus \{a^\infty\}$. In particular, there exists an injective étale homomorphism \[\Phi \colon \G(\llbracket \mathcal{G}_E \rrbracket, \partial E) \to \GG_{E_2}.\]
\end{lemma}
\begin{proof}
For transparency we first treat the case when $E^0$ is finite. The infinite case requires only a minor tweak. Let $n = | E^0 |$. Label the vertices and edges of $E$ (arbitrarily) as \[E^0 = \{w_1, w_2, \ldots w_n\} \quad \text{and} \quad w_i E^1 = \{ e_{i,j} \mid  1 \leq j \leq  k(i) \} \text{ for each } 1 \leq i \leq n, \]
where $k(i) = | s^{-1}(w_i) |$. When $w_i$ is an infinite emitter, $k(i) = \infty$, and we let $j$ range over~$\mathbb{N}$. For each pair $j,i$ with $j \in \mathbb{N}, \ i \in \mathbb{N} \cup \{\infty\}$ and $j \leq i$ we define a finite path $\alpha_{j,i} \in E_2^*$ as follows: $\alpha_{1,1} \coloneqq v$ and for $j \geq 2$ 
\[ \alpha_{j,i} \coloneqq \left\lbrace\begin{matrix}
b & \text{ if } j=1, \\
a^{j-1}b & \text{\qquad if } 1 < j < i, \\
a^{j-1} & \text{ if } j=i.
\end{matrix} \right. \]
Observe that for each fixed $i \in \mathbb{N}$, the set $\{ Z(\alpha_{j,i}) \mid 1 \leq j \leq i \}$ forms a partition of $E_2^\infty$. And for~$i = \infty$, $\{ Z(\alpha_{j,i} ) \mid 1 \leq j < \infty \}$ forms a partition of $E_2^\infty \setminus \{a^\infty\}$.

We now define the map $\phi \colon \partial E \to E_2^\infty$ as follows. For an infinite path~${x=e_{i_1,j_1}e_{i_2,j_2} \ldots}$ in~$E$ we set
\[\phi(x) = \alpha_{i_1,n} \alpha_{j_1,k(i_1)} \alpha_{j_2,k(i_2)} \ldots \ . \]
If $w_i \in E^0$ is an infinite emitter, then 
\[ \phi(w_i)= \alpha_{i,n} a^\infty. \]
For notational convenience, we define 
\[ \phi^*(\mu) \coloneqq \alpha_{i_1,n} \alpha_{j_1,k(i_1)} \alpha_{j_2,k(i_2)} \ldots \alpha_{j_m,k(i_m)} \in E_2^*  \]
for each finite path $\mu = e_{i_1,j_1} e_{i_2,j_2} \ldots e_{i_m,j_m} \in E^*$. Finally, if $\mu$ is a finite boundary path, then 
\[ \phi(\mu) = \phi^*(\mu) a^\infty. \]

Recall that $v \alpha = \alpha = \alpha v$ for each $\alpha \in E_2^*$. A priori, $\phi(x)$ could be a finite path in $E_2$. We argue that this is not the case. For a finite path $\mu \in E^*$, $\phi(\mu)$ is clearly infinite. For an infinite path $x = e_{i_1,j_1}e_{i_2,j_2} \ldots$, $\phi(x)$ is finite if and only if for some $M \in \mathbb{N}$, $\alpha_{j_m,k(i_m)} = v$ for all $m > M$, that is $k(i_m) = 1$ and $j_m = 1$. This means that $e_{i_{M+1},j_{M+1}}e_{i_{M+2},j_{M+2}} \ldots$ is either a semi-tail, or an eventually periodic point whose cycle has no exit. But there are by assumption no such paths in $E$. So we conclude that $\phi$ is well-defined.

Using the fact that $\{ Z(\alpha_{j,i})\}$ for fixed $i$ forms a partition of $E_2^\infty$, or $E_2^\infty \setminus \{a^\infty\}$, one easily sees that $\phi$ is a bijection. As for continuity, we define $F_{i,l} \coloneqq \{e_{i,1}, e_{i,2}, \ldots, e_{i,l}\} \subseteq w_iE^1$ for~$1 \leq l < k(w_i) +1$. Let $\mu = e_{i_1,j_1} e_{i_2,j_2} \ldots e_{i_m,j_m} \in E^*$ and suppose $r(\mu) = w_i$. Observe that 
\[\phi(Z(\mu)) = Z(\phi^*(\mu))\] and 
\[\phi(Z(\mu \setminus F_{i,l})) = Z(\phi^*(\mu) a^l).\]
For arbitrary $F = \{e_{i,j_1}, \ldots, e_{i,j_m}\}$ we have
\begin{equation}\label{eq:1}
Z(\mu \setminus F) = Z(\mu \setminus F_{i,j_m +1}) \bigsqcup \bigsqcup_{j \in J_F} Z(\mu e_{i,j}),
\end{equation}
where $J_F$ is the set of $j$'s with $1 \leq j \leq j_m$ and $e_{i,j} \notin F$. Thus $\phi$ is an open map. Conversely, we have that for $\beta \in E_2^*$
\[ \phi^{-1}(Z(\beta)) = \left( \bigcup_{\beta \leq \phi^*(\mu)} Z(\mu) \right) \bigcup \left( \bigcup_{l=1}^\infty \bigcup_{\beta \leq \phi^*(\lambda) a^l} Z(\lambda \setminus F_{r(\lambda),l}) \right),  \]
(and these unions may actually be taken to be finite). Hence $\phi$ is a homeomorphism. 

To see that $\phi \circ \llbracket \mathcal{G}_E \rrbracket \circ \phi^{-1}\subseteq \llbracket \mathcal{G}_{E_2} \rrbracket$, let $\mu, \lambda \in E^*$ with $r(\mu) = r(\lambda) = w_i$ be given, and let $1 \leq l < k(w_i) +1$. Observe that 
\[\phi \circ \pi_{Z(\mu, \lambda)} \circ \phi^{-1} = \pi_{Z(\phi^*(\mu), \phi^*(\lambda))} \colon Z(\phi^*(\lambda) \to Z(\phi^*(\mu)), \]
and
\[\phi \circ \pi_{Z(\mu, F_l ,\lambda)} \circ \phi^{-1} = \pi_{Z(\phi^*(\mu)a^l, \phi^*(\lambda)a^l)} \colon Z(\phi^*(\lambda) a^l \to Z(\phi^*(\mu)a^l), \]
as partial homeomorphisms. Utilizing a similar decomposition as in Equation~(\ref{eq:1}) for the basic set~$Z(\mu, F, \lambda)$ for arbitrary $F$, together with the description of elements in $\llbracket \mathcal{G}_E \rrbracket$ from Proposition~\ref{prop:bis}, we see that for each $\pi_U \in \llbracket \mathcal{G}_E \rrbracket$, the homeomorphism $\phi \circ \pi_U \circ \phi^{-1}$ belongs to~$\llbracket \mathcal{G}_{E_2} \rrbracket$.

In the case that $E^0$ is infinite, all the arguments above still go through, with the minor adjustment that the first word in $\phi(x)$ is $\alpha_{i_1, \infty}$. This word always ends with a $b$, so we see that $\phi$ becomes a homeomorphism from $\partial E$ onto $E_2^\infty \setminus \{a^\infty\}$.

The final statement follows from Corollary~\ref{corol_inj_groupoid} and Proposition~\ref{prop_germs} ($\llbracket \mathcal{G}_{E_2} \rrbracket$ covers $\mathcal{G}_{E_2}$ since the groupoid is minimal).
\end{proof}

\begin{remark}
The local homeomorphism $\phi$ constructed in the preceding proof depends on the choice of labeling of the graph. And there are of course many ways to label a graph, but each one gives a local homeomorphism $\phi$ with the desired properties.
\end{remark}

In order to conclude that $\GG_E$ embeds into $\GG_{E_2}$ it seems like we have to assume that $\llbracket \mathcal{G}_E \rrbracket$ covers $\GG_E$ (as this is not always the case). However, in the proof of Lemma~\ref{lem:emb} we are really showing that $\phi \circ \mathcal{P}_c(\mathcal{G}_E) \subseteq \mathcal{P}_c(\mathcal{G}_{E_2}) \circ \phi$, where $\mathcal{P}_c(\mathcal{G})$ denotes the inverse semigroup of partial homeomorphisms $\pi_U \colon s(U) \to r(U)$ coming from compact bisections $U \subseteq \mathcal{G}$. It is a sub-inverse semigroup of Renault's pseudogroup as in~\cite{Ren2}, \cite{BCW} (when $\mathcal{G}$ is effective). The constructions in Sections~\ref{sec:germ} and~\ref{sec:SpatG} apply more or less verbatim to $\mathcal{P}_c(\mathcal{G})$. The crucial difference is that $\mathcal{P}_c(\mathcal{G})$ always covers $\mathcal{G}$, when $\mathcal{G}$ is  ample. Thus, the analogs of Corollary~\ref{corol_inj_groupoid} and Proposition~\ref{prop_germs} for $\mathcal{P}_c(\mathcal{G})$ applied to $\phi$ induces the desired embedding of the graph groupoids---which we record in the following theorem. 

\begin{theorem}\label{thm:emb}
Let $E$ be a countable graph satisfying Condition~(L) and having no sinks nor semi-tails. Then there is an embedding of étale groupoids $\Phi \colon \GG_E \hookrightarrow \GG_{E_2}$. If $E^0$ is finite, then $\Phi$ maps $\partial E$ onto $E_2^\infty$.
\end{theorem}

\begin{remark}
Theorem~\ref{thm:emb} is optimal in the sense there is no embedding if one relaxes the assumptions on $E$. For if $\partial E$ has isolated points, then there is no local homeomorphism from $\partial E$ to $E_2^\infty$, as the latter has no isolated points. And if $E$ is uncountable, then there is no embedding either, for then $\partial E$ is not second countable, while $E_2^\infty$ is. Similarly, $\partial E$ cannot map onto $E_2^\infty$ if $E^0$ is infinite, for then the former is not compact.
\end{remark}

\subsection{Diagonal embeddings of graph algebras}\label{subs:diagonal}
From Theorem~\ref{thm:emb} we recover Brownlowe and Sørensen's embedding theorem for Leavitt path algebras (albeit for the slightly smaller class of graphs $E$ with $\partial E$ having no isolated points). However, we get the additional conclusion that when $E^0$ is finite (i.e.\ the algebras are unital), the embedding can be chosen to not only be unital, but also to map the diagonal \emph{onto} the diagonal.

\begin{corollary}\label{cor:graphAlgEmb}
Let $E$ be a countable graph with no sinks, no semi-tails, and satisyfing Condition~(L).
\begin{enumerate}
\item There is an injective $*$-homomorphism $\psi \colon C^*(E) \to \mathcal{O}_2$ such that $\psi(\mathcal{D}(E)) \subseteq \mathcal{D}_2$. If $E^0$ is finite, then $\psi$ is unital and $\psi(\mathcal{D}(E)) = \mathcal{D}_2$.
\item For any commutative unital ring $R$, there is an injective $*$-algebra homomorphism $\rho \colon L_R(E) \to L_R(E_2)$ such that $\rho(D_R(E)) \subseteq D_R(E_2)$. If $E^0$ is finite, then $\rho$ is unital and $\rho(D_R(E)) = D_R(E_2)$. 
\end{enumerate}
\end{corollary}

\begin{remark}\label{rem:emb}
For each labeling of a graph $E$ as in the proof of Lemma~\ref{lem:emb}, one obtains explicit embeddings of both the graph $C^*$-algebras and the Leavitt path algebras into~$\mathcal{O}_2$ and~$L_R(E_2)$, respectively, in terms of their canonical generators. This is done by expanding the scheme in~\cite[Proposition~5.1]{BS}. The canonical isomorphism between both $C^*(E)$ and $C^*(\mathcal{G}_E)$, and $L_R(E)$ and $A_R(\mathcal{G}_E)$ is given by $p_v \leftrightarrow 1_{Z(v)}$ for $v \in E^0$ (vertex projections) and $s_e \leftrightarrow 1_{Z(e, r(e))}$ for $e \in E^1$ (edge partial isometries). Denote the generators in $\mathcal{O}_2$ and~$L_R(E_2)$ by $s_a$ and $s_b$. Given a labeling $E^0 = \{w_1, w_2, w_3 \ldots\}$ and~${E^1 = \{ e_{i,j} \mid 1 \leq i \leq n, \ 1 \leq j \leq  k(i) \}}$, the embedding of the algebras induced by $\phi$ as in Lemma~\ref{lem:emb} is given on the generators by
\[p_{w_i} \longmapsto s_{\phi^*(w_i)} \left(s_{\phi^*(w_i)}\right)^*, \qquad s_{e_{i,j}} \longmapsto s_{\phi^*(e_{i,j})}\left(s_{\phi^*(r(e_{i,j}))}\right)^*,\]
where $\phi^*(\mu) \in \{a,b\}^*$ is as in the proof of Lemma~\ref{lem:emb} (recall that for $\mu = e_1, \ldots e_n \in E^*$, we define $s_{\mu} \coloneqq s_{e_1} \cdots s_{e_2}$).
\end{remark}

\begin{remark}
In the case that $E$ has infinitely many vertices, the image of the diagonals in Corollary~\ref{cor:graphAlgEmb} can be described as follows:
\[ \psi(\mathcal{D}(E)) = \overline{\spn \{ s_\alpha s_\alpha^* \mid \alpha \in E_2^* \setminus \{a, a^2, a^3, \ldots\}  \}}, \]
and 
\[ \rho(D_R(E)) = \spn_R \{ s_\alpha s_\alpha^* \mid \alpha \in E_2^* \setminus \{a, a^2, a^3, \ldots\}  \}. \]
\end{remark}

For examples of explicit embeddings for finite graphs satisfying Condition~(L) (possibly even having sinks), see Section~5 of~\cite{BS}. As for  infinite graphs, we provide a few examples below.

\begin{example}
Consider the following graph, whose graph $C^*$-algebra is the Cuntz algebra~$\mathcal{O}_\infty$:
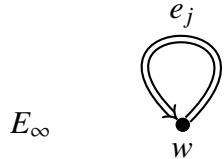
\[ \begin{tikzpicture}[vertex/.style={circle, draw = black, fill = black, inner sep=0pt,minimum size=5pt}, implies/.style={double,double equal sign distance,-implies}]
\node at (-2,0) {$E_\infty$};
\node[vertex] (a) at (0,0) [label=below:$w$] {};

\path (a) 	edge[implies, thick, loop, min distance = 20mm, looseness = 10, out = 45, in = 135]  node[above] {$e_j$} (a);
\end{tikzpicture} \]
The double arrow indicates infinitely many edges, i.e.\ $E^1 = \{  e_1, e_2, e_3, \ldots \}$. For simplicity, we denote the edge isometries by $s_j$ for $j \in \mathbb{N}$. We label $w = w_1$ and $e_j = e_{1,j}$. Following the recipe in Remark~\ref{rem:emb} we obtain a unital embedding of $\mathcal{O}_\infty$ into $\mathcal{O}_2$ (and similarly of~$L_R(E_\infty)$ into $L_R(E_2)$) which maps the diagonal onto the diagonal, in terms of generators as follows:
\[p_w = 1_{\mathcal{O}_\infty} \longmapsto 1_{\mathcal{O}_2} = p_v, \qquad s_j \longmapsto s_{a^{j-1}b}. \]
\end{example}

\begin{example}
Next, consider the following graph:
\[ \begin{tikzpicture}[vertex/.style={circle, draw = black, fill = black, inner sep=0pt,minimum size=5pt}, implies/.style={double,double equal sign distance,-implies}]
\node at (-2,0) {$E$};
\node[vertex] (a) at (0,0) [label=below:$w_1$] {};
\node[vertex] (b) at (2,0) [label=below:$w_2$] {};

\path (a) edge[implies, thick, loop, min distance = 20mm, looseness = 10, out = 45, in = 135]  node[above] {$e_j$} (a) edge[thick, decoration={markings, mark=at position 0.99 with {\arrow{triangle 45}}}, postaction={decorate} ] node[above] {$h$} (b)
 (b) edge[implies, thick, loop, min distance = 20mm, looseness = 10, out = 45, in = 135]  node[above] {$f_j$}   (b);
\end{tikzpicture} \]
By labeling the edges as $h = e_{1,1}, \ e_j = e_{1,j+1}, \ f_j = e_{2,j}$ we get the following unital diagonal preserving embedding of $C^*(E)$ into $\mathcal{O}_2$:
\begin{align*}
&p_{w_1}  \longmapsto s_b s_b^*, \qquad p_{w_2}  \longmapsto s_a s_a^*, \\ &s_h \longmapsto s_{bb}s_a^*, \qquad s_{e_j} \longmapsto s_{ba^jb}s_b^*, \qquad s_{f_j} \longmapsto s_{ba^jb}s_a^*.
\end{align*}
\end{example}

\begin{example}
Finally, let us look at a graph with infinitely many vertices:
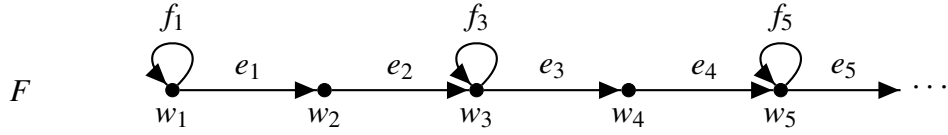
\[ \begin{tikzpicture}[vertex/.style={circle, draw = black, fill = black, inner sep=0pt,minimum size=5pt}]

\node at (-4,0) {$F$};
\node[vertex] (a) at (-2,0) [label=below:$w_1$] {};
\node[vertex] (b) at (0,0) [label=below:$w_2$]  {};
\node[vertex] (c) at (2,0) [label=below:$w_3$] {};
\node[vertex] (d) at (4,0)[label=below:$w_4$] {};
\node[vertex] (e) at (6,0)[label=below:$w_5$] {};
\node (f) at (8,0) {$\cdots$};

\path (a) edge[thick,loop, min distance = 10mm, looseness = 20, out = 45, in = 135, decoration={markings, mark=at position 0.99 with {\arrow{triangle 45}}}, postaction={decorate}] node[above] {$f_1$}  (a)
		edge[thick, decoration={markings, mark=at position 0.99 with {\arrow{triangle 45}}}, postaction={decorate} ] node[above] {$e_1$} (b)
	(b) edge[thick, decoration={markings, mark=at position 0.99 with {\arrow{triangle 45}}}, postaction={decorate} ] node[above] {$e_2$}  (c)
	(c)	edge[thick, decoration={markings, mark=at position 0.99 with {\arrow{triangle 45}}}, postaction={decorate} ] node[above] {$e_3$} (d)
	edge[thick,loop, min distance = 10mm, looseness = 20, out = 45, in = 135, decoration={markings, mark=at position 0.99 with {\arrow{triangle 45}}}, postaction={decorate}] node[above] {$f_3$}  (c)
	(d) edge[thick, decoration={markings, mark=at position 0.99 with {\arrow{triangle 45}}}, postaction={decorate} ] node[above] {$e_4$}  (e)
	(e) edge[thick, decoration={markings, mark=at position 0.99 with {\arrow{triangle 45}}}, postaction={decorate} ] node[above] {$e_5$} (f) edge[thick,loop, min distance = 10mm, looseness = 20, out = 45, in = 135, decoration={markings, mark=at position 0.99 with {\arrow{triangle 45}}}, postaction={decorate}] node[above] {$f_5$}  (e);
\end{tikzpicture} \]
We label the edges as $e_j = e_{j,1}$ for $j \in \mathbb{N}$, and $f_j = e_{j,2}$ for $j$ odd. The induced diagonal preserving embedding of $C^*(F)$ into $\mathcal{O}_2$ is then given on the generators as follows:
\begin{align*}
&p_{w_i}  \longmapsto s_{a^{i-1}b} \left(s_{a^{i-1}b} \right)^*, \qquad s_{f_j} \longmapsto s_{a^{j-1}ba} \left(s_{a^{j-1}b} \right)^* \ (j \text{ odd}),\\ 
& s_{e_j} \longmapsto \left\lbrace\begin{matrix}
s_{a^{j-1}b} \left(s_{a^{j}b} \right)^* & j \text{ even}, \\
s_{a^{j-1}b^2} \left(s_{a^{j}b} \right)^* & j \text{ odd}.
\end{matrix} \right.
\end{align*}
\end{example}

\subsection{Analytic properties of $\llbracket \mathcal{G}_E \rrbracket$}
Before generalizing the groupoid embedding theorem to a larger class of groupoids in the next subsection we take brief pause to discuss some analytic properties of the topological full groups $\llbracket \mathcal{G}_E \rrbracket$ for graphs $E$ as in Lemma~\ref{lem:emb}. First of all, $\llbracket \mathcal{G}_E \rrbracket$ is generally not amenable, as it often contains free products~\cite[Proposition~4.10]{Mat}.

Let $E_n$ for $n \geq 2$ denote the graph consisting of a single vertex and $n$ edges. And more generally, for $r \in \mathbb{N}$, let $E_{n,r}$ be the graph with $r$ vertices $w_1, w_2, \ldots, w_r$ and~${n+r-1}$ edges~$e_1, \ldots, e_n, f_1, \ldots, f_{r-1}$ such that $s(e_i) = w_1, r(e_i) = w_r$ for each $1 \leq i \leq n$ and~${s(f_i) = w_{i+1}, r(f_i) = w_i}$ for each $1 \leq i \leq r-1$. According to~\cite[Section~6]{Mat}, the topological full group $\llbracket \mathcal{G}_{E_{n,r}} \rrbracket $ is isomorphic to the \emph{Higman-Thompson group} $V_{n,r}$. In particular, $\llbracket \mathcal{G}_{E_2} \rrbracket \cong V_{2,1} = V$ (Thompson's group~$V$). As Lemma~\ref{lem:emb} in particular induces an algebraic embedding of the topological full groups, we have that $\llbracket \mathcal{G}_E \rrbracket$ embeds into $V$ for each graph $E$ as in Lemma~\ref{lem:emb}. Thus, Lemma~\ref{lem:emb} may be considered a generalization of the well-known embedding of $V_{n,r}$ into $V$. As $V$ has the \emph{Haagerup property}~\cite{Far}, we deduce that $\llbracket \mathcal{G}_E \rrbracket$ does as well.

\begin{corollary}
Let $E$ be a countable graph with no sinks, no semi-tails, and suppose $E$ satisfies Condition~(L). Then the topological full group $\llbracket \mathcal{G}_E \rrbracket$ has the Haagerup property.
\end{corollary}

\begin{remark}
For finite, strongly connected graphs, this was proved directly, using so-called \emph{zipper actions}, by Matui in~\cite{Mat}. Later, in~\cite{Mat3}, Matui proved that for any finite, strongly connected graph $E$, $\llbracket \mathcal{G}_E \rrbracket$ embeds into $\llbracket \mathcal{G}_{E_2} \rrbracket$. In fact, he proved even more, namely that $\mathcal{G}_{E_2}$ could be replaced by any groupoid with similar properties (see~\cite[Proposition~5.14]{Mat3} for the details). By our results, one may relax the conditions on $E$ considerably in Matui's embedding result.
\end{remark}

\subsection{Embedding equivalent groupoids}
We are now going to expand on the embedding theorem for graph groupoids to include all groupoids that are merely groupoid equivalent to a graph groupoid. To accomplish this we will make us of the fundamental results by Carlsen, Ruiz and Sims in~\cite{CRS}. Following their notation, let $\mathcal{R}$ denote the countably infinite discrete full equivalence relation, that is $\mathcal{R} = \mathbb{N} \times \mathbb{N}$ equipped with the discrete topology, whose product and inverse are given by $(k,m) \cdot (m,n) \coloneqq (k,n)$ and $(m,n)^{-1} \coloneqq (n,m)$. We refer to the product groupoid $\mathcal{G} \times \mathcal{R}$ as the \emph{stabilization} of the groupoid $\mathcal{G}$. For a graph $E$, let $SE$ denote the graph obtained from $E$ by adding a \emph{head} at every vertex---see the example below (see also~\cite{Tomf}). It is shown in~\cite{CRS} that $\mathcal{G}_{E} \times \mathcal{R} \cong \mathcal{G}_{S E}$ as topological groupoids for any graph $E$.

\begin{example}
The stabilized graph of $E_2$ is the following graph:
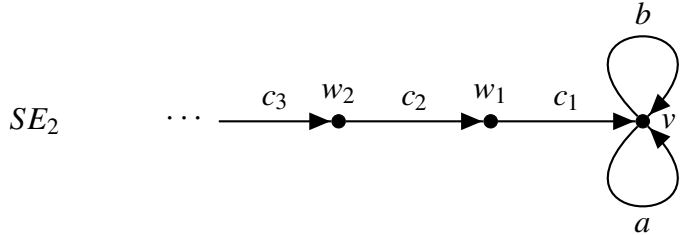
\[ \begin{tikzpicture}[vertex/.style={circle, draw = black, fill = black, inner sep=0pt,minimum size=5pt}]

\node at (-4,0) {$SE_2$};
\node[vertex] (a) at (0,0) [label=above:$w_2$]  {};
\node[vertex] (c) at (2,0) [label=above:$w_1$]  {};
\node[vertex] (b) at (4,0) [label=right:$v$] {};
\node (f) at (-2,0) {$\cdots$};

\path (f)	edge[thick, decoration={markings, mark=at position 0.99 with {\arrow{triangle 45}}}, postaction={decorate} ] node[above] {$c_3$}  (a)
	 (a) edge[thick, decoration={markings, mark=at position 0.99 with {\arrow{triangle 45}}}, postaction={decorate} ] node[above] {$c_2$}  (c)
	(b) 	edge[thick, loop, min distance = 20mm, looseness = 10, out = 225, in = 315,decoration={markings, mark=at position 0.99 with {\arrow{triangle 45}}}, postaction={decorate}] node[below] {$a$} (b)
edge[thick, loop, min distance = 20mm, looseness = 10, out = 135, in = 45,decoration={markings, mark=at position 0.99 with {\arrow{triangle 45}}}, postaction={decorate}] node[above] {$b$} (b)
	(c)	edge[thick, decoration={markings, mark=at position 0.99 with {\arrow{triangle 45}}}, postaction={decorate} ] node[above] {$c_1$}  (b);
\end{tikzpicture} \]
\end{example}

Let us first just say a few words on necessary conditions for an étale groupoid $\mathcal{H}$ to be embeddable into $\mathcal{G}_{E_2}$. First of all, it is clearly necessary that $\mathcal{H}$ is ample, Hausdorff and second countable, since $\mathcal{G}_{E_2}$ is. As we observed for the graph groupoids, it is also necessary that $\mathcal{H}^{(0)}$ has no isolated points, and hence that $\mathcal{H}^{(0)}$ is a locally compact Cantor space. Furthermore, since subgroupoids of effective groupoids are effective, it is also necessary that $\mathcal{H}$ be effective. As a final observation in this regard, any embedding $\Phi \colon \mathcal{H} \hookrightarrow \mathcal{G}_{E_2}$ induces an embedding of the isotropy bundles $\mathcal{H}' \hookrightarrow \left(\mathcal{G}_{E_2}\right)'$, meaning that $\Phi$ restricts to an embedding of the isotropy group $\mathcal{H}_y^y$ into $\left(\mathcal{G}_{E_2}\right)_{\Phi(y)}^{\Phi(y)}$ for each $y \in \mathcal{H}^{(0)}$. Now recall that for any graph groupoid $\mathcal{G}_E$ the isotropy groups are 
\[\left(\mathcal{G}_E\right)_x^x \cong \left\lbrace\begin{matrix}
\mathbb{Z} & \text{ if } x \text{ is eventually periodic}, \\
0 & \text{ otherwise}.
\end{matrix} \right. \]
Thus, a final necessary condition for embeddability is that the istropy bundle of $\mathcal{H}$ consists only of the groups $0$ and $\mathbb{Z}$. This rules out for instance (most) products of graph groupoids, since they typically have isotropy groups that are free abelian of rank up to the number of factors in the product. Note however, that taking the product with a principal groupoid does no harm in this regard. As we'll see imminently, taking the product with $\mathcal{R}$ (i.e.\ stabilizing) does not affect embeddability into $\mathcal{G}_{E_2}$.

\begin{proposition}\label{stabEmb}
Let $\mathcal{H}$ be an effective ample second countable Hausdorff groupoid with~$\mathcal{H}^{(0)}$ a locally compact Cantor space. Then $\mathcal{H}$ embeds into $\mathcal{G}_{E_2}$ if and only if the stabilized groupoid $\mathcal{H} \times \mathcal{R}$ embeds into $\mathcal{G}_{E_2}$.
\end{proposition}
\begin{proof}
The ``if statement'' is trivial as a groupoid always embeds into its stabilization. Suppose $\Phi \colon \mathcal{H} \to \mathcal{G}_{E_2}$ is an injective étale homomorphism. Then $\phi \times \id \colon \mathcal{H} \times \mathcal{R} \to \mathcal{G}_{E_2} \times \mathcal{R}$ is an injective étale homomorphism as well. By~\cite[Lemma~4.1]{CRS} we have $\mathcal{G}_{E_2} \times \mathcal{R} \cong \mathcal{G}_{S E_2}$, and $S E_2$ is a countable graph satisfying Condition~(L) with no sinks nor semi-tails. So by Theorem~\ref{thm:emb}, $\mathcal{G}_{S E_2}$ embeds into $\mathcal{G}_{E_2}$. Thus $\mathcal{H} \times \mathcal{R}$ embeds into $\mathcal{G}_{E_2}$.
\end{proof}

The next lemma shows that any étale embedding of a groupoid $\mathcal{H}$, with compact unit space, into $\mathcal{G}_{E_2}$ can be ``twisted'' into an embedding that hits the whole unit space of $\mathcal{G}_{E_2}$.

\begin{lemma}\label{unitalEmb}
Let $\mathcal{H}$ be an effective ample second countable Hausdorff groupoid with $\mathcal{H}^{(0)}$ a Cantor space. If $\mathcal{H}$ embeds into $\mathcal{G}_{E_2}$, then there exists an embedding $\Phi \colon \mathcal{H} \hookrightarrow \mathcal{G}_{E_2}$ such that $\Phi\left(\mathcal{H}^{(0)}\right) = E_2^\infty$.
\end{lemma}
\begin{proof}
Let $\Psi \colon \mathcal{H} \to \mathcal{G}_{E_2}$ be an injective étale homomorphism and let $Y = \Psi\left(\mathcal{H}^{(0)}\right)$. Then~$Y$ is a compact open (hence clopen) subset of $E_2^\infty$. We claim that there exists a compact open bisection $U \subseteq \mathcal{G}_{E_2}$ such that $s(U) = Y$ and $r(U) = E_2^\infty$. The claim follows from~\cite[Theorem~6.4]{Mat} and~\cite[Example~3.3~(3)]{Mat2} by identifying $\mathcal{G}_{E_2}$ with the \emph{SFT-groupoid} of the $1 \times 1$ matrix $A = [2]$ (see~\cite[Example~2.5]{Mat2}). Now define $\Phi(h) = U \cdot \Psi(h) \cdot U^{-1}$ for $h \in \mathcal{H}$. Then $\Phi$ is an injective étale homomorphism and \[\Phi\left(\mathcal{H}^{(0)}\right) = U  Y  U^{-1} = U U^{-1} = r(U) = E_2^\infty.\]
\end{proof}

We now state the most general version of our embedding theorem.

\begin{theorem}\label{eqEmbed}
Let $\mathcal{H}$ be an effective ample second countable Hausdorff groupoid whose unit space~$\mathcal{H}^{(0)}$ is a locally compact Cantor space. If $\mathcal{H}$ is groupoid equivalent to $\mathcal{G}_E$, for some countable graph $E$ satisfying Condition~(L) and having no sinks nor semi-tails, then $\mathcal{H}$ embeds into~$\mathcal{G}_{E_2}$. Moreover, if $\mathcal{H}^{(0)}$ is compact, then the embedding maps $\mathcal{H}^{(0)}$ onto $E_2^\infty$.
\end{theorem}
\begin{proof}
Suppose $\mathcal{H}$ is groupoid equivalent to $\mathcal{G}_E$ as above. Then by~\cite[Theorem~3.2]{CRS} we have $\mathcal{H} \times \mathcal{R} \cong \mathcal{G}_E \times \mathcal{R}$. By Theorem~\ref{thm:emb} and Proposition~\ref{stabEmb}, $\mathcal{G}_E \times \mathcal{R}$ embeds into~$\mathcal{G}_{E_2}$, hence so does $\mathcal{H} \times \mathcal{R}$ and $\mathcal{H}$. The second statement follows from Lemma~\ref{unitalEmb}.
\end{proof}

\begin{remark}
We note that for any groupoid $\mathcal{H}$ as in the above theorem, its topological full group $\llbracket \mathcal{H} \rrbracket$ also has the Haagerup property.
\end{remark}

\subsection{Embedding AF-groupoids}
A well-studied class of groupoids satisfying the hypothesis of Theorem~\ref{eqEmbed}, yet conceptually different from graph groupoids, are the \emph{AF-groupoids}. See~\cite{GPS2} (wherein they are dubbed \emph{AF-equivalence relations}). Let $\mathcal{G}$ be an ample Hausdorff second countable groupoid with $\mathcal{G}^{(0)}$ a locally compact Cantor space. Then $\mathcal{G}$ is called an \emph{AF-groupoid} if there exists an increasing sequence $\mathcal{K}_1 \subseteq \mathcal{K}_2 \subseteq \ldots \subseteq \mathcal{G}$ of clopen subgroupoids such that 
\begin{itemize}
\item $\mathcal{K}_n$ is principal for each $n \in \mathbb{N}$.
\item $\mathcal{K}_n^{(0)} = \mathcal{G}^{(0)}$ for each $n \in \mathbb{N}$.
\item $\mathcal{K}_n \setminus \mathcal{G}^{(0)}$ is compact for each $n \in \mathbb{N}$.
\item $\bigcup_{n=1}^\infty \mathcal{K}_n = \mathcal{G}$.
\end{itemize}
This entails that $\mathcal{G}$ is principal.

\begin{remark}
The terminology AF-groupoid is due to Renault~\cite{Ren}, and is also used by Matui in~\cite{Mat1} and~\cite{Mat2}. Note however, that Matui only considered the case of a compact unit space therein.
\end{remark}

In the following example we explain how \emph{Bratteli diagrams} give rise to AF-groupoids.

\begin{example}[c.f.~{\cite[Example~2.7(ii)]{GPS2}}]\label{ex:BD}
A \emph{Bratteli diagram} $B$ is a directed graph whose vertex set $V$ and edge set $E$ can be written as countable disjoint unions of non-empty finite sets
\begin{equation}\label{eq:BD}
V = V_0 \sqcup V_1 \sqcup V_2 \sqcup \ldots \quad \text{and} \quad  E =  E_1 \sqcup E_2 \sqcup E_3 \sqcup \ldots
\end{equation}
such that the source and range maps satisfy $s(E_n) = V_{n-1}$ and $r(E_n) \subseteq V_n$\footnote{This notation is inconsistent with what we have been using for directed graphs so far. But since Bratteli diagrams are very special kinds of graphs we have chosen to use the well-established notation from the literature. In this way we can, albeit somewhat artificially, distinguish a Bratteli diagram from its underlying graph.}. In particular, there are no sinks in $B$. Let $S_B \subseteq V$ denote the set of sources in $B$. Then $V_0 \subseteq S_B$. We call~$B$ a \emph{standard} Bratteli diagram if there is only one source in $B$, i.e.\ $S_B = \{v_0\} = V_0$. We say that $B$ is \emph{simple} if for every vertex $v \in V_n$, there is an $m > n$ such that there is a path from $v$ to every vertex in $V_m$. The partitions of the vertices and edges (into \emph{levels} as in Equation \eqref{eq:BD}) is considered part of the data of the Bratteli diagram $B$. We let $E_B$ denote the underlying graph where we ``forget'' about the partitions.

For a source $v \in S_B \cap V_n$ on level $n$ we let $X_v$ denote the set of infinite paths starting in~$v$, that is 
\[X_v \coloneqq \{ e_{n+1} e_{n+2} e_{n+3} \ldots  \mid s(e_{n+1}) = v, e_{n+k} \in E_{n+k}, s(e_{n+k}) = r(e_{n+k-1}), k > 1  \}\] 
The \emph{path space} of $B$ is 
\[X_B \coloneqq \bigsqcup_{v \in S_B} X_v \]
whose topology is given by the basis of \emph{cylinder sets} \[C(\mu) \coloneqq \{ e_{n+1} e_{n+2} \ldots \in X_{s(\mu)} \mid e_{n+1} \ldots e_{n+ \vert \mu \vert} = \mu \}\] where $\mu$ is a finite path such that $s(\mu) = v$ for some source $v \in S_B \cap V_n$. The path space $X_B$ is Boolean, and it is compact if and only if $S_B$ is finite. Further, $X_B$ is perfect if and only if~$E_B$ has no semi-tails. Two infinite paths in $X_B$ are \emph{tail-equivalent} if they agree from some level on. With this equivalence relation as the starting point, let for each $N \in \mathbb{N}$ \[\mathcal{P}_N \coloneqq \{ (x,y) \in X_B \times X_B \mid s(x) \in V_m, s(y) \in V_n, m,n \leq N, x_k = y_k \text{ for all } k > N  \}.\]
That is, $\mathcal{P}_N$ consists of all pairs of infinite paths which start before the $N$'th level and agrees from the $N$'th level and onwards. Equipping $\mathcal{P}_N$ with the relative topology from $X_B \times X_B$ makes $\mathcal{P}_N$ a compact principal ample Hausdorff groupoid whose unit space is identified with $\bigsqcup_{n=1}^N \bigsqcup_{v \in S_B \cap V_n} Z(v)$.

We define the \emph{groupoid of the Bratteli diagram $B$} as the increasing union \[\mathcal{G}_B \coloneqq \bigcup_{N=1}^\infty \mathcal{P}_N\] equipped with the inductive limit topology. For two finite paths $\mu, \lambda$ with $s(\mu), s(\lambda) \in S_B$ and $r(\mu) = r(\lambda)$ we define \[C(\mu, \lambda) \coloneqq \left\{ (x,y) \in C(\mu) \times C(\lambda) \mid x_{\left[\vert \mu \vert +1, \infty\right)} = y_{\left[\vert \lambda \vert +1, \infty\right)} \right\}. \]
A straightforward computation shows that the family of $C(\mu, \lambda)$'s form a compact open basis for the inductive limit topology on $\mathcal{G}_B$. We identify $\mathcal{G}_B^{(0)}$ with $X_B$. By setting $\mathcal{K}_n = \mathcal{P}_n \cup \mathcal{G}_B^{(0)}$ one sees that $\mathcal{G}_B$ is an AF-groupoid. The groupoid $\mathcal{G}_B$ is minimal if and only if $B$ is a simple Bratteli diagram.
\end{example}

\begin{remark}
Although the AF-groupoid $\mathcal{G}_B$ is defined in terms of a very special graph, namely the Bratteli diagram $B$, it is generally not isomorphic to a graph groupoid. To see this, recall that $\mathcal{G}_B$ is always principal, while a graph groupoid $\mathcal{G}_E$ is principal if and only if the graph~$E$ has no cycles. If $X_B$ is compact, perfect and infinite (this is essentially stipulating that the Bratteli diagram is standard and ``non-degenerate''), then $\mathcal{G}_B$ cannot be isomorphic to any graph groupoid. For any such  $\mathcal{G}_E$ would have a compact unit space, i.e.~$E$ has finitely many vertices, and $E$ would have no cycles and no sinks. There are clearly no such graphs.
\end{remark}

Giordano, Putnam and Skau showed that, just as with AF-algebras~\cite{Bra}, every AF-groupoid can be realized by a Bratteli diagram as in Example~\ref{ex:BD}.

\begin{theorem}[{\cite[Theorem~3.9]{GPS2}}]
Let $\mathcal{H}$ be an AF-groupoid. Then there exists a Bratteli diagram $B$ such that $\mathcal{H} \cong \mathcal{G}_B$. If $\mathcal{H}^{(0)}$ is compact, then $B$ can be chosen to be standard.
\end{theorem}

\begin{remark}\label{AFfullgroups}
As another example of a concrete description of the topological full group of an ample groupoid, we remark that Matui described the topological full group of an AF-groupoid with compact unit space in terms of a definining Bratteli diagram in~\cite[Proposition~3.3]{Mat5}. The topological full group $\llbracket \mathcal{G}_B \rrbracket$, where $B$ is a Bratteli diagram, is the direct limit of the finite groups $\Gamma_N$ for $N \in \mathbb{N}$, where $\Gamma_N \leq \homeo(X_B)$ consists of all permutations of the finite set of paths from level $V_0$ to $V_N$ such that the permutation preserves the range of these paths (and the action on $X_B$ is by permuting the intial segment of an infinite path). We should also mention that these groups were originally studied by Krieger in~\cite{Kri}, without emphasis on the underlying groupoids.
\end{remark}

By the preceding remark it is clear that the topological full group of any AF-groupoid is a locally finite group. And actually, this characterizes the AF-groupoids. This is somewhat of a folklore result, but a proof is published by Matui in the compact case, and it is not hard to see that his proof extends to locally compact unit spaces as well. 

\begin{proposition}[c.f.\ {\cite[Proposition~3.2]{Mat5}}]\label{AFlocfin}
Let $\mathcal{G}$ be an ample principal Hausdorff second countable groupoid with $\mathcal{G}^{(0)}$ a locally compact Cantor space. Then the topological full group $\llbracket \mathcal{G} \rrbracket$ is locally finite if and only if $\mathcal{G}$ is an AF-groupoid.
\end{proposition}

\begin{remark}\label{rem:LDA}
The commutator subgroups $\DD(\mathcal{G}) \leq \llbracket \mathcal{G} \rrbracket$ for AF-groupoids $\mathcal{G}$ are quite interesting in their own right. In fact, these exhaust\footnote{With the single exception of the infinite finitary alternating group.} the class of so-called \emph{strongly diagonal limits of products of alternating groups} (also called \emph{LDA-groups}, see~\cite{LN} where these are classified using the dimension groups of their Bratteli diagrams). These form a subclass of the locally finite simple groups. By Corollary~\ref{AFembed} below, all the LDA-groups embed into Thompson's group $V$.
\end{remark}

We now demonstrate that every AF-groupoid is groupoid equivalent to a graph groupoid. This is essentially just a reformulation of the main theorem from~\cite{Drin}, wherein it is shown that any AF-algebra can be recovered as a certain \emph{pointed} graph $C^*$-algebra of a defining Bratteli diagram. In contrast, in Proposition~\ref{AFgrpdeq} below we emphasize the groupoids, rather than their $C^*$-algebras. Also, since we use ``unlabeled'' Bratteli diagrams here, as opposed to \emph{labeled Bratteli diagrams} (c.f.~\cite[Section~2]{Drin}), the computations are easier.

\begin{proposition}\label{AFgrpdeq}
Let $B$ be a Bratteli diagram. Then the AF-groupoid $\mathcal{G}_B$ is isomorphic to the restriction of the graph groupoid $\mathcal{G}_{E_B}$ to the open subset $ \bigsqcup_{v \in S_B} Z(v) \subseteq E_B^\infty$. In particular, every AF-groupoid is groupoid equivalent to a graph groupoid. 
\end{proposition}
\begin{proof}
Let $A = \bigsqcup_{v \in S_B} Z(v)$. Then 
\[\left(\mathcal{G}_{E_B}\right)_{| A} = \{ (x,k,y) \mid s(x), s(y) \in S_B, \sigma_{E_B}(x)^m = \sigma_{E_B}(y)^n, k = m - n  \}.  \]
Due to the special structure of the graph $E_B$, the lag $k$ in $(x,k,y) \in \left(\mathcal{G}_{E_B} \right)_{| A}$ is uniquely determined by $x$ and $y$. In fact, $k$ is determined by the levels on which $x$ and $y$ start in the Bratteli diagram. Indeed, let $m,n \in \mathbb{N}$ be such that $s(x) \in V_m$ and $s(y) \in V_n$, then~${k = n-m}$. This means that the map $\Phi \colon \left(\mathcal{G}_{E_B} \right)_{| A} \to \mathcal{G}_B$ defined by~${\Phi((x,k,y)) = (x,y)}$ is a bijection. It is easy to see that $\Phi$ is also a groupoid homomorphism. Finally, to see that $\Phi$ is a homeomorphism simply note that the family of $Z(\mu, \lambda)$'s where $\mu, \lambda$ are finite paths with~$s(\mu), s(\lambda) \in S_B$ and $r(\mu) = r(\lambda)$ form a basis for $\left(\mathcal{G}_{E_B}\right)_{| A}$, and that~${\Phi(Z(\mu, \lambda)) = C(\mu, \lambda)}$. Thus $\left(\mathcal{G}_{E_B}\right)_{| A} \cong \mathcal{G}_B$ as étale groupoids.

We claim that $A$ is a $\mathcal{G}_{E_B}$-full subset of $E_B^\infty$, and then the second statement follows from~\cite[Theorem~3.2]{CRS}. To see this, let $z \in E_B^\infty$ be an infinite path starting anywhere in the Bratteli diagram and simply note that by following $s(z)$ upwards in the Bratteli diagram, one eventually reaches a source $v \in S_B$ such that $v$ connects to $s(z)$. Letting $\mu$ be any path from $v$ to $s(z)$ we have that $z$ belongs to the $\mathcal{G}_{E_B}$-orbit of $\mu z \in A$.
\end{proof}

As a special case of Theorem~\ref{eqEmbed} we obtain the following.

\begin{corollary}\label{AFembed}
Let $\mathcal{G}$ be an AF-groupoid with $\mathcal{G}^{(0)}$ perfect. Then there exists an embedding of étale groupoids $\mathcal{G} \hookrightarrow \mathcal{G}_{E_2}$. If $\mathcal{G}^{(0)}$ is compact, then $\mathcal{G}^{(0)}$ maps onto $E_2^\infty$.
\end{corollary}

From this we obtain an analogue of Corollary~\ref{cor:graphAlgEmb} for AF-algebras and their diagonals. Let $A$ be an AF-algebra. By an \emph{AF Cartan subalgebra} $D \subseteq A$ we mean a Cartan subalgebra arising from the diagonalization method of Str\u atil\u a and Voiculescu~\cite{SV}. See~\cite[Section~4]{Drin} for a description of these diagonals for non-unital AF-algebras. Note that they are also $C^*$-diagonals in the sense of Kumjian~\cite{Kum}. According to~\cite[Subsection~6.2]{Ren2} these are precisely the Cartain pairs arising as $\left(C^*_r\left(\mathcal{G}_B\right), C_0(X_B)\right)$ for a Bratteli diagram $B$.

\begin{corollary}\label{cor:AFemb}
Let $A$ be an infinite-dimensional AF-algebra and let $D \subseteq A$ be any AF Cartan subalgebra in $A$ whose spectrum is perfect. Then there is an injective \mbox{$*$-homomorphism} ${\psi \colon A \hookrightarrow \mathcal{O}_2}$ such that $\psi(D) \subseteq \mathcal{D}_2$. If $A$ is unital, then so is $\psi$, and $\psi(D) = \mathcal{D}_2$. 
\end{corollary}

\begin{remark}\label{LCCMfullgroups}
As a final remark, we note that certain transformation groupoids (by virtue of actually being AF-groupoids) also embed into $\mathcal{G}_{E_2}$. Let $X$ be a \emph{non-compact} locally compact Cantor space and let $T$ be a minimal homeomorphism on $X$. It follows from~\cite[Theorem~4.3]{GPS2} that the transformation groupoid $\mathbb{Z} \ltimes_T X$ is an AF-groupoid, and consequently~$\mathbb{Z} \ltimes_T X$ embeds into $\mathcal{G}_{E_2}$.

An indirect way of seeing that $\mathbb{Z} \ltimes_T X$ is an AF-groupoid is via Proposition~\ref{AFlocfin}. By realizing the dynamical system $(X,T)$ as a so-called \emph{Bratteli-Vershik system} on a (standard) \emph{almost simple orderered Bratteli diagram} $B = (V,E, \geq)$ c.f.~\cite{Dan}, one easily observes (as Matui did in~\cite{Mat4}) that $\llbracket \mathbb{Z} \ltimes_T X \rrbracket$ is locally finite. This is because each element of~$\llbracket \mathbb{Z} \ltimes_T X \rrbracket$ only depends on the inital edges down to level $N$ for some fixed $N$ (determined by the group element), for each infinite path in $X_B$. This actually allows one to describe the topological full group $\llbracket \mathbb{Z} \ltimes_T X \rrbracket$ explicitly in terms of a conjugate Bratteli-Vershik system.

A third way of demonstrating that $\mathbb{Z} \ltimes_T X$ is an AF-groupoid is that one can go from a conjugate Bratteli-Vershik system on an ordered Bratteli diagram $B = (V,E, \geq)$ to an ``unordered'' Bratteli diagram $B'$ such that $\mathbb{Z} \ltimes_T X \cong \mathcal{G}_{B'}$ as étale groupoids. Indeed, let~$e_1 e_2 e_3 \ldots \in X_B$ denote the unique maximal and minimal path in $X_B$ (c.f.~\cite{Dan}). By ``forgetting'' the ordering and removing each of the edges $e_n$ for all $n \in \mathbb{N}$, and thereby introducing a source at each of the vertices $s(e_n)$, one obtains the modified Bratteli diagram~$B'$, and it is not hard to see that the AF-groupoid $\mathcal{G}_{B'}$ is isomorphic to $\mathbb{Z} \ltimes_T X$.
\end{remark}

\newcommand{\etalchar}[1]{$^{#1}$} 

\end{document}